\documentclass[11pt,letterpaper]{amsart}
\usepackage[foot]{amsaddr}

\usepackage{mathtools}

\usepackage{amsmath}
\usepackage[T1]{fontenc}
\usepackage[latin9]{inputenc}
\usepackage{geometry}
\usepackage{wrapfig}
\usepackage{multicol}
\usepackage{graphicx}
\usepackage{soul}
\usepackage{xcolor}
\usepackage{amssymb}
\usepackage{placeins}
\usepackage{bbm}
\usepackage{cite}
\usepackage{caption}
\usepackage{enumerate}
\usepackage{afterpage}
\usepackage{enumitem}
\usepackage{bmpsize}
\usepackage{hyperref}
\usepackage{tabu}
\usepackage{enumitem}
\numberwithin{equation}{section}
\usepackage{stmaryrd}
\usepackage{tikz}
\usepackage{ytableau}
\usetikzlibrary{decorations.pathreplacing}
\usetikzlibrary{matrix,graphs,arrows,positioning,calc,decorations.markings,shapes.symbols}
\usepackage{ifthen}
\usepackage{amsthm}
\usetikzlibrary{math}
\usetikzlibrary{matrix,graphs,arrows,positioning,calc,decorations.markings,shapes.symbols}

\makeatletter
\renewcommand{\email}[2][]{%
  \ifx\emails\@empty\relax\else{\g@addto@macro\emails{,\space}}\fi%
  \@ifnotempty{#1}{\g@addto@macro\emails{\textrm{(#1)}\space}}%
  \g@addto@macro\emails{#2}%
}
\makeatother

\newtheorem{theorem}{Theorem}[section]
{\theoremstyle{remark}
\newtheorem{remark}[theorem]{Remark}}
\newtheorem{lemma}[theorem]{Lemma}
\newtheorem{proposition}[theorem]{Proposition}

{ \theoremstyle{definition}
\newtheorem{definition}[theorem]{Definition}}
\newtheorem{assumption}[theorem]{Assumption}

\newcommand{\im}{\mathsf{i}}

\newcommand{\cm}{\mathsf{m}}

\newcommand{\const}{\mathsf{const}}
\newcommand{\cd}{\mathsf{D}}
\newcommand{\cR}{\mathsf{R}}
\newcommand{\cq}{\mathsf{q}}

\title{Asymptotics of Jack measures with homogeneous specializations}
\date{\today}
\author{Evgeni Dimitrov} 
\author{Xiaohan Gao}
\author{Andy Gu}
\author{Ryan Niedernhofer}

\begin{document}

\begin{abstract}
We consider Jack measures on partitions with homogeneous defining specializations. For each of the six distinct classes of measures obtained this way we prove a global law of large numbers with an explicit limiting particle density. We also demonstrate for one of these classes how to obtain a global central limit theorem, global and edge large deviation principles, and edge universality using the results of the paper. Our argument is based on explicitly evaluating Jack symmetric functions at homogeneous specializations, relating the Jack measures to the discrete $\beta$-ensembles from \cite{BGG17} and using the discrete loop equations to substantially reduce computations.    
\end{abstract}

\maketitle

\tableofcontents

%
%
\section{Introduction }\label{Section1} The {\em Jack measures} are probability measures on the set of partitions $\mathbb{Y}$, or equivalently Young diagrams, which take the form
\begin{equation}\label{Eq.JackInformal}
\mathcal{J}_{\rho_1, \rho_2}(\lambda)   \propto J_{\lambda}(\rho_1;\theta) \tilde{J}_{\lambda}(\rho_2; \theta).
\end{equation}
Here, $J_{\lambda}, \tilde{J}_{\lambda}$ are the Jack symmetric functions and their duals, see \cite[Section 6.10]{Mac98} (where $\alpha = \theta^{-1}$). The parameter $\theta > 0$ is fixed, and $\theta = \beta/2$ with $\beta$ sometimes referred to as the {\em inverse temperature}. Lastly, $\rho_1, \rho_2$ are {\em Jack-positive specializations}, i.e., algebra homomorphisms from the ring of symmetric functions to the complex field that map the Jack symmetric functions to non-negative reals. The measures $\mathcal{J}_{\rho_1, \rho_2}$ were initially proposed by Borodin and Olshanski \cite{BO05}. They form a one-parameter generalization of the {\em Schur measures} from \cite{Oko01} (corresponding to $\theta = 1$) and a distinguished degeneration of the {\em Macdonald measures} from \cite{BC14}.  

Over the last few decades there has been a significant interest in studying the measures $\mathcal{J}_{\rho_1, \rho_2}$ for different choices of specializations and parameters $\theta$. For example, when $\theta = 1$ the measures $\mathcal{J}_{\rho_1, \rho_2}$ can be related to the uniform measure on permutations with $\lambda_1$ corresponding to the length of the longest increasing subsequence; see \cite{LS77, VK77, BOO00, BDJ99}. For general $\theta > 0$ the measures $\mathcal{J}_{\rho_1, \rho_2}$ arise as distinguished mixtures of the measures on partitions of fixed size, as studied by Kerov, Okounkov and Olshanski \cite{KOO98} in the context of harmonic functions on the Young graph with Jack multiplicities. They also arise in the setting of coherent systems of measures on the Gelfand-Tsetlin graph with Jack weights. When $\theta = 1$ the latter are described in \cite{BO12} and appear implicitly for general $\theta > 0$ in \cite{OO98}. A key motivation for studying the measures $\mathcal{J}_{\rho_1, \rho_2}$ is their interpretation as discrete analogs of the $\beta$ log-gases, which arise as eigenvalue distributions in random matrix theory; see \cite{BDS16,CDM23} and the references therein. In particular -- and as discussed in detail in Section \ref{Section2} -- for certain choices of $\rho_1$ and $\rho_2$, the measures $\mathcal{J}_{\rho_1, \rho_2}$ can be identified with the {\em discrete $\beta$-ensembles} introduced in \cite{BGG17}.

In \cite{KOO98} Kerov, Okounkov and Olshanski established a complete characterization of the Jack-positive specializations, see Proposition \ref{S12P2} below. Specifically, they showed that each Jack-positive specialization is a mixture of three types of specializations: a {\em pure-$\alpha$ specialization}, encoded by a sequence $\{\alpha_i\}_{i \geq 1}$, a {\em pure-$\beta$ specialization}, encoded by a sequence $\{\beta_i\}_{i \geq 1}$, and a {\em Plancherel specialization}, encoded by a parameter $\gamma \geq 0$. Depending on the choice and scaling of these specializations, the measures $\mathcal{J}_{\rho_1, \rho_2}$ exhibit different global behaviors, some of which have been previously studied. In \cite{H21}, Huang proved a global law of large numbers (LLN) and a central limit theorem (CLT) when $\rho_1 = \rho_1(N)$ is a {\em homogeneous} pure-$\alpha$ specialziation (i.e. when $\alpha_1 = \cdots = \alpha_N = a > 0$ and $N \rightarrow \infty$), under some technical assumptions on $\rho_2$. The proof relies on the use of {\em Nazarov-Sklyanin operators} and {\em Jack generating functions}. Using a combinatorial approach of the Nazarov-Sklyanin operators Moll \cite{Moll23} and Cuenca, Do\l\c{e}ga and Moll \cite{CDM23} established global LLNs and CLTs in two additional cases: when $\rho_1 = \rho_2$, and when $\rho_2$ is a Plancherel specialization, respectively.\\

\vspace{-1mm}
The goal of the present paper is to study $\mathcal{J}_{\rho_1, \rho_2}$ by relating the latter to the discrete $\beta$-ensembles from \cite{BGG17}, and by employing the {\em discrete loop equations}, also known as {\em Nekrasov's equations}. We carry out this analysis in the case where each of the specializations $\rho_1$ and $\rho_2$ is either a homogeneous pure-$\alpha$ specialization, a homogeneous pure-$\beta$ specialization, or a Plancherel specialization. While this setup yields nine possible combinations, the symmetry $\mathcal{J}_{\rho_1, \rho_2} = \mathcal{J}_{\rho_2, \rho_1}$ reduces the number of distinct cases to six. We believe these six configurations encompass all instances of the Jack measures that are directly related to discrete $\beta$-ensembles. For each case, under general conditions on the parameters, we establish a global LLN and provide an explicit formula for the limiting particle density; see Theorem \ref{thmMain}. We mention that some of these cases--sometimes under additional parameter restrictions--can also be analyzed using the methods developed in \cite{CDM23, H21}, although the LLN in these papers are in terms of limiting moments as opposed to limiting particle densities.

The results in \cite{CDM23} extend well beyond the Jack measures defined in (\ref{Eq.JackInformal}). For instance, they also consider other classes of measures arising from Jack characters, previously studied in \cite{DS19}. Moreover, both \cite{CDM23} and \cite{Moll23} investigate asymptotic behaviors in regimes where the inverse temperature parameter $\theta$ is scaled alongside the specializations. A key motivation for our analysis of Jack measures from the discrete $\beta$-ensemble perspective is the extensive body of results now available for $\beta$-ensembles, which include a global CLT \cite{BGG17}, a global large deviation principle (LDP) \cite{DZ23}, edge LDPs \cite{DD22}, and edge universality \cite{GH17}, all under general conditions. To our knowledge, these types of results are not currently accessible via the techniques developed in \cite{Moll23, CDM23, H21} with the exception of the global CLT. In the course of establishing our LLN results, we verify many of the assumptions required by \cite{DD22, DK22,DZ23,GH17}, and with some additional but modest effort, one can extend our analysis to obtain the full range of asymptotic results mentioned above. We illustrate this in Theorem~\ref{thmMain2} for one of the six cases considered in this paper.

Overall, we believe this paper serves as a useful complement to the previous works on Jack measures \cite{CDM23, Moll23, H21}. While our scope is somewhat narrower, the class of measures we study appears to be the most general for which a direct connection to discrete $\beta$-ensembles--and consequently to discrete loop equations--exists. We note that the proof of the LLN presented here does not rely on the loop equations. In fact, a limit shape phenomenon can be derived using the LLN results from \cite{BGG17} and the LDPs established in \cite{DD22}. These results imply that the limiting particle density is characterized implicitly as the minimizer of a certain energy functional. One of the key contributions of this paper is to demonstrate how the loop equations can be used to compute these minimizers explicitly, thereby avoiding the need to solve complex integral equations, as was done in \cite{DD22}.

%
%
\section{Definitions and main results }\label{Section11} In Section \ref{Section1.2} we give a formal introduction of the main objects we study in the paper, called the {\em Jack measures}. In Section \ref{Section1.3} we present our results regarding the global laws of large numbers, i.e. limit shape phenomenon, and in Section \ref{Section1.4} we discuss our results regarding the global central limit theorem, global large deviation principle and edge asymptotics. In Section \ref{Section1.5} we discuss the main arguments in the paper. 

%
%
\subsection{Jack measures}\label{Section1.2} Our exposition here follows \cite[Section 6.1]{DD22} and \cite[Section 6.1]{DK22}, and we start by introducing some useful notation for symmetric functions, using \cite{Mac98} as a main reference.

A {\em partition} is a sequence $\lambda = (\lambda_1, \lambda_2, \dots)$ of non-negative integers such that $\lambda_1 \geq \lambda_2 \geq \cdots$ and all but finitely many terms are zero. We denote the set of all partitions by $\mathbb{Y}$. The {\em length} of a partition, denoted $\ell(\lambda)$, is the number of non-zero $\lambda_i$, and the {\em weight} of a partition is given by $|\lambda| = \lambda_1 + \lambda_2 + \cdots$. An alternative representation is given by $\lambda = 1^{m_1} 2^{m_2} \cdots$, where $m_j(\lambda) = |\{i \in \mathbb{N}: \lambda_i = j\}|$ is called the {\em multiplicity} of $j$ in the partition $\lambda$. There is a natural ordering on $\mathbb{Y}$, called the {\em reverse lexicographic order}, given by
$$\lambda > \mu \iff \exists k \in \mathbb{N} \mbox{ such that } \lambda_i = \mu_i \mbox{ for } i < k \mbox{ and } \lambda_k > \mu_k.$$
A {\em Young diagram} is a graphical representation of a partition $\lambda$, with $\lambda_1$ left justified boxes in the top row, $\lambda_2$ in the second row and so on. In general, we do not distinguish between a partition $\lambda$ and the Young diagram representing it. The {\em conjugate} of a partition $\lambda$ is the partition $\lambda'$, whose Young diagram is the transpose of the diagram $\lambda$. In particular, we have $\lambda_i' = |\{j \in \mathbb{N}: \lambda_j \geq i\}|$. For a box $\square = (i,j)$ of a Young diagram $\lambda$ we let $a(\square), \ell(\square)$ denote the {\em arm} and {\em leg lengths}, i.e.
$$a(\square) = \lambda_i - j, \hspace{5mm} \ell(\square) = \lambda_j' - i.$$
Further, we let $a'(\square)$ and $\ell'(\square)$ denote the {\em co-arm} and {\em co-leg lengths}:
$$a'(\square) = j - 1, \hspace{5mm} \ell'(\square) = i - 1.$$  

Let $\Lambda$ be the $\mathbb{Z}_{\geq 0}$ graded algebra over $\mathbb{Q}$ of symmetric functions in countably many variables $X = (x_1, x_2, \dots)$. A symmetric function can be viewed as a formal symmetric power series of bounded degree. One way to interpret $\Lambda$ is as an algebra of polynomials in the {\em Newton power sums}
$$p_k(X) = \sum_{i = 1}^{\infty} x_i^k \mbox{ for } k \geq 1.$$
For any partition $\lambda$ we define 
$$p_{\lambda}(X) = \prod_{i = 1}^{\ell(\lambda)} p_{\lambda_i}(X),$$
and note that $\{p_{\lambda}: \lambda \in \mathbb{Y}\}$ forms a linear basis for $\Lambda$.

In what follows we fix a parameter $\theta$, which will eventually be a fixed positive real, but for now can be treated as a formal parameter. We consider the following scalar product $\langle \cdot, \cdot \rangle$ on $\Lambda \otimes \mathbb{Q}(\theta)$
\begin{equation}\label{S12E1}
\langle p_{\lambda}, p_{\mu}\rangle = {\bf 1}\{\lambda = \mu\} \cdot \theta^{-\ell(\lambda)} \prod_{i = 1}^{\lambda_1} i^{m_i(\lambda)} m_i(\lambda)!.
\end{equation}
\begin{proposition}\label{S12P1} \cite[Section 6.10]{Mac98}. There are unique symmetric functions $J_{\lambda}(X;\theta) \in \Lambda \otimes \mathbb{Q}(\theta)$ for all $\lambda \in \mathbb{Y}$ such that
\begin{itemize}
    \item $\langle J_{\lambda}, J_{\mu} \rangle = 0$ unless $\lambda = \mu$,
    \item the leading (with respect to reverse lexicographic order) monomial in $J_{\lambda}$ is $\prod_{i = 1}^{\ell(\lambda)} x_i^{\lambda_i}$.
\end{itemize}
\end{proposition}
The $J_{\lambda}(X;\theta)$ in Proposition \ref{S12P1} are called {\em Jack symmetric functions} and they form a homogeneous linear basis for $\Lambda$ that is different from the $p_\lambda$ above. We mention that in \cite[Section 6.10]{Mac98} these functions are denoted by $P_{\lambda}^{(\alpha)}$. Our choice of parameterizing the Jack symmetric functions by $\theta$ is made after \cite{KOO98} and one has the relationship $J_{\lambda}(X;\theta) = P_{\lambda}^{(\theta^{-1})}(X)$. Given $\lambda \in \mathbb{Y}$, we define the {\em dual Jack symmetric functions $\tilde{J}_{\lambda}$} through
\begin{equation}\label{S12E2}
\tilde{J}_{\lambda}(X;\theta) = J_\lambda(X;\theta) \cdot \prod_{\square \in \lambda} \frac{a(\square) + \theta \cdot \ell(\square) + \theta}{a(\square) + \theta \cdot \ell(\square) + 1}.
\end{equation}

In the remainder of this section we assume that $\theta > 0$ is fixed.
\begin{definition} A {\em specialization} $\rho$ is an algebra homomorphism from $\Lambda$ to the set of complex numbers. A specialization is called {\em Jack-positive} if its values on all Jack symmetric functions is real and non-negative.
\end{definition}
The following statement provides a classification of all Jack-positive specializations.
\begin{proposition}\label{S12P2} \cite{KOO98} For any fixed $\theta > 0$, the Jack-positive specializations can be parameterized by triplets $(\alpha, \beta, \gamma)$, where $\alpha, \beta$ are sequences of real numbers with
$$\alpha_1 \geq \alpha_2 \geq \cdots \geq 0, \hspace{4mm} \beta_1 \geq \beta_2 \geq \cdots \geq 0, \hspace{4mm} \sum_{i = 1}^{\infty} (\alpha_i + \beta_i) < \infty,$$
and $\gamma$ is a non-negative real number. The specialization corresponding to a triplet $(\alpha, \beta, \gamma)$ is given by its value on $p_k$
\begin{equation}\label{S12E3}
\begin{split}
&p_1 \rightarrow p_1(\alpha,\beta, \gamma) = \gamma + \sum_{i = 1}^{\infty} (\alpha_i + \beta_i),\\
&p_k \rightarrow p_k(\alpha, \beta, \gamma) = \sum_{i = 1}^{\infty} \alpha_i^k + (-\theta)^{k-1} \sum_{i = 1}^{\infty} \beta_i^k, \hspace{4mm} k \geq 2.
\end{split}
\end{equation}
\end{proposition}
\begin{remark}\label{S12R1} As the $p_k$'s are algebraically independent algebraic generators of $\Lambda$, we see that (\ref{S12E3}) uniquely specifies a given specialization. 
\end{remark}

With the above notation in place we can now define probability measures on $\mathbb{Y}$ as follows.
\begin{definition}\label{DefJM} Fix $\theta > 0$ and let $\rho_1$ and $\rho_2$ be two Jack-positive specializations such that the (non-negative) series $\sum_{\lambda \in \mathbb{Y}} J_{\lambda}(\rho_1; \theta) \tilde{J}_{\lambda}(\rho_2;\theta)$ is finite. The {\em Jack (probability) measure} $\mathcal{J}_{\rho_1, \rho_2}$ on $\mathbb{Y}$ is defined through
\begin{equation}\label{S12E4}
\mathcal{J}_{\rho_1, \rho_2} (\lambda) = \frac{J_{\lambda}(\rho_1; \theta) \tilde{J}_{\lambda}(\rho_2;\theta)}{H_{\theta}(\rho_1;\rho_2)},
\end{equation}
where the normalization constant is given by
\begin{equation}\label{S12E5}
H_{\theta}(\rho_1; \rho_2) = \sum_{\lambda \in \mathbb{Y}} J_{\lambda}(\rho_1; \theta) \tilde{J}_{\lambda}(\rho_2;\theta).
\end{equation}
\end{definition}
\begin{remark}\label{S12R2} The construction of probability measures through specializations was first considered in \cite{Oko01} in the context of {\em Schur measures}. Since then, the construction has been extended to a much more general family of symmetric functions, called {\em Macdonald} symmetric functions (see \cite{BC14}), which includes the Jack symmetric functions as a special case.
\end{remark}

In this paper we will work with three kinds of specializations. The first is called a {\em pure-$\alpha$} specialization and corresponds to fixing non-negative $\alpha_1, \dots, \alpha_N$ and setting all other parameters to zero. If $\rho$ is such a specialization, we have that 
\begin{equation}\label{S12E6}
J_{\lambda}(\rho;\theta) = J_{\lambda}(\alpha_1, \dots, \alpha_N; \theta).
\end{equation}
In other words, a pure-$\alpha$ specialization corresponds to setting $x_i = \alpha_i$ for $i = 1, \dots, N$ and setting all other variables to zero. The second is called a {\em pure-$\beta$} specialization and corresponds to fixing non-negative $\beta_1, \dots, \beta_N$ and setting all other parameters to zero. If $\rho$ is such a specialization, we have that 
\begin{equation}\label{S12E7}
J_{\lambda}(\rho;\theta) = \tilde{J}_{\lambda'}(\theta\beta_1, \dots, \theta\beta_N; \theta^{-1}).
\end{equation}
In other words, a pure-$\beta$ specialization is the composition of a pure-$\alpha$ specialization, with $\alpha_i = \theta \beta_i$, and the canonical automorphism $\omega_{\theta^{-1}}$ on $\Lambda \otimes \mathbb{Q}(\theta)$, specified by
$$\omega_{\theta^{-1}} (p_k) = (-1)^{k-1} \theta^{-1} p_k,$$
see \cite[(10.6)]{Mac98}. Lastly, we will work with the {\em Plancherel specialization} $\tau_s$, which corresponds to fixing $\gamma = s \geq 0$ and setting all other parameters to zero. The latter can be interpreted as 
\begin{equation}\label{S12E8}
J_{\lambda}(\tau_s;\theta) = \lim_{n \rightarrow \infty} J_{\lambda}(\underbrace{s/n, \dots, s/n}_\text{$n$ times};\theta).
\end{equation}

%
%
\subsection{Laws of large numbers}\label{Section1.3} In this section we present the laws of large numbers we establish for the Jack measures from Definition \ref{DefJM}. In what follows we denote by $a^N$ the pure-$\alpha$ specialization with $\alpha_1 = \alpha_2 = \cdots = \alpha_N = a$, see (\ref{S12E6}). We also denote by $b_{\beta}^N$ the pure-$\beta$ specialization with $\beta_1 = \beta_2 = \cdots = \beta_N = \theta^{-1} \cdot b$, see (\ref{S12E7}). The next definition summarizes the six different classes of Jack measures that we consider, and how we scale their parameters. 

\begin{definition}\label{DefScale} We consider sequences of Jack-positive specializations $\rho_1(N), \rho_2(N)$ that fall into one of the following six cases.
\begin{enumerate}[label=\Roman*.]
    \item For fixed $a,b, \cm > 0$ with $ab \in (0,1)$, we let $\rho_1 = a^N$, $\rho_2 = b^M$, where $M = N\cm  + O(1)$.
    \item For fixed $a,b, \cm > 0$, we let $\rho_1 = a^N$, $\rho_2 = b_{\beta}^M$, where $M =  N \cm + O(1)$. 
    \item For fixed $a,t > 0$, we let $\rho_1 = a^N$, $\rho_2 = \tau_{s}$, where $s = Nt$. 
    \item For fixed $a,b, \cm > 0$ with $ab \in (0,1)$, we let $\rho_1 = a_{\beta}^N$, $\rho_2 = b_{\beta}^M$, where $M =  N \cm + O(1)$.
    \item For fixed $a,t > 0$, we let $\rho_1 = a_{\beta}^N$, $\rho_2 = \tau_{s}$, where $s = Nt$. 
    \item For fixed $t_1, t_2 > 0$, we let $\rho_1 = \tau_{s_1}$, $\rho_2 = \tau_{s_2}$, where $s_1 = N t_1$, $s_2 = N t_2$.
\end{enumerate}
In Cases I, II and IV, we assume that $M \in \mathbb{N}$ and $\cd$ is a large enough so that $\cd \geq |M - N\cm|$.
\end{definition}
\begin{remark}\label{Rem.Commute} In Definition \ref{DefScale} we are essentially considering each of the cases when $\rho_1$ and $\rho_2$ are either homogeneous pure-$\alpha$, pure-$\beta$ or Plancherel specializations. In principle, this should yield nine cases; however, in view of (\ref{S12E2}), we have that the Jack measure $\mathcal{J}_{\rho_1, \rho_2}$ is the same as $\mathcal{J}_{\rho_2, \rho_1}$. This reduces the number of distinct cases from nine to six.
\end{remark}

\begin{definition}\label{Def.Arccos} Throughout the paper we denote by $\arccos(x)$ the function, which agrees with the usual arccosine function on $[-1,1]$, is equal to $\pi$ if $x \leq -1$, and is equal to $0$ if $x \geq 1$. 
\end{definition}

We now turn to the first main result of the paper, whose proof is given in Section \ref{Section5}.\\
\begin{theorem}\label{thmMain} Fix $\theta > 0$. Suppose that $\lambda^N$ is a sequence of random partitions with distribution $\mathcal{J}_{\rho_1(N), \rho_2(N)}$ as in Definition \ref{DefJM}, where $\rho_1(N), \rho_2(N)$ are as in Definition \ref{DefScale}. Set $Z_i^N = \lambda^N_i - i \cdot \theta$ for $i \geq 1$, and let $\nu_N$ denote the random measures on $\mathbb{R}$, defined by
\begin{equation}\label{S13E1}
\nu_N = \frac{1}{N} \cdot \sum_{i = 1}^{\infty} \delta\left( \frac{Z_i^N}{N} \right).
\end{equation}
Then, $\nu_N$ converge vaguely in probability to a deterministic measure $\nu$ on $\mathbb{R}$. For each of the six cases in Definition \ref{DefScale} the measure $\nu$ has the following density $f(x)$ with respect to Lebesgue measure.
{\bf \raggedleft Case I.} Denoting the density by $f^I$, we have
\begin{equation}\label{MTE1}
 f^{I}(x) = \begin{cases} \dfrac{1 }{\theta \pi } \cdot \mathrm{arccos} \left( \dfrac{(1+ab)x +  ab \theta (\cm + 1)}{2 \sqrt{ab(x+\theta) (x + \cm \theta ) }}\right) & \mbox{ if } x > -\theta \min(\cm,1) \\ \theta^{-1} &\mbox{ if } x < -\theta \min(\cm,1).\end{cases}
\end{equation}
{\bf \raggedleft Case II.} Denoting the density by $f^{II}$, we have
\begin{equation}\label{MTE2}
 f^{II}(x) =\begin{cases} 0 &\mbox{ if } x > \cm, \\ \dfrac{1 }{\theta \pi } \cdot \mathrm{arccos} \left( \dfrac{(1-ab)x  +  ab(\cm-\theta)}{2 \sqrt{ab(x+\theta) (\cm - x) }}\right) & \mbox{ if } -\theta<x<\cm , \\ \theta^{-1} &\mbox{ if } x < - \theta.\end{cases}
\end{equation}
{\bf \raggedleft Case III.} Denoting the density by $f^{III}$, we have
\begin{equation}\label{MTE3}
 f^{III}(x) =\begin{cases} \dfrac{1 }{\theta \pi } \cdot \mathrm{arccos} \left( \dfrac{x+at\theta}{2 \sqrt{at\theta(x+\theta)}}\right) & \mbox{ if } x > -\theta, \\ \theta^{-1} &\mbox{ if } x < - \theta.\end{cases} 
\end{equation}
{\bf \raggedleft Case IV.} Denoting the density by $f^{IV}$, we have for $\alpha = \frac{ab (\cm + 1)+ 2\sqrt{ab \cm} }{1-ab}$
\begin{equation}\label{MTE4}
f^{IV}(x) =\begin{cases} 0 & \mbox{ if } x > \min(\cm, 1), \\
\dfrac{1}{\theta \pi } \cdot \mathrm{arccos} \left( \dfrac{\left(1+ab\right)x -ab(\cm+1 ) }{2 \sqrt{ab (1 - x)(\cm - x)}}\right)& \mbox{ if } -\alpha <x<\min(\cm, 1) , \\ \theta^{-1} & \mbox{ if } x < -\alpha.\end{cases}
\end{equation}
{\bf \raggedleft Case V.} Denoting the density by $f^{V}$, we have for $\alpha = at\theta + 2 \sqrt{at \theta}$
\begin{equation}\label{MTE5}
 f^{V}(x) = \begin{cases} 0 & \mbox{ if } x > 1, \\
 \dfrac{1}{\theta \pi } \cdot \mathrm{arccos}\left(\dfrac{x- at \theta}{2\sqrt{at \theta (1-x)}}\right) & \mbox{ if } -\alpha < x< 1, \\ \theta^{-1} &\mbox{ if } x < - \alpha .\end{cases} 
 \end{equation}
{\bf \raggedleft Case VI.} Denoting the density by $f^{VI}$, we have
\begin{equation}\label{MTE6}
 f^{VI}(x) = \begin{cases} 0 &\mbox{ if } x > 2 \theta \sqrt{t_1t_2}, \\
 \dfrac{1}{\theta \pi} \cdot \arccos \left( \dfrac{ x}{2\theta \sqrt{t_1t_2}} \right)& \mbox{ if } -2\theta \sqrt{t_1t_2} < x < 2 \theta \sqrt{t_1t_2}, \\
\theta^{-1} & \mbox{ if } x < - 2\theta \sqrt{t_1t_2}.
\end{cases}
\end{equation}
In the above equations $\arccos$ is as in Definition \ref{Def.Arccos}. 
\end{theorem}
\begin{remark}\label{S13R2} 
Let us briefly elaborate on the meaning of Theorem \ref{thmMain}. The measures $\nu_N$ in (\ref{S13E1}) are random measures in $\mathcal{M}_{\mathbb{R}}$ -- the space of locally bounded measures on $\mathbb{R}$. If one equips $\mathcal{M}_{\mathbb{R}}$ with the {\em vague} topology, it becomes a {\em Polish space} (i.e. a complete separable metric space), and then the convergence $\nu_N$ to $\nu$ is simply that of weak convergence as {\em random elements} in the metric space $\mathcal{M}_{\mathbb{R}}$ in the sense of \cite[Section 3]{Billing}. Since $\nu$ is itself deterministic, an equivalent way to say that $\nu_N$ converge vaguely in probability to $\nu$ is to say that for each compactly supported continuous function $h$ on $\mathbb{R}$ we have 
$$\int_{\mathbb{R}} h(x) \nu_N(dx) \Rightarrow
\int_{\mathbb{R}} h(x) \nu(dx) \left( = \int_{\mathbb{R}} h(x) f(x)\, dx \right),$$
where the latter is simply weak convergence of random variables. We refer the reader \cite[Chapter 4]{Kall} for background on the space $\mathcal{M}_{\mathbb{R}}$ and the vague topology.    
\end{remark}

%
%
\subsection{A detailed view of Case II}\label{Section1.4} In the course of proving Theorem \ref{thmMain} we show that (under appropriate conditioning) the Jack measures in Definition \ref{DefScale} satisfy certain technical assumptions, see Assumptions \ref{Ass.Finite}, \ref{Ass.Infinite}, \ref{Ass.Loop1} and \ref{Ass.Loop2}. The latter in term verify many of the required assumptions in \cite{BGG17, DK22, DD22, DZ23, GH17}. With some additional but modest effort one can verify the remaining assumptions in those papers and obtain several results for the Jack measures in Definition \ref{DefScale}. We illustrate this point by deriving four kinds of statements for the measures in Case II of Definition \ref{DefScale}, which are summarized in the following theorem, whose proof is presented in Section \ref{Section6}.

\begin{theorem}\label{thmMain2} Fix $\theta > 0$. Suppose that $\lambda^N$ is a sequence of random partitions with distribution $\mathcal{J}_{\rho_1(N), \rho_2(N)}$ as in Definition \ref{DefJM}, where $\rho_1(N), \rho_2(N)$ are as in Case II of Definition \ref{DefScale}. Set $\ell_i^N = \lambda_i^N + (N-i)\cdot \theta$ for $i \geq 1$.\\

{\bf \raggedleft Global CLT.} Fix any functions $f_1, \dots, f_n$ that are analytic in a complex neighborhood $U$ of $[0, \cm + \theta]$ and real-valued when restricted to $U \cap \mathbb{R}$. Then, the random variables 
$$X_k^N = \sum_{i = 1}^N f_k (\ell_i^N/N) - \sum_{i = 1}^N \mathbb{E}[f_k (\ell_i^N/N)] \mbox{ for } k = 1, \dots, n,$$
converge jointly in the sense of moments to a Gaussian vector $(X_1, \dots, X_n)$ with $\mathbb{E}[X_i] = 0$ for $i = 1, \dots, n$ and for $1 \leq i, j \leq n$
\begin{equation}\label{Eq.CLT}
\mathbb{E}[X_iX_j] = \frac{1}{(2\pi \im)^2} \oint_{\gamma} \oint_{\gamma}  \frac{\theta^{-1}f_i(z) f_j(w)}{(z-w)^2} \cdot \left( \frac{(z-\alpha)(w-\beta) + (w - \alpha)(z-\beta)}{\sqrt{(z-\alpha)(z-\beta)} \sqrt{(w-\alpha)(w-\beta)}} - 1 \right).
\end{equation}
In equation (\ref{Eq.CLT}) $\gamma$ is a positively oriented contour that is contained in $U$ and encloses the interval $[0, \cm + \theta]$,
\begin{equation}\label{Eq.AlphaBeta}
\alpha = \frac{ab \cm + \theta - 2 \sqrt{ab\cm \theta}}{1 + ab}, \hspace{2mm} \beta =  \frac{ab \cm + \theta + 2 \sqrt{ab\cm \theta}}{1 + ab},
\end{equation}
and we picked the branch of the square root so that $\sqrt{(z-\alpha)(z-\beta)}$ is analytic in $\mathbb{C}\setminus [\alpha, \beta]$ and $\sqrt{(z-\alpha)(z-\beta)} \sim z$ as $|z| \rightarrow \infty$. \\

{\bf \raggedleft Global LDP.} Let us denote by $\mathcal{M}(\mathbb{R})$ the set of probability measures on $\mathbb{R}$, and by $\mathcal{M}_{\theta}([0, \cm + \theta])$ the set of probability measures on $\mathbb{R}$ whose support is in $[0, \cm + \theta]$ and which have a density with respect to Lebesgue measure that is bounded by $\theta^{-1}$. We endow $\mathcal{M}(\mathbb{R})$ with the weak topology. We define $V$ on $[0, \cm + \theta]$ by 
$$V(x) = x\log(x)+(\cm+\theta-x)\log (\cm+\theta-x)-x\log (ab),$$
and for $\mu \in \mathcal{M}_{\theta}([0, \cm + \theta])$ set
$$E_V(\mu) =\iint_{[0, \cm + \theta]^2} \left[ \log|x-y|^{-1} +  \frac{V(x)}{2} + \frac{ V(y)}{2}  \right] \mu(dx) \mu(dy).$$
Lastly, let $\mu^{II}$ be the probability measure whose density, denoted $\mu^{II}(x)$, is given by
\begin{equation}\label{Eq.muII}
\mu^{II}(x) = \frac{{\bf 1}\{x \in (0, \cm + \theta )\} }{\theta \pi } \cdot \mathrm{arccos} \left( \frac{(1-ab)x + ab\cm-\theta}{2 \sqrt{ab x (\cm +\theta - x) }}\right),
\end{equation}
where we recall that $\arccos$ is as in Definition \ref{Def.Arccos}. If we define $\mu_N =N^{-1} \sum_{i = 1}^N \delta (\ell_i^N/N)$, then the laws of these measures in $\mathcal{M}(\mathbb{R})$ satisfy an LDP with speed $N^2$ and good rate function
\begin{equation}\label{Eq.RateGLDP}
I^{\mathsf{Global}} (\mu) = \begin{cases} \theta  \left[ E_V(\mu) - E_V(\mu^{II}) \right], &\mbox{ for } \mu \in \mathcal{M}_{\theta}([0, \cm + \theta]) \\ \infty &\mbox{ for } \mu \in \mathcal{M}(\mathbb{R}) \setminus \mathcal{M}_{\theta}([0, \cm + \theta]). \end{cases}
\end{equation}

{\bf \raggedleft Edge LDP.} For any $t \in [\theta, \cm + \theta)$ we have 
\begin{equation}\label{Eq.ELDPMain}
\lim_{N \rightarrow \infty} \frac{1}{N} \log \mathbb{P} (\ell^N_1 \geq tN) = -I^{\mathsf{Edge}}(t).
\end{equation}
When $\cm \in (0, ab \theta]$ and $t \in [\theta, \cm + \theta)$, or $\cm > ab \theta$ and $t \in [\theta, \beta)$, we have $I^{\mathsf{Edge}}(t) = 0$. When $\cm > ab \theta$ and $t \in [\beta, \cm + \theta)$, we have
\begin{equation}\label{Eq.RateELDP2}
I^{\mathsf{Edge}}(t) =  \int_{\beta}^{t} 2 \log \left( \dfrac{(1-ab)y+ab\cm-\theta + (1 + ab) \sqrt{(y- \alpha)(y-\beta)}}{2 \sqrt{aby (m + \theta -y)}  } \right)  dy ,
\end{equation}
where $\alpha, \beta$ are as in (\ref{Eq.AlphaBeta}). \\

{\bf \raggedleft Edge universality.} We assume $\beta = 2\theta \geq 1$, $\cm > ab\theta$ and $ab\cm \neq \theta$. Let $(X^N_1, \dots, X^N_N)$ be distributed as a $G\beta E$-ensemble, i.e., the random vector has density
\begin{equation}\label{Eq.GSE}
f(x_1, \dots, x_N) = \frac{{\bf 1} \{x_1 \geq  x_2 \geq \cdots \geq x_N \}}{Z_N} \cdot \prod_{1 \leq i < j \leq N} (x_i - x_j)^{\beta} \cdot \prod_{i = 1}^N e^{-N \beta x_i^2/4},
\end{equation}
where $Z_N$ is a normalization constant. In addition, let $\gamma^{\mathsf{sc}}_{m, N}$ and $\gamma^{II}_{m,N}$ for $m = 1, \dots, N$ be the quantiles defined through 
$$\frac{m-1/2}{N} = \int_{\gamma^{\mathsf{sc}}_{m, N}}^{\infty} \mu^{\mathsf{sc}}(x) dx \mbox{ and } \frac{m-1/2}{N} = \int_{\gamma^{II}_{m, N}}^{\infty} \mu^{II}(x) dx,$$
where $\mu^{\mathsf{sc}}(x) = (2\pi)^{-1} {\bf 1}\{|x| \leq 2 \} \cdot \sqrt{4 - x^2}$ is the semicircle law, and $\mu^{II}$ is as in (\ref{Eq.muII}). Then, for any $n \geq 1$ and any continuously differentiable compactly supported function $F: \mathbb{R}^n \rightarrow \mathbb{R}$, we have
\begin{equation}\label{Eq.EdgeUniversality}
\begin{split}
\lim_{N \rightarrow \infty} \mathbb{E}\left[ F(\tilde{\ell}_1^N, \dots, \tilde{\ell}_n^N) \right] - \mathbb{E}\left[ F(\tilde{X}_1^N, \dots, \tilde{X}_n^N) \right] = 0,
\end{split}
\end{equation}
where $\tilde{X}_i^N = N^{2/3} \cdot (X_i^N/N - \gamma^{\mathsf{sc}}_{i, N})$, $\tilde{\ell}^N_i = (Ns_B)^{2/3} \cdot (\ell_i^N/N - \gamma^{II}_{i, N})$ for $i =1, \dots, N$, and 
$$s_B = \frac{\sqrt{(1+ab)(\beta - \alpha)} }{2 \theta \sqrt{ ab \beta ( \cm + \theta - \beta)}} \mbox{ with $\alpha, \beta$ as in (\ref{Eq.AlphaBeta})}.$$
\end{theorem}

\begin{remark}\label{GLDPRem} From the ``Global LDP'' part of the theorem we see that the measures $\mu_N$ concentrate around the deterministic measure $\mu^{II}$, which is sometimes referred to as the {\em equilibrium measure}. Its density $\mu^{II}(x)$ agrees with $f^{II}(x)$ from (\ref{MTE2}) upon restricting the latter to $(-\theta, \infty)$ and shifting the argument $x$ by $\theta$. The shift in the argument can be traced to the identity $Z^N_i = \ell_i^N + N \theta$. On the other hand, as will be shown in Section \ref{Section5.2}, we have that $Z_i^N = - i \cdot \theta$ for $i > N$ deterministically, which accounts for the constant $\theta^{-1}$ value of $f^{II}$ on $(-\infty, - \theta)$.
\end{remark}

\begin{remark}\label{CLTRem} The regions where $\mu^{II}(x) = 0$ are called {\em voids}, the regions where $\mu^{II}(x) = \theta^{-1}$ are called {\em saturated regions} and the regions where $\mu^{II}(x) \in (0,\theta^{-1})$ are called {\em bands} -- this terminology is borrowed from \cite{BKMM03}. If $\alpha, \beta$ are as in (\ref{Eq.AlphaBeta}), we have that $(0, \alpha]$ is a void if $ab\cm > \theta$ and a saturated region if $ab\cm < \theta$. In addition, $[\beta, \cm + \theta)$ is a void if $\cm > ab \theta$ and a saturated region if $\cm < ab\theta$. The interval $(\alpha, \beta)$ is the only band. We observe that the covariance in (\ref{Eq.CLT}) depends on $\mu^{II}$ only through the endpoints $\alpha, \beta$ of the band region, which is consistent with the asymptotic behavior of discrete $\beta$-ensembles from \cite{BGG17} in the {\em one-cut} case. 
\end{remark}

\begin{remark}\label{ELDPRem} One expects $\ell_1/N$ to concentrate around $b_{\cm}$ -- the rightmost endpoint of the support of the measure $\mu^{II}$. When $\cm \leq ab\theta$, we have from Remark \ref{CLTRem} that $b_{\cm} = \cm + \theta$, which is why we have $I^{\mathsf{Edge}}(t) = 0$ for $t \in [\theta, \cm + \theta)$ in this case. When $\cm > ab\theta$, we have that $b_{\cm} = \beta$, which is why we have that $I^{\mathsf{Edge}}(t) = 0$ for $t \in [\theta, \beta)$ and $I^{\mathsf{Edge}}(t) > 0$ when $t \in (\beta, \cm + \theta)$. We mention that since $\ell_1^N \in [0, M(N) + (N-1)\theta]$, we have that $\mathbb{P} (\ell^N_1 \geq tN) = 0$ for all large $N$ when $t > \cm + \theta$, and $\mathbb{P} (\ell^N_1 \geq tN) = 1$ for all large $N$ when $t < \theta$. Consequently, the ``Edge LDP'' part of the theorem handles all meaningful $t$'s with the exception of $t = \cm + \theta$. The value $t = \cm + \theta$ is special and excluded from consideration as was done in \cite{DD22}. For example, the limit in (\ref{Eq.ELDPMain}) may not exist for some parameter choices at $t = \cm + \theta$, as different subsequences might converge to different values. 
\end{remark}

\begin{remark}\label{EdgeRem} Let us briefly discuss the ``Edge universality'' component of Theorem \ref{thmMain2}, whose proof ultimately follows from \cite[Theorem 1.11]{GH17} by Guionnet and Huang. The condition $ab\cm \neq \theta$ is a technical assumption we make to ensure $\alpha > 0$, where $\alpha$ is defined in (\ref{Eq.AlphaBeta}). This positivity assumption is required by \cite[Assumption 1.5]{GH17}. Additionally, the assumption $\cm > ab\theta$ guarantees that the rightmost endpoint of the support of $\mu^{II}$, which is $b_{\cm} = \beta$ as in (\ref{Eq.AlphaBeta}), satisfies $b_{\cm} \neq \cm + \theta $, and $\mu^{II}$ exhibits a square root behavior near $b_{\cm}$. This behavior is also necessary to derive the edge asymptotics from \cite[Theorem 1.11]{GH17}. The proof of \cite[Theorem 1.11]{GH17} relies on the edge universality result for continuous $\beta$-ensembles by Bourgade, Erd\H{o}s and Yau in \cite{BEY14}, which in turn requires $\beta = 2\theta \geq 1$. Our formulation of (\ref{Eq.EdgeUniversality}) is tailored to facilitate a direct comparison with \cite[Theorem 1.11]{GH17}, though the latter actually implies a significantly stronger statement than what we state here; we omit those details for the sake of simplicity. Finally, using the well-known asymptotics of the Gaussian $\beta$-ensembles, statement (\ref{Eq.EdgeUniversality}) is readily seen to be equivalent to the finite-dimensional convergence of $-\tilde{\ell}_1^N, \dots, -\tilde{\ell}_n^N$ to the $n$ smallest eigenvalues of the {\em stochastic $\beta$-Airy operator}; see \cite[Corollary 2.2]{BEY14} and \cite{RRV11}. 
\end{remark}

%
%
\subsection{Main ideas and paper outline}\label{Section1.5} In Section \ref{Section2} we show that if we condition the Jack measures from Definition \ref{DefScale} on the event $E_K^R = \{\ell(\lambda) \leq K, \lambda_1 \leq R\}$ -- that is, the event that $\lambda$ is contained in the $K \times R$ rectangle -- then, for any $K \in \mathbb{N}$ and $R \in [0, \infty]$, the resulting measures can be interpreted as discrete $\beta$-ensembles in the sense of \cite{BGG17}. The general strategy is to analyze the asymptotics of $\mathcal{J}_{\rho_1(N), \rho_2(N)}(\cdot | E_K^R)$ when $K$ and $R$ scale together with $N$, and then argue that the conditioned and unconditioned Jack measures, $\mathcal{J}_{\rho_1(N), \rho_2(N)}(\cdot | E_K^R)$ and $\mathcal{J}_{\rho_1(N), \rho_2(N)}$, are close in total variation distance. In Section \ref{Section3}, we introduce two general sets of assumptions for discrete $\beta$-ensembles -- see Assumptions \ref{Ass.Finite} and \ref{Ass.Infinite} -- which imply a global LLN and an edge large deviation estimate, as established in \cite{DD22}. We then proceed to check, using some careful expansions of Gamma functions, that $\mathcal{J}_{\rho_1(N), \rho_2(N)}(\cdot | E_K^R)$ satisfy these assumptions. This leads to two key conclusions: (1) certain measures derived from $\mathcal{J}_{\rho_1(N), \rho_2(N)}(\cdot | E_K^R)$ converge weakly in probability, and (2) by allowing the $K \times R$ rectangle to grow sufficiently quickly, the total variation distance between the conditioned and unconditioned measures decays exponentially fast in $N$.

The arguments in Sections \ref{Section2} and \ref{Section3} are essentially sufficient to show that $\nu_N$ in Theorem \ref{thmMain} converge vaguely in probability to {\em some} deterministic measure $\nu$ on $\mathbb{R}$. However, these arguments do not provide an explicit description of $\nu$. In \cite{DD22}, it was shown that $\nu$ can be related to the minimizer of a certain variational functional. In Case III, and in Case II when $ab = 1$, this minimizer was computed explicitly in \cite{DD22} by solving intricate integral equations. In the present paper, we adopt a different approach. Instead of solving integral equations, we use the discrete loop equations (also known as Nekrasov's equations) introduced in \cite{BGG17}, in combination with results from \cite{DK22}, to derive an explicit formula for $\nu$. This approach is developed in Section \ref{Section4}, where we introduce two additional assumptions for discrete $\beta$-ensembles, see Assumption \ref{Ass.Loop1} and \ref{Ass.Loop2}, which imply a general formula for the limiting particle density, stated in Lemma \ref{Lem.MuThroughR}. We then verify that, for each of the six cases in Definition \ref{DefScale}, the conditioned measures $\mathcal{J}_{\rho_1(N), \rho_2(N)}(\cdot | E_K^R)$  satisfy the assumptions of Lemma \ref{Lem.MuThroughR}. This allows us to express the limiting density in terms of a certain analytic function $R_{\mu}$. For all six cases under consideration, this function turns out to be a polynomial of degree at most two, that can be computed by expanding the Stieltjes transform of the limiting density near infinity. 

Using the results from Sections \ref{Section3} and \ref{Section4}, the proof of Theorem \ref{thmMain} becomes relatively straightforward and is presented in Section \ref{Section5}. While some cases -- such as Case II -- fit seamlessly into the framework developed earlier, most of the others require some work to be incorporated into the same approach. For instance, to establish that the total variation distance between $\mathcal{J}_{\rho_1(N), \rho_2(N)}(\cdot | E_K^R)$ and $\mathcal{J}_{\rho_1(N), \rho_2(N)}$ decays exponentially fast in $N$, we rely on the large deviation estimate from Proposition \ref{Prop.UpperTail} in Cases I and III. In Cases IV and V, we deduce this decay by relating the relevant measures to their {\em duals}, which correspond precisely to Cases I and III; see Lemmas \ref{Lem.CaseIVLengthBound} and \ref{Lem.CaseVLengthBound}. Finally, in Case VI we prove exponential decay in Lemma \ref{Lem.CaseVILengthBound} by expressing the measures as {\em Poissonized Jack measures} and adapting an argument due to Fulman \cite{Fulman03}. 

In Section \ref{Section6}, we present the proof of Theorem \ref{thmMain2}, which primarily consists of verifying a set of assumptions from \cite{DD22, DZ23, DK22, GH17}, most of which have already been addressed in Sections \ref{Section3} and \ref{Section4}.

%
%

\subsection*{Acknowledgments}\label{Section1.6} ED was partially supported by the Simons Award TSM-00014004 from the Simons Foundation International.

%
%
\section{Jack measures as $\beta$-ensembles}\label{Section2} {\em Discrete $\beta$-ensembles} or {\em $\beta$ log-gases} are probability measures that were introduced in \cite{BGG17} as integrable discretizations of continuous log-gases. For future reference we give their definition here.

\begin{definition}\label{Def.BetaEnsembles} Fix $K \in \mathbb{N}$, $R \in [0,\infty]$ and $\theta > 0$. We define the sets 
\begin{equation}\label{Eq.State} 
\begin{split}
&\mathbb{Y}^R_K = \{(\lambda_1, \dots, \lambda_K): 0 \leq \lambda_K \leq \lambda_{K-1} \leq \cdots \leq \lambda_1 \leq R \mbox{ and } \lambda_i \in \mathbb{Z} \mbox{ for } i = 1, \dots, K\},\\
&\mathbb{W}^{\theta, R}_{K} = \{(\ell_1, \dots, \ell_K): \ell_i = \lambda_i + (K-i) \cdot \theta \mbox{ for }(\lambda_1, \dots, \lambda_K) \in \mathbb{Y}^R_K \}. 
\end{split}
\end{equation}
A discrete $\beta$-ensemble is a probability measure $\mathbb{P}$ on $\mathbb{W}_K^{\theta, R}$ that takes the form 
\begin{equation}\label{Eq.BetaEnsemble}
\mathbb{P}(\ell_1, \dots, \ell_K) = \frac{1}{Z_K} \prod_{1 \leq i < j \leq K} \frac{\Gamma(\ell_i - \ell_j + 1)\Gamma(\ell_i - \ell_j + \theta)}{\Gamma(\ell_i - \ell_j) \Gamma(\ell_i - \ell_j + 1 - \theta)} \cdot \prod_{i = 1}^K w(\ell_i; K).
\end{equation}
Here, $\Gamma(x)$ is the {\em Euler Gamma function} and $Z_K \in (0,\infty)$ is a normalization constant. If $R < \infty$ we assume that $w(x;K)$ is a positive continuous function on $[0, R + (K-1)\theta]$. If $R = \infty$ we assume that $w(x;K)$ is a positive continuous function on $[0,\infty)$ such that for all large $x$
\begin{equation}\label{Eq.WeightDecay}
\log w(x;N) \leq - [\theta \cdot N + 1] \cdot \log \left(1 + (x/N)^2 \right).
\end{equation}
We mention that (\ref{Eq.WeightDecay}) ensures that $Z_K < \infty$, so that (\ref{Eq.BetaEnsemble}) defines an honest probability measure when $R = \infty$, see \cite[Lemma 7.1]{DD22}.
\end{definition}

In this section we show that the Jack measures from Definition \ref{DefScale} are well-defined and can be interpreted as discrete $\beta$-ensembles. We do this in Sections \ref{Section2.2}-\ref{Section2.7}, which are named after the six cases in Definition \ref{DefScale}. We continue with the same notation as in Sections \ref{Section1.2} and \ref{Section1.3}.

%
%
\subsection{Evaluations of Jack symmetric functions}\label{Section2.1} In this section we compute the Jack and dual Jack functions at the different specializations, considered in Definition \ref{DefScale}. The formulas we are after will be for the case when $\lambda \in \mathbb{Y}$ satisfies $\ell(\lambda) \leq K$, and in terms of the variables
\begin{equation}\label{S21E1}
\ell_i = \lambda_i + (K-i) \cdot \theta.
\end{equation}

%
%
\subsubsection{Pure $\alpha$-specialization}\label{Section2.1.1} In this section we find formulas for $J_{\lambda}(a^M; \theta)$ and $\tilde{J}_{\lambda}(a^M; \theta)$. Using the homogeneity of Jack symmetric functions, and \cite[Chapter VI, (10.20)]{Mac98}, we have for $\lambda \in \mathbb{Y}$ with $\ell(\lambda) \leq K$ that 
\begin{equation}\label{S211E1}
J_{\lambda}(a^M;\theta) = a^{|\lambda|} \cdot \prod_{ \square \in \lambda} \frac{M\theta + a'(\square) - \theta \cdot \ell'(\square)}{a(\square) + \theta \cdot \ell(\square) + \theta} = a^{|\lambda|} \cdot \prod_{i = 1}^K \prod_{j = 1}^{\lambda_i} \frac{M\theta + (j-1) - \theta \cdot (i-1)}{\lambda_i - j + \theta \cdot (\lambda_j' - i) + \theta}.
\end{equation}
As shown in the displayed equation under \cite[(6.4)]{DD22}, the denominator in (\ref{S211E1}) can be rewritten as
\begin{equation}\label{Eq.ArmLegPropTheta}
\prod_{i = 1}^K \prod_{j = 1}^{\lambda_i} \frac{1}{\lambda_i - j + \theta \cdot (\lambda_j' - i) + \theta} = \prod_{1 \leq i < j \leq K} \frac{\Gamma(\ell_i - \ell_j + \theta)}{\Gamma(\ell_i - \ell_j)} \cdot \prod_{i = 1}^K \frac{\Gamma(\theta)}{\Gamma(\ell_i + \theta)}.
\end{equation}
Using the functional equation $\Gamma(z+1) = z\Gamma(z)$, we see that the numerator in (\ref{S211E1}) can be rewritten for $K \leq M$ as
$$\prod_{i = 1}^K \prod_{j = 1}^{\lambda_i} [M\theta + (j-1) - (i-1)\theta] = \prod_{i = 1}^K \frac{\Gamma((M-K+1)\theta + \ell_i)}{\Gamma((M-i +1)\theta)}.$$
We mention that we imposed the assumption that $K \leq M$ to avoid obtaining $\frac{\Gamma(0)}{\Gamma(0)}$ when $i = M+1$ and $\lambda_{M+1} = 0$, which should be resolved as equal to $1$ when comparing with the left side above. Overall, we obtain for $\lambda \in \mathbb{Y}$ with $\ell(\lambda) \leq K$ and $K \leq M$ that 
\begin{equation}\label{S211E2}
J_{\lambda}(a^M;\theta) = \const \times \prod_{1 \leq i < j \leq K} \frac{\Gamma(\ell_i - \ell_j + \theta)}{\Gamma(\ell_i - \ell_j)} \cdot \prod_{i = 1}^K \frac{a^{\ell_i} \cdot \Gamma((M-K+1)\theta + \ell_i)}{\Gamma(\ell_i + \theta)},
\end{equation}
for some {\em positive} constant that depends on $a,\theta,M,K$ but not $\lambda$. 

From (\ref{S12E2}) and (\ref{S211E1}) we have for $\lambda \in \mathbb{Y}$ with $\ell(\lambda) \leq K$
\begin{equation}\label{S211E3}
\tilde{J}_{\lambda}(a^M;\theta) = J_\lambda(a^M;\theta) \cdot \prod_{\square \in \lambda} \frac{a(\square) + \theta \cdot \ell(\square) + \theta}{a(\square) + \theta \cdot \ell(\square) + 1} =  a^{|\lambda|} \cdot \prod_{i = 1}^K \prod_{j = 1}^{\lambda_i} \frac{M\theta + (j-1) - \theta \cdot (i-1)}{\lambda_i - j + \theta \cdot (\lambda_j' - i) + 1}.
\end{equation}
The numerator in (\ref{S211E3}) can be rewritten as before, while the denominator can be rewritten as
\begin{equation}\label{Eq.ArmLegProd1}
\begin{split}
& \prod_{i = 1}^K \prod_{j = 1}^{\lambda_i} \frac{1}{\lambda_i - j + \theta \cdot (\lambda_j' - i) + 1} = \prod_{1 \leq i \leq k \leq K} \prod_{j =\lambda_{k+1} + 1}^{\lambda_k} \frac{1}{\lambda_i - j + \theta(k - i) + 1 }  \\
&= \prod_{1 \leq i \leq k \leq K}  \frac{\Gamma(\lambda_i - \lambda_k + \theta(k - i) + 1)}{\Gamma(\lambda_i - \lambda_{k+1} + \theta(k-i) + 1)} = \prod_{1 \leq i < k \leq K}  \frac{\Gamma(\lambda_i - \lambda_k + \theta(k - i) + 1 )}{\Gamma(\lambda_i - \lambda_{k}+ \theta(k-i-1) + 1 )} \\
& \times \prod_{i = 1}^{K} \frac{\Gamma(1)}{\Gamma(\lambda_i + \theta(K-i) + 1)} =  \prod_{1 \leq i < j \leq K} \frac{\Gamma(\ell_i - \ell_j + 1)}{\Gamma(\ell_i - \ell_j + 1 - \theta)}  \cdot \prod_{i = 1}^{K}  \frac{ 1}{\Gamma( \ell_i + 1)}.
\end{split}
\end{equation}
Overall, we obtain for $\lambda \in \mathbb{Y}$ with $\ell(\lambda) \leq K$ and $K \leq M$ that 
\begin{equation}\label{S211E4}
\tilde{J}_{\lambda}(a^M;\theta) = \const \times \prod_{1 \leq i < j \leq K} \frac{\Gamma(\ell_i - \ell_j + 1)}{\Gamma(\ell_i - \ell_j + 1 - \theta)}  \cdot \prod_{i = 1}^K \frac{a^{\ell_i} \cdot \Gamma((M-K+1)\theta + \ell_i)}{\Gamma(\ell_i + 1)},
\end{equation}
for some positive constant that depends on $a,\theta,M,K$ but not $\lambda$.

%
%
\subsubsection{Pure $\beta$-specialization}\label{Section2.1.2} In this section we find formulas for $J_{\lambda}(a_{\beta}^M; \theta)$ and $\tilde{J}_{\lambda}(a_{\beta}^M; \theta)$. Using the homogeneity of $\tilde{J}_{\lambda}$, and \cite[(6.7)]{DD22}, we get for $\lambda \in \mathbb{Y}$ with $\lambda_1 \leq M$ and $\ell(\lambda)\leq K$ 
\begin{equation}\label{S212E1}
\tilde{J}_{\lambda}(a_{\beta}^M;\theta) = \const \times \prod_{1 \leq i < j \leq K} \frac{\Gamma(\ell_i - \ell_j + 1)}{\Gamma(\ell_i - \ell_j + 1 - \theta)}  \cdot \prod_{i = 1}^K \frac{a^{\ell_i} }{\Gamma(\ell_i + 1) \Gamma(M +  K \theta - \ell_i + 1 - \theta)},
\end{equation}
for some positive constant that depends on $a,\theta,M,K$ but not $\lambda$. In addition, we have for $\lambda \in \mathbb{Y}$ with $\lambda_1 \leq M$ and $\ell(\lambda) \leq K$ 
\begin{equation}\label{S212E2}
\begin{split}
J_\lambda(a_\beta^M;\theta) &= \tilde{J}_\lambda(a_\beta^M;\theta) \cdot \prod_{\square\in\lambda}\frac{\lambda_i - j + \theta \cdot (\lambda_j' - i) + 1}{\lambda_i - j + \theta \cdot (\lambda_j' - i) + \theta}\\
&=\const\times\prod_{1 \leq i < j \leq K}\frac{\Gamma(\ell_i-\ell_j+\theta)}{\Gamma(\ell_i-\ell_j)}\prod_{i=1}^K\frac{a^{\ell_i}}{\Gamma(\ell_i+\theta)\Gamma(M+K\theta-\ell_i+1-\theta)},
\end{split}
\end{equation}
for some positive constant that depends on $a,\theta,M,K$ but not $\lambda$. We mention that the first equality in (\ref{S212E2}) used (\ref{S12E2}), and in going from the first to the second line we used (\ref{Eq.ArmLegPropTheta}), (\ref{Eq.ArmLegProd1}) and (\ref{S212E1}).

%
%
\subsubsection{Plancherel specialization}\label{Section2.1.3} In this section we find formulas for $J_{\lambda}(\tau_s; \theta)$ and $\tilde{J}_{\lambda}(\tau_s; \theta)$. Combining (\ref{S12E8}) and (\ref{S211E1}), we get for $\lambda \in \mathbb{Y}$ with $\ell(\lambda) \leq K$
\begin{equation}\label{Eq.PlanchEval}
J_{\lambda}(\tau_s;\theta)  = \lim_{n \rightarrow \infty}  \prod_{i = 1}^K \prod_{j = 1}^{\lambda_i} \frac{(s/n)[n\theta + (j-1) - \theta \cdot (i-1)]}{\lambda_i - j + \theta \cdot (\lambda_j' - i) + \theta} =  \prod_{i = 1}^K \prod_{j = 1}^{\lambda_i} \frac{s\theta}{\lambda_i - j + \theta \cdot (\lambda_j' - i) + \theta}.
\end{equation}
The last equation and (\ref{Eq.ArmLegPropTheta}) give for $\lambda \in \mathbb{Y}$ with $\ell(\lambda) \leq K$
\begin{equation}\label{S213E1}
J_{\lambda}(\tau_s;\theta)  = \const \times \prod_{1 \leq i < j \leq K} \frac{\Gamma(\ell_i - \ell_j + \theta)}{\Gamma(\ell_i - \ell_j)} \cdot \prod_{i = 1}^K \frac{(s\theta)^{\ell_i}}{\Gamma(\ell_i + \theta)},
\end{equation}
for some positive constant that depends on $s,\theta,K$ but not $\lambda$. From \cite[(6.6)]{DD22} we also have
\begin{equation}\label{S213E2}
\tilde{J}_{\lambda}(\tau_s;\theta)  = \const \times \prod_{1 \leq i < j \leq K} \frac{\Gamma(\ell_i - \ell_j + 1)}{\Gamma(\ell_i - \ell_j + 1 - \theta)} \cdot \prod_{i = 1}^K \frac{(s\theta)^{\ell_i}}{\Gamma(\ell_i + 1)}=J_{\lambda'}(\tau_{s\theta};\theta^{-1}),
\end{equation}
for some positive constant that depends on $s,\theta,K$ but not $\lambda$. The leftmost and rightmost sides in (\ref{S213E2}) are equal due to (\ref{S12E2}), (\ref{S12E8}) and (\ref{S211E1}). We mention that there is a small typo in \cite[(6.6)]{DD22}, and the ``$N(N-1)$'' in that equation should be multiplied by $\theta$.

%
%

\subsection{Case I: $\rho_1 = a^N$, $\rho_2 = b^M$}\label{Section2.2} From \cite[Section 6.10]{Mac98} we have that 
\begin{equation}\label{Eq.MacToJack}
J_{\lambda}(X;\theta) = \lim_{q \rightarrow 1-} P_{\lambda}(X;q,q^{\theta}) \mbox{ and } \tilde{J}_{\lambda}(X;\theta) = \lim_{q \rightarrow 1-} Q_{\lambda}(X;q,q^{\theta}),
\end{equation}
where $P_{\lambda}(X;q,t), Q_{\lambda}(X;q,t)$ are the Macdonald symmetric functions. Using the last identity, setting $t = q^{\theta}$ in \cite[(2.27)]{BC14} and letting $q \rightarrow 1-$, we conclude that for any two Jack-positive specializations $\rho_1, \rho_2$ we have 
\begin{equation}\label{Eq.PartFunctSpec}
 \sum_{\lambda \in \mathbb{Y}} J_{\lambda}(\rho_1;\theta) \tilde{J}_{\lambda}(\rho_2;\theta) = \exp \left( \sum_{k=1}^{\infty}\frac{\theta p_k(\rho_1)p_k(\rho_2)}{k}\right),
\end{equation}
provided that the sum in the exponential converges absolutely. Specializing the last identity to $\rho_1 = a^N$, $\rho_2 = b^M$ with $ab < 1$, we see that for $H_{\theta}(\rho_1; \rho_2)$ as in (\ref{S12E5}) we have
$$H_{\theta}(a^N; b^M) =  \exp \left( \sum_{k=1}^{\infty}\frac{\theta N M (ab)^k }{k}\right) = \frac{1}{(1-ab)^{\theta NM}} \in (0, \infty).$$
In particular, we conclude that the Jack measure $\mathcal{J}_{a^N, b^M}$ is well-defined.\\
We next note that 
\begin{equation}\label{Eq.Vanish}
J_{\lambda}(a^N;\theta) = \tilde{J}_{\lambda}(a^N;\theta) = 0 \mbox{ if } \ell(\lambda) > N \mbox{, and } J_{\lambda}(a_{\beta}^N;\theta) = \tilde{J}_{\lambda}(a_{\beta}^N;\theta) = 0 \mbox{ if } \lambda_1 > N.
\end{equation}
The first equality can be deduced from \cite[(10.18)]{Mac98}, and the fact that $J_{\lambda}, \tilde{J}_{\lambda}$ are multiples of each other, cf. (\ref{S12E2}). The second one follows from the first and (\ref{S12E7}). From (\ref{Eq.Vanish}) we conclude that the measure $\mathcal{J}_{a^N, b^M}$ is supported on $E_{\min(M,N)}^{\infty}$, where for $K \in [0,\infty]$ and $R \in [0,\infty]$ 
\begin{equation}\label{Eq.DefEventTrunc}
E_K^R = \{ \lambda \in \mathbb{Y}: \ell(\lambda) \leq K \mbox{ and } \lambda_1 \leq R \}.
\end{equation}
If we now set $K = \min (M,N)$ and $\ell_i = \lambda_i + (K-i) \cdot \theta$ for $i = 1,\dots, K$, we see from (\ref{S211E2}) with $M = N$, and (\ref{S211E4}) with $a = b$ that for any $R \in [0,\infty]$ we have
\begin{equation}\label{Eq.CaseIBetaEnsemble}
\begin{split}
&\mathcal{J}_{a^N, b^M}(\ell_1, \dots, \ell_K |E_K^R) = \const \times \prod_{1 \leq i < j \leq K} \frac{\Gamma(\ell_i - \ell_j + 1)\Gamma(\ell_i - \ell_j + \theta)}{\Gamma(\ell_i - \ell_j) \Gamma(\ell_i - \ell_j + 1 - \theta)} \cdot \prod_{i = 1}^K w(\ell_i;K), \\
& \mbox{ where } w(x;K) = \frac{(ab)^x \cdot \Gamma((M -K + 1)\theta +x)\Gamma((N -K + 1)\theta +x)}{\Gamma(x+1) \Gamma(x+ \theta)},
\end{split}
\end{equation}
for some positive constant that depends on $a,b,\theta, M,N,R$. We conclude that under $\mathcal{J}_{a^N, b^M}(\cdot | E_K^R)$ the $(\ell_1, \dots, \ell_K)$ form a discrete $\beta$-ensemble as in Definition \ref{Def.BetaEnsembles}.
%
%
\subsection{Case II: $\rho_1 = a^N$, $\rho_2 = b_{\beta}^M$}\label{Section2.3} From (\ref{Eq.PartFunctSpec}) with $a,b >0$ and $ab < 1$ (which ensures that the sum in the exponential converges absolutely) we have
\begin{equation}\label{Eq.CaseIINormalization}
 \sum_{\lambda \in \mathbb{Y}} J_{\lambda}(a^N;\theta) \tilde{J}_{\lambda}(b_{\beta}^M;\theta) =\exp\left(-NM\sum_{k=1}^{\infty}\frac{(-1)^k\cdot(ab)^k}{k}\right)=(1+ab)^{NM}.
\end{equation}
From (\ref{Eq.Vanish}) we have that the left side in (\ref{Eq.CaseIINormalization}) is in fact finite, and hence a polynomial in $ab$ (using the homogeneity of the Jack functions). As the rightmost side is also a polynomial in $ab$, and the two agree when $ab \in [0,1)$, we conclude that they agree everywhere. In other words, for any $a,b > 0$
\begin{equation*}
H_{\theta}(a^N; b_{\beta}^M) = \sum_{\lambda \in \mathbb{Y}} J_{\lambda}(a^N;\theta) \tilde{J}_{\lambda}(b_{\beta}^M;\theta) =(1+ab)^{NM}\in (0, \infty).
\end{equation*}
The above discussion shows that $\mathcal{J}_{a^N, b_{\beta}^M}$ is well-defined and supported on $E^M_N$ as in (\ref{Eq.DefEventTrunc}). 

If we now set $\ell_i=\lambda_i+(N-i)\cdot\theta$ for $i=1,\dots,N$, we see from (\ref{S211E2}) with $M = N$ and $K = N$, and (\ref{S212E1}) with $a = b$ and $K = N$ that we have
\begin{equation}\label{Eq.CaseIIBetaEnsemble}
\begin{split}
&\mathcal{J}_{a^N,b^M_\beta}(\ell_1, \dots, \ell_N ) = \const \times \prod_{1 \leq i < j \leq N} \frac{\Gamma(\ell_i - \ell_j + 1)\Gamma(\ell_i - \ell_j + \theta)}{\Gamma(\ell_i - \ell_j) \Gamma(\ell_i - \ell_j + 1 - \theta)} \cdot \prod_{i = 1}^N w(\ell_i;N),\\
& \mbox{ where } w(x;N) =\frac{(ab)^{x}}{\Gamma(x+1)\Gamma(M+N\theta-x+1-\theta)},
\end{split}    
\end{equation}
for some positive constant that depends on $a,b,\theta, M,N$. We conclude that under $\mathcal{J}_{a^N, b^M_\beta}(\cdot)$ the $(\ell_1, \dots, \ell_N)$ form a discrete $\beta$-ensemble as in Definition \ref{Def.BetaEnsembles} with $K$ replaced with $N$.

%
%
\subsection{Case III: $\rho_1 = a^N$, $\rho_2 = \tau_s$}\label{Section2.4}
From (\ref{Eq.PartFunctSpec}) we get
\begin{equation*}
H_{\theta}(a^N;\tau_s)=  \sum_{\lambda \in \mathbb{Y}} J_{\lambda}(a^N;\theta) \tilde{J}_{\lambda}(\tau_s;\theta) = \exp(\theta\cdot aNs)\in(0,\infty).
\end{equation*}
Consequently, the Jack measure $\mathcal{J}_{a^N,\tau_s}$ is well-defined and from (\ref{Eq.Vanish}) it is supported on $E_N^\infty$. If we now set $\ell_i=\lambda_i+(N-i)\cdot\theta$ for $i=1,\dots,N$, then we see from (\ref{S211E2}) with $M = N$ and $K = N$, and (\ref{S213E2}) with $K = N$ that for any $R\in [0,\infty]$ we have
\begin{equation}\label{Eq.CaseIIIBetaEnsemble}
\begin{split}
&\mathcal{J}_{a^N,\tau_s}(\ell_1, \dots, \ell_N | E_N^R) =
\const\times\prod_{1 \leq i < j \leq N} \frac{\Gamma(\ell_i - \ell_j + 1)\Gamma(\ell_i - \ell_j + \theta)}{\Gamma(\ell_i - \ell_j) \Gamma(\ell_i - \ell_j + 1 - \theta)}\cdot\prod_{i = 1}^N w(\ell_i;N),\\
& \mbox{ where } w(x;N) =\frac{(as\theta)^{x}}{\Gamma(x+1)},
\end{split}
\end{equation}
for some positive constant that depends on $a,s,\theta,N, R$. We conclude that under $\mathcal{J}_{a^N, \tau_s}(\cdot | E_N^R)$ the $(\ell_1, \dots, \ell_N)$ form a discrete $\beta$-ensemble as in Definition \ref{Def.BetaEnsembles} with $K$ replaced with $N$.

%
%
\subsection{Case IV: $\rho_1 = a_{\beta}^N$, $\rho_2 = b_{\beta}^M$}\label{Section2.5} 
From (\ref{Eq.PartFunctSpec}) we have for $a,b > 0$, $ab<1$
\begin{equation}\label{Eq.CaseIVNormalization_fml}
H_{\theta}(a^N_\beta; b^M_\beta) = \sum_{\lambda \in \mathbb{Y}} J_{\lambda}(a_{\beta}^N;\theta) \tilde{J}_{\lambda}(b_{\beta}^M;\theta) =  \exp \left( \sum_{k=1}^{\infty}\frac{\theta^{-1} N M (ab)^k }{k}\right) = (1-ab)^{\frac{-NM}{\theta}}\in(0,\infty),
\end{equation}
which shows that the Jack measure $\mathcal{J}_{a^N_\beta,b^M_\beta}$ is well-defined. In addition, from (\ref{Eq.Vanish}) we have that 
\begin{equation}\label{Eq.VanishCaseIV}
\mathcal{J}_{a^N_\beta,b^M_\beta}(\lambda_1 > \min(M,N)) = 0.
\end{equation}
We now fix $K \in \mathbb{N}$, and set $\ell_i = \lambda_i + (K-i) \cdot \theta$ for $i = 1,\dots, K$. From (\ref{S212E1}) with $a = b$ and (\ref{S212E2}) with $M = N$ we conclude
\begin{equation}\label{Eq.CaseIVBetaEnsemble}
\begin{split}
&\mathcal{J}_{a^N_\beta, b^M_\beta}(\ell_1, \dots, \ell_K |E_K^{\min(M,N)}) = \const \times \prod_{1 \leq i < j \leq K} \frac{\Gamma(\ell_i - \ell_j + 1)\Gamma(\ell_i - \ell_j + \theta)}{\Gamma(\ell_i - \ell_j) \Gamma(\ell_i- \ell_j + 1 - \theta)} \cdot \prod_{i = 1}^K w(\ell_i;K), \\
& \mbox{ where } w(x;K) = \frac{(ab)^x}{\Gamma(x+1)\Gamma(x+\theta)\Gamma(N+K\theta-x+1-\theta)\Gamma(M+K\theta-x+1-\theta)},
\end{split}
\end{equation}
for some positive constant that depends on $a,b,\theta, M,N,K$. We conclude that under the measure $\mathcal{J}_{a^N_\beta, b^M_\beta}(\cdot | E_K^{\min(M,N)})$ the $(\ell_1, \dots, \ell_K)$ form a discrete $\beta$-ensemble as in Definition \ref{Def.BetaEnsembles}.

%
%
\subsection{Case V: $\rho_1 = a_{\beta}^N$, $\rho_2 = \tau_s$}\label{Section2.6} From (\ref{Eq.PartFunctSpec}) we have that 
\begin{equation}\label{Eq.CaseVNormalization_fml}
H_{\theta}(a^N_\beta; \tau_s) = \sum_{\lambda \in \mathbb{Y}} J_{\lambda}(a_{\beta}^N;\theta) \tilde{J}_{\lambda}(\tau_s;\theta) =  \exp(aNs) \in(0,\infty),
\end{equation}
which shows that the Jack measure $\mathcal{J}_{a^N_\beta,\tau_s}$ is well-defined. In addition, from (\ref{Eq.Vanish}) we have that $\mathcal{J}_{a^N_{\beta}, \tau_s}(\lambda_1 > N) = 0$. We now fix $K \in \mathbb{N}$, and set $\ell_i = \lambda_i + (K-i) \cdot \theta$ for $i = 1,\dots, K$. From (\ref{S212E2}) with $M = N$, and (\ref{S213E2}) we conclude
\begin{equation}\label{Eq.CaseVBetaEnsemble}
\begin{split}
&\mathcal{J}_{a^N_\beta, \tau_s}(\ell_1, \dots, \ell_K |E_K^N) = \const \times \prod_{1 \leq i < j \leq K} \frac{\Gamma(\ell_i - \ell_j + 1)\Gamma(\ell_i - \ell_j + \theta)}{\Gamma(\ell_i - \ell_j) \Gamma(\ell_i - \ell_j + 1 - \theta)} \cdot \prod_{i = 1}^K w(\ell_i;K), \\
& \mbox{ where } w(x;K) = \frac{(as\theta)^x}{\Gamma(x+1)\Gamma(x+\theta)\Gamma(N+K\theta-x+1-\theta)},
\end{split}
\end{equation}
for some positive constant that depends on $a,s,\theta,N,K$. We conclude that under $\mathcal{J}_{a^N_\beta, \tau_s}(\cdot | E_K^N)$ the $(\ell_1, \dots, \ell_K)$ form a discrete $\beta$-ensemble as in Definition \ref{Def.BetaEnsembles}.

%
%
\subsection{Case VI: $\rho_1 = \tau_{s_1}$, $\rho_2 = \tau_{s_2}$}\label{Section2.7} From (\ref{Eq.PartFunctSpec}) we have that 
\begin{equation}\label{Eq.CaseVINormalization_fml}
H_{\theta}(\tau_{s_1}; \tau_{s_2}) = \sum_{\lambda \in \mathbb{Y}} J_{\lambda}(\tau_{s_1};\theta) \tilde{J}_{\lambda}(\tau_{s_2};\theta) =  \exp(\theta s_1 s_2) \in(0,\infty),
\end{equation}
which shows that the Jack measure $\mathcal{J}_{\tau_{s_1},\tau_{s_2}}$ is well-defined. We now fix $K \in \mathbb{N}$, and set $\ell_i = \lambda_i + (K-i) \cdot \theta$ for $i = 1,\dots, K$. From (\ref{S213E1}) with $s = s_1$, and (\ref{S213E2}) with $s = s_2$ we conclude for any $R \in [0,\infty]$ that
\begin{equation}\label{Eq.CaseVIBetaEnsemble}
\begin{split}
&\mathcal{J}_{\tau_{s_1}, \tau_{s_2}}(\ell_1, \dots, \ell_K |E_K^R) = \const \times \prod_{1 \leq i < j \leq K} \frac{\Gamma(\ell_i - \ell_j + 1)\Gamma(\ell_i - \ell_j + \theta)}{\Gamma(\ell_i - \ell_j) \Gamma(\ell_i - \ell_j + 1 - \theta)} \cdot \prod_{i = 1}^K w(\ell_i;K), \\
& \mbox{ where } w(x;K) = \frac{(s_1 s_2 \theta^2)^x}{\Gamma(x+1)\Gamma(x+\theta)},
\end{split}
\end{equation}
for some positive constant that depends on $s_1,s_2,\theta,K, R$. We conclude that under $\mathcal{J}_{\tau_{s_1}, \tau_{s_2}}(\cdot | E_K^R)$ the $(\ell_1, \dots, \ell_K)$ form a discrete $\beta$-ensemble as in Definition \ref{Def.BetaEnsembles}.

%
%
\section{Weak convergence of $\beta$-ensembles}\label{Section3} In Section \ref{Section3.1} we recall some of the results for discrete $\beta$-ensembles from \cite{DD22}. These results include a global law of large numbers, see Proposition \ref{Prop.GlobalLLN}, and an upper-tail bound, see Proposition \ref{Prop.UpperTail}. These statements will hold under certain technical assumptions on the weight functions $w(x;K)$ from Definition \ref{Def.BetaEnsembles}, which are summarized in Assumptions \ref{Ass.Finite} and \ref{Ass.Infinite}. In Sections \ref{Section3.2}-\ref{Section3.7} we verify that the Jack measures from Definition \ref{DefScale} satisfy these assumptions and deduce various consequences.

%
%
\subsection{General setup}\label{Section3.1} We consider sequences of discrete $\beta$-ensembles as in Definition \ref{Def.BetaEnsembles}, where the weight functions $w(x;K)$ and $R = R_K$ are appropriately scaled as $K \rightarrow \infty$. We will consider two types of scalings. The first, presented in Assumption \ref{Ass.Finite}, deals with the case when $R_K = K \mathsf{R} + O(1)$ for some fixed $\mathsf{R} > 0$. The second, presented in Assumption \ref{Ass.Infinite}, deals with the case when $R_K = \infty$. Throughout $\theta > 0$ is fixed, as usual.

\begin{assumption}\label{Ass.Finite} We assume that $\mathbb{P}_K$ is a sequence of measures as in Definition \ref{Def.BetaEnsembles} with $R =R_K$ that satisfies the following assumptions.
\begin{itemize}
    \item We assume that we have a real parameter $\mathsf{R} > 0$ and that $R_K \in \mathbb{Z}_{\geq 0}$ satisfy
    \begin{equation}\label{Eq.FinAss1}
        |R_K - K \mathsf{R}| \leq A_1 \mbox{, for some $A_1 > 0$}. 
    \end{equation}
    \item We assume that $w(x;K) = \exp\left( - K V_K(x/K)\right)$ on the interval $[0, R_K + (K-1)  \theta]$ for a function $V_K$ that is continuous on the interval $I_K = [0, R_K K^{-1} + (K-1) K^{-1} \theta]$.
    \item We assume that there is a continuous function $V$ on $ [0, \infty)$, such that 
    \begin{equation}\label{Eq.FinAss2}
    |V_K(s) - V(s)| \leq A_2  K^{-1} \log (K + 1) \mbox{ for } s \in I_K , \mbox{ and } |V(s)| \leq A_3 \mbox{ for } s\in I = [0, \mathsf{R} + \theta],
    \end{equation}
    for some constants $A_2, A_3 > 0$.
    \item We assume that $V$ is differentiable on $(0, \mathsf{R}+ \theta)$ and for some $A_4 > 0$ we have
    \begin{equation}\label{Eq.FinAss3}
    |V'(s)| \leq A_4 \left(1 + |\log (s) | + |\log (\mathsf{R} + \theta - s)| \right) \mbox{ for } s\in (0, \mathsf{R} + \theta).
    \end{equation}
\end{itemize}
\end{assumption}

\begin{assumption}\label{Ass.Infinite} We assume that $\mathbb{P}_K$ is a sequence of measures as in Definition \ref{Def.BetaEnsembles} with $R = \infty$ that satisfies the following assumptions.
\begin{itemize}
    \item We assume that $w(x; K) = \exp\left( -K V_K(x/K) \right)$
    for a continuous function $V_K$ on $[0, \infty)$. 
    \item We assume that $V_K \rightarrow V$ uniformly on compact subsets of $[0, \infty)$, where $V$ is a continuous function on $[0, \infty)$. More specifically, we assume that there is a sequence $r_K > 0$ with $\lim_{K \to \infty} K^{3/4} r_K = 0$ and an increasing function $F_1 : (0, \infty) \to (0, \infty)$, such that the following holds for all $u > 0$ and all $K \in \mathbb{N}$
    \begin{equation}\label{Eq.InfAss1}
        \sup_{s \in [0, u]} |V_K(s) - V(s)| \leq F_1(u)  r_K. 
    \end{equation}
    \item We assume that $V$ is differentiable on $(0, \infty)$ and that there exists $B_0 > 0$ and an increasing function $F_2 : (0, \infty) \to (0, \infty)$, such that for all $u > 0$ and $s \in (0, u]$ we have
    \begin{equation}\label{Eq.InfAss2}
        |V'(s)| \leq B_0 \left( F_2(u) + |\log (s)| \right). 
    \end{equation}
    \item We assume that there is a constant $\xi > 0$ such that for all $s \geq 0$ and $K \in \mathbb{N}$
    \begin{equation}\label{Eq.InfAss3}
        V_K(s) \geq (1 + \xi)  \theta \cdot \log (1 + s^2).
    \end{equation}
\end{itemize}
\end{assumption}

In the remainder of this section we state two results from \cite{DD22} that we will require for our arguments. 
\begin{proposition}\label{Prop.GlobalLLN} Suppose that $\mathbb{P}_K$ is a sequence of measures as in Definition \ref{Def.BetaEnsembles} that satisfy either Assumption \ref{Ass.Finite} or Assumption \ref{Ass.Infinite}. Let $\mu_K$ denote the random measures on $\mathbb{R}$, defined by
\begin{equation}\label{Eq.EmpMeas}
\mu_K = \frac{1}{K} \cdot \sum_{i = 1}^{K} \delta\left( \frac{\ell^K_i}{K} \right),
\end{equation}
where $(\ell_1^K, \dots, \ell_K^K)$ is distributed according to $\mathbb{P}_K$. Then, $\mu_K$ converges weakly in probability to a deterministic probability measure $\mu$ on $\mathbb{R}$ in the sense that for any bounded continuous function $g$ the sequence of random variables
$$\int_{\mathbb{R}} g(x) \mu_K(dx) - \int_{\mathbb{R}} g(x) \mu(dx)$$
converges to $0$ in probability. Moreover, the measure $\mu$ is compactly supported in $[0, \infty)$, is absolutely continuous with respect to Lebesgue measure and its density, which by a slight abuse of notation we denote by $\mu(x)$, is bounded by $\theta^{-1}$. 
\end{proposition}
\begin{proof}  Suppose first that $\mathbb{P}_K$ satisfies Assumption \ref{Ass.Finite}. Then, we see that $\mathbb{P}_K$ satisfies \cite[Assumptions 2.1 and 2.2]{DD22} with $N$ there replaced with $K$ here, $M_N$ replaced with $R_K$, $\mathsf{M} =\mathsf{R}$, $a_0 = A_0 = \mathsf{R}$, $q_N = 1/N$ and $p_N = N^{-1}\log (N + 1)$. The result now follows from \cite[Proposition 4.1]{DD22}. Suppose instead that $\mathbb{P}_K$ satisfies Assumption \ref{Ass.Infinite}. Then, $\mathbb{P}_K$ satisfies the assumptions of \cite[Definition 1.1]{DD22} with $N$ there replaced with $K$ here. The result now follows from \cite[Theorem 1.4]{DD22}.
\end{proof}
\begin{remark}\label{Rem.Minimizer} The measure $\mu$ in Proposition \ref{Prop.GlobalLLN} is sometimes called the {\em equilibrium measure}. We mention that in \cite{DD22} the density of $\mu$ in Proposition \ref{Prop.GlobalLLN} is denoted by $\phi^{\theta, \mathsf{R} + \theta}_V$ and is described as the unique minimizer of a certain functional $I_V^{\theta}$. We will not use this description in the present paper, and refer the interested reader to \cite{DD22} for the details.
\end{remark}

\begin{proposition}\label{Prop.UpperTail} Suppose that $\mathbb{P}_K$ is a sequence of measures as in Definition \ref{Def.BetaEnsembles} that satisfy Assumption \ref{Ass.Infinite}. Let $K_0 \geq \max(2, \theta^{-1} \xi^{-1})$ be such that $K^{3/4} r_K \leq 1$ for all $K \geq K_0$ and fix any $A_0 > 0$. We can find constants $C_1, D_1 > 0$ that depend on $\theta, V, \xi, B_0, F_1, F_2$ from Assumption \ref{Ass.Infinite} and $A_0$, such that for $K \geq K_0$
\begin{equation}\label{Eq.UpperTail}
\mathbb{P}_K(\ell_1 \geq  K D_1) \leq C_1 e^{-K A_0}.
\end{equation}
\end{proposition}
\begin{proof} We see that $\mathbb{P}_K$ satisfies the conditions of \cite[Definition 3.3]{DD22} with $N$ there replaced with $K$ here. The existence of $C_1, D_1$ as in the statement of the proposition, and the fact that (\ref{Eq.UpperTail}) holds for all $K \geq K_0$ now follows from \cite[Proposition 3.4]{DD22} with $N = K$, $A = A_0$ and $R_0 = 1$.
\end{proof}

%
%
\subsection{Case I: $\rho_1 = a^N$, $\rho_2 = b^M$}\label{Section3.2} We assume the same conditions as in Case I of Definition \ref{DefScale}, and also that $M = M(N) \geq N$ for all $N$. Note that in particular this implies that $\cm \geq 1$. For this choice we proceed to rewrite the measure in (\ref{Eq.CaseIBetaEnsemble}) in a way that is suitable for the application of Propositions \ref{Prop.GlobalLLN} and \ref{Prop.UpperTail}. From (\ref{Eq.CaseIBetaEnsemble}) we have for any $A \in \mathbb{R}$ and $R \in [0, \infty]$ that
\begin{equation}\label{Eq.CaseIBetaEnsembleV2}
\begin{split}
&\mathcal{J}_{a^N, b^M}(\ell_1, \dots, \ell_N |E_N^R) = \frac{1}{Z_{N,A}} \times \prod_{1 \leq i < j \leq N} \frac{\Gamma(\ell_i - \ell_j + 1)\Gamma(\ell_i - \ell_j + \theta)}{\Gamma(\ell_i - \ell_j) \Gamma(\ell_i - \ell_j + 1 - \theta)} \cdot \prod_{i = 1}^N e^{-NV_N(\ell_i/N)}, \\
& \mbox{ where } V_N(x) = A+\frac{1}{N}\log \left(\frac{\Gamma(Nx+1)N^{(M-N + 1)\theta -1}}{\Gamma(Nx+(M-N+1)\theta)(ab)^{Nx}e^{(M-N+1)\theta - 1}} \right).
\end{split}
\end{equation}
Indeed, the difference between (\ref{Eq.CaseIBetaEnsemble}) and (\ref{Eq.CaseIBetaEnsembleV2}) comes from setting $K = N$ (since we assumed $M \geq N$) and multiplying each $w(\ell_i;N)$ by a deterministic constant $C$ that depends on $M, N, \theta$ and $A$. This modifies the normalization constant by multiplying it by $C^N$, resulting in the same probability measure. We proceed to verify that when $R = \infty$ the measures in (\ref{Eq.CaseIBetaEnsembleV2}) satisfy the conditions of Assumption \ref{Ass.Infinite}, where $K$ will be replaced with $N$ in this case. In the process we will specify the value of $A$, which will be a large enough constant that depends on $a,b,\theta,\cm, \cd$, that allows us to verify (\ref{Eq.InfAss3}) with $\xi = 1$.\\

We start with the following useful lemma.
\begin{lemma}\label{Lem.GammaApprox} Fix $\Delta > 0$. For each $x \geq 0$, $K \geq 1$ 
\begin{equation}\label{Eq.GA1}
\begin{split}
& \frac{1}{K}\log \left(\Gamma(Kx + \Delta) e^{Kx + \Delta} K^{-Kx-\Delta }\right) =   (x + \Delta/K + 1/K) \log(x + \Delta/K + 1/K) \\
&+ \frac{\log(K)}{K} - \frac{ c_1 \log(Kx + \Delta + 1)}{K} - \frac{\log(Kx + \Delta)}{K}, \\
\end{split}
\end{equation}
for some $c_1 \in [0,1]$ that depends on $x, K, \Delta$. In addition, for each $K \geq 1$, and $x \in [0, \Delta/K]$ 
\begin{equation}\label{Eq.GA2}
\begin{split}
&\frac{1}{K}\log \left(\Gamma( \Delta - Kx + 1) e^{\Delta - Kx + 1} K^{-\Delta + Kx - 1}\right) = (\Delta/K - x + 1/K)\log(\Delta/K - x + 1/K) \\
& + \frac{1}{K} - \frac{c_2 \log (\Delta - Kx + 1)}{K},
\end{split}
\end{equation}
for some $c_2 \in [0,1]$ that depends on $x, K, \Delta$.
\end{lemma}
\begin{proof} From \cite[Theorem 1]{LC07}, we have for all $x \geq 1$ that
$$\frac{x^{x - \gamma}}{e^{x - 1}} \leq \Gamma(x) \leq \frac{x^{x - 1/2}}{e^{x - 1}},$$
where $\gamma= 0.577215...$ is the Euler-Mascheroni constant. From the functional equation $\Gamma(z+1) = z \Gamma(z)$ and the intermediate value theorem we conclude
$$\Gamma(Kx + \Delta) = \frac{\Gamma(Kx + \Delta + 1)}{Kx + \Delta} = \frac{(Kx + \Delta + 1)^{Kx + \Delta +1 -c_1} e^{-Kx - \Delta}}{Kx + \Delta },$$
for some $c_1 \in [1/2, \gamma]$. Taking logarithms of the last equation, and dividing by $K$ gives (\ref{Eq.GA1}). We similarly have by the intermediate value theorem for some $c_2 \in [1/2, \gamma]$
$$\Gamma( \Delta - Kx + 1) = (\Delta - Kx + 1)^{\Delta - Kx + 1 - c_2} e^{-\Delta + Kx}.$$
Taking logarithms of the last equation, and dividing by $K$ gives (\ref{Eq.GA2}). 
\end{proof}

After applying (\ref{Eq.GA1}) twice to (\ref{Eq.CaseIBetaEnsembleV2}) we get for some $\cq_1, \cq_2 \in [0,1]$
\begin{equation}\label{Eq.CaseIEqVN}
\begin{split}
&V_N(x) = A - x \log (ab) + ( x + 2/N) \log (x + 2/N) - (x + [\Delta+1]/N)\log (x + [\Delta+1]/N ) \\
& -\frac{\cq_1 \log(Nx+2)}{N} -\frac{ \log(Nx+1)}{N} + \frac{\cq_2 \log(Nx+\Delta + 1)}{N} + \frac{ \log(Nx+\Delta)}{N},
\end{split}
\end{equation}
where we have set $\Delta = (M-N+1)\theta$. We observe that for any $U \geq 0$, $u,v \in [0, U]$ and $x \geq 1$
\begin{equation}\label{Eq.XlogX}
0 \leq (x+v)\log (x+u) - x \log (x) \leq U + U \log ( x+ U).
\end{equation}
Picking $U \geq 1$ to be sufficiently large so that $[\Delta+1]/N  \leq U$ (recall that $|\Delta/N -  (\cm -1) \theta | \leq  (1 + \cd) \theta /N$ by assumption), we conclude that for $x \geq 1$
\begin{equation}\label{Eq.CaseIIneqVN}
V_N(x) \geq A - x \log (ab) - U -  U \log (x + U) -\frac{ \log(Nx+2)}{N} -\frac{ \log(Nx+1)}{N}.
\end{equation}
On the other hand, using that $|\Delta/N -  (\cm -1) \theta| \leq  (1 + \cd)\theta/N$ and $\Delta/N \geq\theta/N$ (recall that $M \geq N$ by assumption), we see from (\ref{Eq.CaseIEqVN}) that for any $u > 0$ and $x \in [0,u]$ we have
\begin{equation}\label{Eq.CaseILimitVN}
V_N(x) = A - x \log (ab) + x \log (x) - (x + (\cm - 1) \theta) \log (x + (\cm - 1) \theta) + O\left( \frac{\log (N+1)}{N} \right),
\end{equation}
where the constant in the big $O$ notation depends on $u, \theta, \cm, \cd$. Equations (\ref{Eq.CaseIIneqVN}) and (\ref{Eq.CaseILimitVN}) imply that we can pick $A$ large enough, depending on $ab, \theta, \cm, \cd$, so that for all $N \geq 1$ and $x \geq 0$
\begin{equation}\label{Eq.CaseILowerBound}
V_N(x) \geq 2 \theta \log (1 + x^2).
\end{equation}
We mention that in deriving the last inequality we implicitly used that the leading term for large $x$ in the right side of (\ref{Eq.CaseIIneqVN}) is $-x \log (ab)$, which is linearly growing, since $ab < 1$ by assumption. We fix this choice of $A$ in the remainder and note that (\ref{Eq.CaseILowerBound}) implies the last point in Assumption \ref{Ass.Infinite} with $\xi = 1$. We proceed to verify the remaining first three points in Assumption \ref{Ass.Infinite} with 
\begin{equation}\label{Eq.CaseIVDef}
V(x)= A - x\log (ab) + x \log (x) - (x + (\cm -1)\theta)\log (x + (\cm - 1) \theta).  
\end{equation}

The continuity of $V_N$ on $[0,\infty)$ is an immediate consequence of the continuity of the Gamma function on $(0, \infty)$, verifying the first point in Assumption \ref{Ass.Infinite}. In addition, since $\cm \geq 1$, we see that $V$ is differentiable on $(0,\infty)$ and 
$$V'(x) = - \log (ab) + \log (x)  - \log (x + (\cm - 1) \theta).$$
Using that for some constant $C> 0$, depending on $\theta, \cm$, and all $x > 0$ we have
$$|\log (x + (\cm - 1) \theta )| \leq |\log (x)| + C,$$
we conclude that 
\begin{equation}\label{Eq.CaseIVDerBound}
|V'(x)| \leq -\log (ab) + 2 |\log (x)| + C.
\end{equation}
Equation (\ref{Eq.CaseIVDerBound}) shows that the the third point in Assumption \ref{Ass.Infinite} holds with $B_0 = 2$ and $F_2(u) = \left(C -\log (ab)\right)/2$. Lastly, from (\ref{Eq.CaseILimitVN}) we see that for any $n \in \mathbb{N}$ we can find an increasing sequence $A_n > 0$, such that
$$\sup_{x \in [0,n]} |V_N(x) - V(x)| \leq A_n \frac{\log (N+1)}{N}.$$
Consequently, the second point in Assumption \ref{Ass.Infinite} holds with $r_N = \frac{\log (N+1)}{N}$ and $F_1(u) = A_{\lceil u \rceil}$.\\
Based on the above work we have the the following consequence of Propositions \ref{Prop.GlobalLLN} and \ref{Prop.UpperTail}.
\begin{lemma}\label{Lem.CaseIWeakConv} Assume the same conditions as in Case I of Definition \ref{DefScale}, and also that $M = M(N) \geq N$ for all $N$. Let $(\ell_1^N, \dots, \ell^N_N)$ be distributed as in (\ref{Eq.CaseIBetaEnsemble}) or equivalently (\ref{Eq.CaseIBetaEnsembleV2}) with $R = \infty$, and set $\mu_N = (1/N)\sum_{i = 1}^N \delta(\ell^N_i/N)$. Then, $\mu_N$ converges weakly in probability to a deterministic measure $\mu^I$ on $[0,\infty)$. In addition, if we fix $A_0 > 0$, we can find constants $C^I_1, C^I_2, D^I_1 > 0$, depending on $a,b,\theta, \cm, \cd$ and $A_0$, such that if instead of $R = \infty$, we set $R = R_N \in [N D^I_1, \infty]$ in (\ref{Eq.CaseIBetaEnsemble}), then $\mu_N$ as above converges weakly in probability to the same $\mu^I$ and for all $N \geq 1$ 
\begin{equation}\label{Eq.CaseIUpperTail}
\mathcal{J}_{a^N, b^M}(\lambda_1 \geq ND^I_1 |E_N^{R}) \leq C_1^I e^{-NA_0}, \hspace{2mm} d_{\operatorname{TV}} \left(\mathcal{J}_{a^N, b^M}(\cdot |E_N^{R}), \mathcal{J}_{a^N, b^M}(\cdot |E_N^{\infty})\right) \leq C_2^I e^{-NA_0},
\end{equation}
where $d_{\operatorname{TV}}$ denotes the usual total variation distance. 
\end{lemma}
\begin{proof} From our work in the section above we see that $\mathcal{J}_{a^N, b^M}(\cdot |E_N^{\infty})$ satisfies the conditions of Assumption \ref{Ass.Infinite}. Consequently, the first statement of the lemma follows from Proposition \ref{Prop.GlobalLLN}. 

To address the second statement, note that from Proposition \ref{Prop.UpperTail} we can find $C_1, D_1 > 0$, depending on $a,b,\theta,\cm, \cd$, such that for all $N \geq 1$
$$\mathcal{J}_{a^N, b^M}(\ell_1 \geq N D_1 |E_N^{\infty}) \leq C_1 e^{-NA_0}.$$
We also note that we have the following inclusion of events for $R \geq N D_1$
$$E_N^{\infty}\setminus E^R_N = \{\lambda_1 > ND_1\} = \{\ell_1 > (N-1) \theta + ND_1 \} \subseteq \{\ell_1 \geq ND_1\}.$$
Combining the last two equations, we see that for some $C > 0$ and all $N \geq 1$, $R \geq N D_1$
$$ d_{\operatorname{TV}} \left(\mathcal{J}_{a^N, b^M}(\cdot |E_N^{R}), \mathcal{J}_{a^N, b^M}(\cdot |E_N^{\infty})\right) \leq C e^{-NA_0}.$$
The second statement of the lemma now follows from the first and the last equation, and (\ref{Eq.CaseIUpperTail}) holds with $D_1^I = D_1$, $C_1^I = C_1 + C$ and $C_2^I = C$.  
\end{proof}

We end this section with the following consequence of Proposition \ref{Prop.GlobalLLN}, which we will require in our arguments below. The proof is postponed to Section \ref{Section3.8}.
\begin{lemma}\label{Lem.UpperTailCaseI} Fix $a,b, \theta, \cm, \cd_1, A_0 > 0$ with $ab < 1$. There exist constants $C_2, D_2 > 0$, depending on all the previously listed parameters, such that if $A, B \in \mathbb{N}$ and $\cd_1 \geq |B - \cm A|$, we have
\begin{equation}\label{Eq.UpperTailAlpha}
\mathcal{J}_{a^A, b^B}(\lambda_1 \geq \min(A,B) D_2) \leq C_2 e^{-\min(A,B) A_0}.
\end{equation}
\end{lemma}

%
%
\subsection{Case II: $\rho_1 = a^N$, $\rho_2 = b_{\beta}^M$}\label{Section3.3} We assume the same conditions as in Case II of Definition \ref{DefScale}, and proceed to rewrite the measure in (\ref{Eq.CaseIIBetaEnsemble}) in a way that is suitable for the application of Proposition \ref{Prop.GlobalLLN}. From (\ref{Eq.CaseIIBetaEnsemble}) we have
\begin{equation}\label{Eq.CaseIIBetaEnsembleV2}
\begin{split}
&\mathcal{J}_{a^N, b_\beta^M}(\ell_1, \dots, \ell_N ) = \frac{1}{Z_{N}} \times \prod_{1 \leq i < j \leq N} \frac{\Gamma(\ell_i - \ell_j + 1)\Gamma(\ell_i - \ell_j + \theta)}{\Gamma(\ell_i - \ell_j) \Gamma(\ell_i - \ell_j + 1 - \theta)} \cdot \prod_{i = 1}^N e^{-NV_N(\ell_i/N)}, \\
& \mbox{ where } V_N(x) = \frac{1}{N}\log \left(\frac{\Gamma(Nx+1)\Gamma(M+N\theta-\theta-Nx+1)}{(ab)^{Nx}\cdot N^{M+N\theta-\theta + 2}e^{-M-N\theta+\theta-2}}\right).
\end{split}
\end{equation}
The difference between (\ref{Eq.CaseIIBetaEnsemble}) and (\ref{Eq.CaseIIBetaEnsembleV2}) comes from multiplying each $w(\ell_i;N)$ by a deterministic constant $C$ that depends on $M,N$ and $\theta$, and multiplying the normalization constant by $C^N$, which results in the same probability measure. We proceed to verify that that the measures in (\ref{Eq.CaseIIBetaEnsembleV2}) satisfy the conditions of Assumption \ref{Ass.Finite}, where $K$ will be replaced with $N$, and $R_K$ with $M = M(N)$.

From Definition \ref{DefScale} we have that (\ref{Eq.FinAss1}) holds with $\mathsf{R} = \mathsf{m}$ and $A_1 = \cd$. The continuity of $V_N$ on $I_N = [0, M  N^{-1} + (N-1)  N^{-1} \theta]$ follows from the continuity of the Gamma function on $(0,\infty)$. Define $V$ on $I = [0, \cm + \theta]$ through
\begin{equation}\label{Eq.DefVCaseII}
V(x)=x\log(x)+(\cm+\theta-x)\log (\cm+\theta-x)-x\log (ab),
\end{equation}
and extend $V$ continuously to $[0, \infty)$ by setting $V(x) = V(\cm + \theta)$ for $x \geq \cm + \theta$. Note that 
\begin{equation}\label{Eq.CaseIIVPrime}
V'(x) = \log(x) - \log(\cm + \theta - x) - \log(ab).
\end{equation}
Consequently, we can find $A_3 > 0, A_4 > 0$, depending on $a,b,\theta, \cm$ that satisfy the second inequality in (\ref{Eq.FinAss2}) and (\ref{Eq.FinAss3}). Summarizing the work so far, we see that to verify Assumption \ref{Ass.Finite} what remains is to show that we can find $A_2 > 0$, depending on $ \theta, \cm, \cd$, such that 
\begin{equation}\label{Eq.S33R1}
|V_N(x) - V(x)| \leq A_2 N^{-1} \log (N + 1) \mbox{ for } x \in I_N.
\end{equation}

From Lemma \ref{Lem.GammaApprox} we get for some $\cq_1, \cq_2\in[0,1]$ and all $x \in I_N$
\begin{equation}\label{Eq.CaseIIExp3}
\begin{split}
&V_N(x) = - x \log (ab) + ( x + 2/N) \log (x + 2/N)+(\Delta/N-x + 1/N)\log(\Delta/N - x + 1/N), \\
& + \frac{\log(N) + 1}{N} - \frac{\cq_1 \log (Nx + 2)}{N} - \frac{\log(Nx + 1)}{N} - \frac{\cq_2 \log (\Delta - Nx +1)}{N},
\end{split}
\end{equation}
where $\Delta = M + N\theta - \theta$. Combining (\ref{Eq.DefVCaseII}) and (\ref{Eq.CaseIIExp3}), we deduce (\ref{Eq.S33R1}) once we use that $|\Delta/N - \cm -\theta| \leq (\cd + \theta)/N$. Overall, we conclude that $\mathcal{J}_{a^N, b_{\beta}^M}$ satisfies the conditions of Assumption \ref{Ass.Finite}. Consequently, Proposition \ref{Prop.GlobalLLN} gives the following statement.
\begin{lemma}\label{Lem.CaseIIWeakConv} Assume the same conditions as in Case II of Definition \ref{DefScale}. Let $(\ell_1^N, \dots, \ell^N_N)$ be distributed as in (\ref{Eq.CaseIIBetaEnsemble}) or equivalently (\ref{Eq.CaseIIBetaEnsembleV2}), and set $\mu_N = (1/N)\sum_{i = 1}^N \delta(\ell^N_i/N)$. Then, $\mu_N$ converges weakly in probability to a deterministic measure $\mu^{II}$ on $[0,\infty)$. 
\end{lemma}

%
%
\subsection{Case III: $\rho_1 = a^N$, $\rho_2 = \tau_s$}\label{Section3.4} We assume the same conditions as in Case III of Definition \ref{DefScale}. In this case the measures $\mathcal{J}_{a^N,\tau_{Nt}}$ were already studied in \cite[Section 6.3]{DD22}, so we will be brief. From (\ref{Eq.CaseIIIBetaEnsemble}) we have for any $A \in \mathbb{R}$ and $R \in [0, \infty]$ that
\begin{equation}\label{Eq.CaseIIIBetaEnsembleV2}
\begin{split}
&\mathcal{J}_{a^N, \tau_{Nt}}(\ell_1, \dots, \ell_N |E_N^R) = \frac{1}{Z_{N,A}} \times \prod_{1 \leq i < j \leq N} \frac{\Gamma(\ell_i - \ell_j + 1)\Gamma(\ell_i - \ell_j + \theta)}{\Gamma(\ell_i - \ell_j) \Gamma(\ell_i - \ell_j + 1 - \theta)} \cdot \prod_{i = 1}^N e^{-NV_N(\ell_i/N)}, \\
& \mbox{ where } V_N(x)=A+\frac{1}{N}\log \left( \frac{\Gamma(Nx+1)}{(atN\theta)^{Nx}}\right).
\end{split}
\end{equation}
The difference between (\ref{Eq.CaseIIIBetaEnsemble}) and (\ref{Eq.CaseIIIBetaEnsembleV2}) is that we multiplied each $w(\ell_i;N)$ by the same constant $e^A$ and the normalization constant by $e^{NA}$, which results in the same measure. We note that (\ref{Eq.CaseIIIBetaEnsembleV2}) agrees with \cite[(6.27)]{DD22} upon replacing $at$ with $t$, and the work in \cite[Section 6.3.1]{DD22} shows that we can choose $A$ large enough, depending on $a, t, \theta$, so that $\mathcal{J}_{a^N, \tau_{Nt}}(\cdot|E_N^{\infty})$ satisfies the conditions of Assumption \ref{Ass.Infinite} with $K$ replaced with $N$. We thus obtain the following statement.

\begin{lemma}\label{Lem.CaseIIIWeakConv} Assume the same conditions as in Case III of Definition \ref{DefScale}. Let $(\ell_1^N, \dots, \ell^N_N)$ be distributed as in (\ref{Eq.CaseIIIBetaEnsemble}) or equivalently (\ref{Eq.CaseIIIBetaEnsembleV2}) with $R = \infty$, and set $\mu_N = (1/N)\sum_{i = 1}^N \delta(\ell^N_i/N)$. Then, $\mu_N$ converges weakly in probability to a deterministic measure $\mu^{III}$ on $[0,\infty)$. In addition, if we fix $A_0 > 0$, we can find constants $C_1^{III}, C_2^{III}, D_1^{III} > 0$, depending on $a,t,\theta$ and $A_0$, such that if instead of $R = \infty$, we set $R = R_N \in [N D_1^{III}, \infty]$ in (\ref{Eq.CaseIIIBetaEnsemble}), then $\mu_N$ as above converges weakly in probability to the same $\mu^{III}$ and for all $N \geq 1$ 
\begin{equation}\label{Eq.CaseIIIUpperTail}
\begin{split}
&\mathcal{J}_{a^N, \tau_{Nt}}(\lambda_1 \geq ND^{III}_1 |E_N^{R}) \leq C_1^{III} e^{-NA_0}, \hspace{2mm} \\
& d_{\operatorname{TV}} \left(\mathcal{J}_{a^N, \tau_{Nt}}(\cdot |E_N^{R}), \mathcal{J}_{a^N, \tau_{Nt}}(\cdot |E_N^{\infty})\right) \leq C_2^{III} e^{-NA_0} .
\end{split}
\end{equation}
\end{lemma}
\begin{proof} Since $\mathcal{J}_{a^N, \tau_{Nt}}(\cdot|E_N^{\infty})$ satisfies the conditions of Assumption \ref{Ass.Infinite}, the first part of the lemma follows from Proposition \ref{Prop.GlobalLLN}. To prove the second statement in the lemma we can verbatim repeat the argument in the proof of Lemma \ref{Lem.CaseIWeakConv}.
\end{proof}

%
%
\subsection{Case IV: $\rho_1 = a_{\beta}^N$, $\rho_2 = b_{\beta}^M$}\label{Section3.5} We assume the same conditions as in Case IV of Definition \ref{DefScale}. Let us briefly explain what we intend to do in this section. Our ultimate goal in this case is to study the measures $ \mathcal{J}_{a^N_\beta, b^M_\beta}$; however, these do not readily fit into the framework of discrete $\beta$-ensembles. In order to remedy this, we introduce a new sequence $\mathbb{P}^d_K$ of $\beta$-ensembles in (\ref{Eq.CaseIVBetaEnsembleV2}) below, which satisfies the property 
\begin{equation}\label{Eq.CaseIVSubseq}
\mathbb{P}^d_{Nd}(\ell_1, \dots, \ell_{Nd} ) = \mathcal{J}_{a^N_{\beta}, b^M_{\beta}}\left(\ell_1, \dots, \ell_{Nd} \vert E^{\min(M,N)}_{Nd} \right),
\end{equation}
where we recall that $E^R_K$ is as in (\ref{Eq.DefEventTrunc}). We will show that $\mathbb{P}^d_K$ satisfies the conditions of Assumption \ref{Ass.Finite}, which will allow us to conclude some statements for these measures. These statements will automatically hold for $\mathbb{P}^d_{Nd}$ as a subsequence, and hence for $\mathcal{J}_{a^N_\beta, b^M_\beta}(\cdot | E^{\min(M,N)}_{Nd})$ in view of (\ref{Eq.CaseIVSubseq}). By picking $d$ sufficiently large, we will be able to further show that $\mathcal{J}_{a^N_\beta, b^M_\beta}(\cdot | E^{\min(M,N)}_{Nd})$ is quite close to the measure $ \mathcal{J}_{a^N_\beta, b^M_\beta}$, and that will allow us to transfer properties of the former to the latter. We now turn to the details of the above outline, and start by specifying how large we need to eventually take $d$ in the following lemma. 
\begin{lemma}\label{Lem.CaseIVLengthBound} Assume the same conditions as in Case IV of Definition \ref{DefScale} and fix $A_0 > 0$. We can find $ C_2^{IV}, D_2^{IV} > 0$, depending on $a,b, \theta,\cm, \cd$ and $A_0$, such that for all $N \geq 1$ and integers $d \geq D_2^{IV}$
\begin{equation}\label{Eq.CaseIVTruncate}
d_{\operatorname{TV}} \left( \mathcal{J}_{a^N_\beta, b^M_\beta}(\cdot | E^{\min(M,N)}_{Nd} ),  \mathcal{J}_{a^N_\beta, b^M_\beta}(\cdot ) \right)  \leq C_2^{IV} \cdot e^{-\min(M,N) A_0}.
\end{equation}
\end{lemma}
\begin{proof} Let us denote by $\mathcal{J}^{\theta}_{\rho_1, \rho_2}$ the Jack measures in Definition \ref{DefJM}, to emphasize the dependence on $\theta$. From (\ref{S12E2}), (\ref{S12E7}) and Remark \ref{Rem.Commute} we get
\begin{equation*}
\begin{split}
\mathcal{J}^{\theta}_{a^N_\beta, b^M_\beta}(\lambda) \propto J_{\lambda}(a^N_\beta;\theta)\tilde{J}_{\lambda}(b^M_\beta;\theta)
=\tilde{J}_{\lambda'}(a^N;\theta^{-1})J_{\lambda'}(b^M;\theta^{-1}) \propto \mathcal{J}^{\theta^{-1}}_{a^N, b^M}(\lambda').
\end{split}
\end{equation*}
In other words, if $\lambda$ is distributed according to $\mathcal{J}^{\theta}_{a^N_\beta, b^M_\beta}$, then $\lambda'$ is distributed according to $\mathcal{J}^{\theta^{-1}}_{a^N, b^M}$. From the last identity and Lemma \ref{Lem.UpperTailCaseI} we can find $C_2, D_2$, depending on $a,b, \theta, \cm, \cd$ and $A_0$, such that for any integer $d \geq D_2$ and $N \geq 1$
\begin{equation}\label{S35IE1}
1- \mathcal{J}^{\theta}_{a^N_\beta, b^M_\beta}(E^{\min(M,N)}_{Nd}) = \mathcal{J}^{\theta}_{a^N_\beta, b^M_\beta} ( \ell(\lambda) > Nd ) \leq \mathcal{J}^{\theta^{-1}}_{a^N, b^M} ( \lambda_1 \geq ND_2 ) \leq C_2 e^{-\min(M,N)A_0}.
\end{equation}
We mention that in the first equality we used that $\mathcal{J}^{\theta}_{a^N_\beta, b^M_\beta}(E^{\min(M,N)}_{\infty}) = 1$, see (\ref{Eq.VanishCaseIV}). The statement of the lemma now follows from (\ref{S35IE1}).
\end{proof}
Fix any $d \in \mathbb{N}$ and define the sequences
$$N_K = \lceil K/d \rceil, \mbox{ and } M_K = M(\lceil K/d \rceil),$$
where $M(N)$ is still as in Case IV of Definition \ref{DefScale}. With this data we define the measures $\mathbb{P}^d_K$ on $\mathbb{W}^{\theta, R_K}_{K}$ as in (\ref{Eq.State}) with $R_K = \min(M_K, N_K)$ through
\begin{equation}\label{Eq.CaseIVBetaEnsembleV2}
\begin{split}
&\mathbb{P}^d_K(\ell_1, \dots, \ell_K) = \frac{1}{Z_K} \times \prod_{1 \leq i < j \leq K} \frac{\Gamma(\ell_i - \ell_j + 1)\Gamma(\ell_i - \ell_j + \theta)}{\Gamma(\ell_i - \ell_j) \Gamma(\ell_i - \ell_j + 1 - \theta)} \cdot \prod_{i = 1}^K e^{-KV_K(\ell_i/K)}, \\
& \mbox{ where } V_K(x) = \frac{1}{K}\log \left(\frac{\Gamma(Kx+1)\Gamma(Kx+\theta)\Gamma(\Delta_K^N-Kx+1)\Gamma(\Delta^M_K-Kx+1)}{(ab)^{Kx}\cdot K^{\Delta_K^M +\Delta_K^N + 3 + \theta}e^{- \Delta_K^N - \Delta_K^M -3 - \theta }}\right),
\end{split}
\end{equation}
and we have set $\Delta^N_K=N_K+(K-1)\theta$ and $\Delta^M_K =M_K+(K-1)\theta$. The key observation is that the sequence $\mathbb{P}^d_K$ satisfies (\ref{Eq.CaseIVSubseq}), which follows from the way we defined our sequences $M_K, N_K, R_K$, and the fact that (\ref{Eq.CaseIVBetaEnsembleV2}) with $K = Nd$ differs from (\ref{Eq.CaseIVBetaEnsemble}) only in that we multiplied each $w(\ell_i;K)$ by the same constant, which results in the same measure. In the remainder of this section, we proceed to check that $\mathbb{P}^d_K$ satisfies the conditions of Assumption \ref{Ass.Finite}, and apply Proposition \ref{Prop.GlobalLLN}.

From Definition \ref{DefScale} we have that (\ref{Eq.FinAss1}) holds with $\cR=d^{-1} \min(1,\cm)$ and $A_1=\cd + 1$. The continuity of $V_K(x)$ on $I_K=[0,\min(M_K,N_K) K^{-1}+(K-1) K^{-1}\theta]$ follows from the continuity of the Gamma function on $(0,\infty)$. Define $V$ on $I=[0,\cR+\theta]$ through
\begin{equation}\label{Eq.DefVCaseIV}
V(x)= -x\log(ab)+2x\log(x)+(d^{-1}+\theta-x)\log(d^{-1}+\theta-x)+(d^{-1}\cm+\theta-x)\log(d^{-1}\cm +\theta-x),
\end{equation}
and extend $V$ continuously to $[0,\infty)$ by setting $V(x) = V(\cR + \theta)$ for $x \geq \cR + \theta$. Note that
$$V'(x)=-\log(ab)+2\log(x)-\log(d^{-1}+\theta-x)- \log(d^{-1}\cm+\theta-x).$$
Consequently, we can find $A_3 > 0, A_4 > 0$, depending on $a,b,\theta, \cm, d$ that satisfy the second inequality in (\ref{Eq.FinAss2}) and (\ref{Eq.FinAss3}). Summarizing the work so far, we see that to verify Assumption \ref{Ass.Finite} what remains is to show that we can find $A_2 > 0$, depending on $\theta, \cm, \cd, d$, such that 
\begin{equation}\label{Eq.S35R1}
|V_K(x) - V(x)| \leq A_2  K^{-1} \log (K + 1) \mbox{ for } x \in I_K.
\end{equation}
From Lemma \ref{Lem.GammaApprox} we get for some $\cq_1, \dots, \cq_4 \in [0,1]$
\begin{equation}\label{Eq.CaseIVEqVK}
\begin{split}
&V_K(x) = -x\log(ab)+ (x+2/K)\log (x + 2/K) + (x+ [1 +\theta]/K)\log(x + [1 + \theta]/K) \\
&- ([\Delta_K^N+1]/K - x ) \log([\Delta_K^N + 1]/K-x) - ([\Delta_K^M + 1]/K - x) \log([\Delta_K^M + 1]/K-x) \\
& +\frac{2\log(K) + 2}{K} - \frac{\cq_1 \log(Kx + 2)}{K} - \frac{ \log(Kx + 1)}{K} - \frac{\cq_2 \log(Kx + 1 + \theta)}{K} - \frac{ \log(Kx + \theta)}{K} \\
&- \frac{\cq_3 \log(\Delta_K^N - Kx + 1)}{K} - \frac{\cq_4 \log(\Delta_K^M - Kx + 1)}{K}.
\end{split}
\end{equation}
Using (\ref{Eq.DefVCaseIV}) and (\ref{Eq.CaseIVEqVK}), we deduce (\ref{Eq.S35R1}) once we use that $|\Delta^N_K/K - 1/d -\theta| \leq (1 + \theta)/K$ and $|\Delta^M_K/K - \cm/d -\theta| \leq (\cd + 1 + \theta)/K$. Overall, we conclude that $\mathbb{P}^d_K$ satisfies the conditions of Assumption \ref{Ass.Finite}, and so from Proposition \ref{Prop.GlobalLLN} we obtain the following statement.
\begin{lemma}\label{Lem.CaseIVWeakConv} Assume the same conditions as in Case IV of Definition \ref{DefScale} and fix $d \in \mathbb{N}$. Let $(\ell_1^K, \dots, \ell^K_K)$ be distributed according to $\mathbb{P}^d_K$ as in (\ref{Eq.CaseIVBetaEnsembleV2}) and set $\mu_K = (1/K)\sum_{i = 1}^K \delta(\ell^K_i/K)$. Then, the sequence $\mu_K$ converges weakly in probability to a deterministic measure $\mu^{IV}_{d}$ on $[0,\infty)$. 
\end{lemma}

%
%
\subsection{Case V: $\rho_1 = a_{\beta}^N$, $\rho_2 = \tau_s$}\label{Section3.6} We assume the same conditions as in Case V of Definition \ref{DefScale}. Our strategy in this case is quite similar to Case IV above. In particular, we introduce a sequence $\mathbb{P}^d_K$ of $\beta$-ensembles in (\ref{Eq.CaseVBetaEnsembleV2}) below, which satisfies the property 
\begin{equation}\label{Eq.CaseVSubseq}
\mathbb{P}^d_{Nd}(\ell_1, \dots, \ell_{Nd} ) = \mathcal{J}_{a^N_{\beta}, \tau_{Nt}}\left(\ell_1, \dots, \ell_{Nd} \vert E^{N}_{Nd} \right).
\end{equation}
We will show that $\mathbb{P}^d_K$ satisfies the conditions of Assumption \ref{Ass.Finite}, which will allow us to conclude some statements for these measures. These statements will automatically hold for $\mathbb{P}^d_{Nd}$ as a subsequence, and hence for $\mathcal{J}_{a^N_\beta, \tau_{Nt}}(\cdot | E^{N}_{Nd})$ in view of (\ref{Eq.CaseVSubseq}). By picking $d$ sufficiently large, we will be able to further show that $\mathcal{J}_{a^N_\beta, \tau_{Nt}}(\cdot | E^{N}_{Nd})$ is quite close to the measure $ \mathcal{J}_{a^N_\beta, \tau_{Nt}}$, and that will allow us to transfer properties of the former to the latter. We now turn to the details of this outline, starting by specifying how large we need to eventually take $d$ in the following lemma.
\begin{lemma}\label{Lem.CaseVLengthBound} Assume the same conditions as in Case V of Definition \ref{DefScale} and fix $A_0 > 0$. We can find $C_2^{V}, D_2^V > 0$, depending on $a,t, \theta$ and $A_0$, such that for all $N \geq 1$ and integers $d \geq D_2^V$
\begin{equation}\label{Eq.CaseVTruncate}
d_{\operatorname{TV}} \left( \mathcal{J}_{a^N_\beta, \tau_{Nt}}(\cdot | E^{N}_{Nd} ),  \mathcal{J}_{a^N_\beta, \tau_{Nt}}(\cdot ) \right)  \leq C_2^{V} e^{- N A_0}.
\end{equation}
\end{lemma}
\begin{proof} 
Let us denote by $\mathcal{J}^{\theta}_{\rho_1, \rho_2}$ the Jack measures in Definition \ref{DefJM}, to emphasize the dependence on $\theta$. From (\ref{S12E2}), (\ref{S12E7}), (\ref{S213E2}) and Remark \ref{Rem.Commute} we get
\begin{equation}\label{S36E1}
\begin{split}
\mathcal{J}^{\theta}_{a^N_\beta, \tau_{Nt}}(\lambda) \propto J_{\lambda}(a^N_\beta;\theta)\tilde{J}_{\lambda}(\tau_{Nt};\theta)
=\tilde{J}_{\lambda'}(a^N;\theta^{-1})J_{\lambda'}(\tau_{Nt\theta};\theta^{-1}) \propto \mathcal{J}^{\theta^{-1}}_{a^N, \tau_{Nt\theta}}(\lambda').
\end{split}
\end{equation}
In other words, if $\lambda$ is distributed according to $\mathcal{J}^{\theta}_{a^N_\beta, \tau_{Nt}}$, then $\lambda'$ is distributed according to $\mathcal{J}^{\theta^{-1}}_{a^N, \tau_{Nt\theta}}$. From the last identity and Lemma \ref{Lem.CaseIIIWeakConv} we can find $C_1^{III}, D_1^{III}$, depending on $a,t,\theta$ and $A_0$, such that for any integer $d \geq D_1^{III}$ and $N \geq 1$
\begin{equation}\label{S36IE1}
1- \mathcal{J}^{\theta}_{a^N_\beta, \tau_{Nt}}(E^{N}_{Nd}) = \mathcal{J}^{\theta^{-1}}_{a^N, \tau_{Nt\theta}} ( \lambda_1 > Nd )\leq \mathcal{J}^{\theta^{-1}}_{a^N, \tau_{Nt\theta}} ( \lambda_1 \geq ND_1^{III} ) \leq C_1^{III} e^{-NA_0}.    
\end{equation}
We mention that in the first equality we implicitly used that $\mathcal{J}^{\theta}_{a^N_\beta, \tau_{Nt}}(E^{N}_{\infty}) = 1$ and when we applied Lemma \ref{Lem.CaseIIIWeakConv} we did so with $R = \infty$, $\theta$ replaced with $\theta^{-1}$, and $t$ replaced with $t \theta$. The statement of the lemma now follows from (\ref{S36IE1}).
\end{proof}

Fix $d \in \mathbb{N}$ and set $N_K = \lceil K/d \rceil$. With this data we define the measures $\mathbb{P}^d_K$ on $\mathbb{W}^{\theta, R_K}_{K}$ as in (\ref{Eq.State}) with $R_K = N_K$ through
\begin{equation}\label{Eq.CaseVBetaEnsembleV2}
\begin{split}
&\mathbb{P}^d_K(\ell_1, \dots, \ell_K) = \frac{1}{Z_{K}} \times \prod_{1 \leq i < j \leq K} \frac{\Gamma(\ell_i - \ell_j + 1)\Gamma(\ell_i - \ell_j + \theta)}{\Gamma(\ell_i - \ell_j) \Gamma(\ell_i - \ell_j + 1 - \theta)} \cdot \prod_{i = 1}^K e^{-KV_K(\ell_i/K)}, \\
& \mbox{ where } V_K(x) =\frac{1}{K}\log\left(\frac{\Gamma(Kx+1) \Gamma(Kx  + \theta) \Gamma(\Delta^N_K - K x + 1)}{(eat \theta/d)^{Kx} K^{\Delta_K^N + Kx +2 + \theta} e^{- \Delta^N_K - Kx - 2 - \theta} }  \right),
\end{split}
\end{equation}
and we have set $\Delta^N_K=N_K+(K-1)\theta$. The key observation is that the sequence $\mathbb{P}^d_K$ satisfies (\ref{Eq.CaseVSubseq}), which follows from the way we defined our sequences $N_K,R_K$, and the fact that (\ref{Eq.CaseVBetaEnsembleV2}) with $K = Nd$ differs from (\ref{Eq.CaseVBetaEnsemble}) only in that we multiplied each $w(\ell_i;K)$ by the same constant, which results in the same measure. In the remainder of this section we proceed to check that $\mathbb{P}^d_K$ satisfies the conditions of Assumption \ref{Ass.Finite}, and apply Proposition \ref{Prop.GlobalLLN}.

From Definition \ref{DefScale} we have that (\ref{Eq.FinAss1}) holds with $\cR = d^{-1}$ for $A_1 = 1$. The continuity of $V_K$ on $I_K = [0, N_K K^{-1} + (K-1) K^{-1}  \theta]$ follows from the continuity of the Gamma function on $(0,\infty)$. Define $V$ on $I = [0, \cR + \theta] = [0, d^{-1} + \theta]$ through
\begin{equation}\label{Eq.DefVCaseV}
V(x)=2x\log(x)+(d^{-1}+\theta-x)\log (d^{-1}+\theta-x)-x\log (eat \theta/d),
\end{equation}
and extend $V$ continuously to $[0, \infty)$ by setting $V(x) = V(\cR + \theta)$ for $x \geq \cR + \theta$. Note that 
$$V'(x) = 2\log(x)-\log(d^{-1}+\theta-x)-\log(at\theta/d).$$
Consequently, we can find $A_3 > 0, A_4 > 0$, depending on $a,t,\theta, d,$ that satisfy the second inequality in (\ref{Eq.FinAss2}) and (\ref{Eq.FinAss3}). Summarizing the work so far, we see that to verify Assumption \ref{Ass.Finite} what remains is to show that we can find $A_2 > 0$, depending on $ \theta, d$, such that 
\begin{equation}\label{Eq.S36R1}
|V_K(x) - V(x)| \leq A_2 K^{-1} \log (K + 1) \mbox{ for } x \in I_K.
\end{equation}
From Lemma \ref{Lem.GammaApprox} we get for some $\cq_1, \cq_2, \cq_3 \in [0,1]$ and all $x \in I_K$
\begin{equation}\label{Eq.CaseVExp4}
\begin{split}
&V_K(x)=   ([\Delta^N_K + 1]/K - x)\log([\Delta^N_K + 1]/K - x)   + (x + [1 + \theta]/K) \log(x + [1+\theta]/K) \\
& +  (x + 2/K)\log(x + 2/K)- x \log(eat\theta /d) + \frac{2\log(K) + 1}{K}  - \frac{\cq_1 \log(Kx + 2)}{K}  \\
&  - \frac{ \log(Kx + 1)}{K} - \frac{\cq_2 \log(Kx + 1 + \theta)}{K} - \frac{ \log(Kx + \theta)}{K} - \frac{\cq_3 \log(\Delta_K^N - Kx + 1)}{K}.
\end{split}
\end{equation}
Using (\ref{Eq.DefVCaseV}) and (\ref{Eq.CaseVExp4}), we deduce (\ref{Eq.S36R1}) once we use that $|\Delta^N_K/ K - d^{-1} - \theta| \leq (1 + \theta)/K$. Overall, we conclude that $\mathbb{P}^d_K$ satisfies the conditions of Assumption \ref{Ass.Finite}, and so from Proposition \ref{Prop.GlobalLLN} we obtain the following statement.

\begin{lemma}\label{Lem.CaseVWeakConv} Assume the same conditions as in Case V of Definition \ref{DefScale} and fix $d \in \mathbb{N}$. Let $(\ell_1^K, \dots, \ell^K_K)$ be distributed according to $\mathbb{P}^d_K$ as in (\ref{Eq.CaseVBetaEnsembleV2}) and set $\mu_K = (1/K)\sum_{i = 1}^K \delta(\ell^K_i/K)$. Then, $\mu_K$ converges weakly in probability to a deterministic measure $\mu^{V}_d$ on $[0,\infty)$.
\end{lemma}

%
%
\subsection{Case VI: $\rho_1 = \tau_{s_1}$, $\rho_2 = \tau_{s_2}$}\label{Section3.7} We assume the same conditions as in Case VI of Definition \ref{DefScale}. Our strategy in this case is similar to Cases IV and V above. In particular, we introduce a sequence $\mathbb{P}^{d}_K$ of $\beta$-ensembles in (\ref{Eq.CaseVIBetaEnsembleV2}) below, which satisfies the property 
\begin{equation}\label{Eq.CaseVISubseq}
\mathbb{P}^{d}_{Nd}(\ell_1, \dots, \ell_{Nd} ) = \mathcal{J}_{\tau_{Nt_1}, \tau_{Nt_2}}\left(\ell_1, \dots, \ell_{Nd} \vert E^{Nd}_{Nd} \right).
\end{equation}
We will show that $\mathbb{P}^{d}_K$ satisfies the conditions of Assumption \ref{Ass.Finite}, which will allow us to conclude some statements for these measures. These statements will automatically hold for $\mathbb{P}^{d}_{Nd}$ as a subsequence, and hence for $\mathcal{J}_{\tau_{Nt_1}, \tau_{Nt_2}}(\cdot | E^{Nd}_{Nd})$ in view of (\ref{Eq.CaseVISubseq}). By picking $d$ sufficiently large, we will be able to further show that $\mathcal{J}_{\tau_{Nt_1}, \tau_{Nt_2}}(\cdot | E^{Nd}_{Nd})$ is quite close to the measure $ \mathcal{J}_{\tau_{Nt_1}, \tau_{Nt_2}}$, and that will allow us to transfer properties of the former to the latter. We now turn to the details of this outline, starting by specifying how large we need to eventually take $d$ in the following lemma, whose proof is given in Section \ref{Section3.9} below.

\begin{lemma}\label{Lem.CaseVILengthBound} Assume the same conditions as in Case VI of Definition \ref{DefScale} and fix $A_0 > 0$. We can find $D_2^{VI}$, depending on $t_1, t_2, \theta$ and $A_0$, such that the following holds for all integers $d \geq D_2^{VI}$. For some $C_1^{VI}, C_2^{VI} > 0$, depending on $t_1, t_2, \theta,A_0$ and $d$, we have for $K \geq 1$
\begin{equation}\label{Eq.CaseVIUpperTail}
\mathcal{J}_{\tau_{Kt_1/d}, \tau_{Kt_2/d}}\left(\lambda_1 = K | E^{K}_{K} \right) \leq C_1^{VI}  e^{-KA_0},
\end{equation}
and also for $N \geq 1$
\begin{equation}\label{Eq.CaseVITruncateLen}
d_{\operatorname{TV}} \left( \mathcal{J}_{\tau_{Nt_1}, \tau_{Nt_2}}(\cdot | E^{Nd}_{Nd} ),  \mathcal{J}_{\tau_{Nt_1}, \tau_{Nt_2}}(\cdot ) \right)  \leq C^{VI}_2 e^{- NA_0}.
\end{equation}
\end{lemma}

Fix $d \in \mathbb{N}$. We define the measures $\mathbb{P}^{d}_K$ on $\mathbb{W}^{\theta, R_K}_{K}$ as in (\ref{Eq.State}) with $R_K = K$ through
\begin{equation}\label{Eq.CaseVIBetaEnsembleV2}
\begin{split}
&\mathbb{P}^{d}_K(\ell_1, \dots, \ell_K) = \frac{1}{Z_{K}} \times \prod_{1 \leq i < j \leq K} \frac{\Gamma(\ell_i - \ell_j + 1)\Gamma(\ell_i - \ell_j + \theta)}{\Gamma(\ell_i - \ell_j) \Gamma(\ell_i - \ell_j + 1 - \theta)} \cdot \prod_{i = 1}^K e^{-KV_K(\ell_i/K)}, \\
& \mbox{ where } V_K(x) =\frac{1}{K}\log\left(\frac{\Gamma(Kx+1) \Gamma(Kx  + \theta) }{(e^2t_1t_2 \theta^2/d^2)^{Kx} K^{2Kx+ \theta + 1}e^{-2Kx - \theta - 1} }  \right).
\end{split}
\end{equation}
The key observation is that the sequence $\mathbb{P}^{d}_K$ satisfies (\ref{Eq.CaseVISubseq}), which follows from the way we defined $R_K$, and the fact that (\ref{Eq.CaseVIBetaEnsembleV2}) with $K = Nd$ differs from (\ref{Eq.CaseVIBetaEnsemble}) only in that we multiplied each $w(\ell_i;K)$ by the same constant, which results in the same measure. In the remainder of this section we proceed to check that $\mathbb{P}^{d}_K$ satisfies the conditions of Assumption \ref{Ass.Finite}, and apply Proposition \ref{Prop.GlobalLLN}.

From Definition \ref{DefScale} we have that (\ref{Eq.FinAss1}) holds with $\cR = 1$ for $A_1 =  1$. The continuity of $V_K$ on $I_K = [0, R_K K^{-1} + (K-1) K^{-1}  \theta]$ follows from the continuity of the Gamma function on $(0,\infty)$. Define $V$ on $I = [0, \cR + \theta] = [0, 1 + \theta]$ through
\begin{equation}\label{Eq.DefVCaseVI}
V(x)= 2x \log (x) - x \log(e^2t_1 t_2 \theta^2/d^2).
\end{equation}
and extend $V$ continuously to $[0, \infty)$ by setting $V(x) = V(\cR + \theta)$ for $x \geq \cR + \theta$. Note that
\begin{equation}\label{CaseVIpm}
V'(x) = 2\log(x)-\log(t_1t_2 \theta^2/d^2).    
\end{equation}
Consequently, we can find $A_3 > 0, A_4 > 0$, depending on $t_1, t_2, \theta, d$ that satisfy the second inequality in (\ref{Eq.FinAss2}) and (\ref{Eq.FinAss3}). Summarizing the work so far, we see that to verify Assumption \ref{Ass.Finite} what remains is to show that we can find $A_2 > 0$, depending on $ \theta, d$, such that 
\begin{equation}\label{Eq.S37R1}
|V_K(x) - V(x)| \leq A_2 K^{-1} \log (K + 1) \mbox{ for } x \in I_K.
\end{equation}
From Lemma \ref{Lem.GammaApprox} we get for some $\cq_1, \cq_2 \in [0,1]$ and all $x \in I_K$
\begin{equation}\label{Eq.CaseVIExp4}
\begin{split}
&V_K(x)=  (x + 2/K)\log(x + 2/K)  + (x + [1 + \theta]/K) \log(x + [1+\theta]/K) \\
& - x \log(e^2t_1t_2 \theta^2/d^2) + \frac{2\log(K)}{K}  - \frac{\cq_1 \log(Kx + 2)}{K}  \\
&  - \frac{ \log(Kx + 1)}{K} - \frac{\cq_2 \log(Kx + 1 + \theta)}{K} - \frac{ \log(Kx + \theta)}{K}.
\end{split}
\end{equation}
Using (\ref{Eq.DefVCaseVI}) and (\ref{Eq.CaseVIExp4}), we deduce (\ref{Eq.S37R1}). Overall, we conclude that $\mathbb{P}^{d}_K$ satisfies the conditions of Assumption \ref{Ass.Finite}, and so from Proposition \ref{Prop.GlobalLLN} we obtain the following statement.

\begin{lemma}\label{Lem.CaseVIWeakConv} Assume the same conditions as in Case VI of Definition \ref{DefScale} and fix $d \in \mathbb{N}$. Let $(\ell_1^K, \dots, \ell^K_K)$ be distributed according to $\mathbb{P}^{d}_K$ as in (\ref{Eq.CaseVIBetaEnsembleV2}) and set $\mu_K = (1/K)\sum_{i = 1}^K \delta(\ell^K_i/K)$. Then, $\mu_K$ converges weakly in probability to a deterministic measure $\mu^{VI}_{d}$ on $[0,\infty)$.
\end{lemma}

%
%
\subsection{Proof of Lemma \ref{Lem.UpperTailCaseI}} \label{Section3.8} We assume the same notation as in the statement of the lemma and directly proceed with the proof. Let $N_0 \in \mathbb{N}$ be sufficiently large, depending on $\theta$, such that for $m \geq N_0$, we have 
\begin{equation}\label{Eq.SmallCases}
m \geq \max(2, \theta^{-1}), \hspace{2mm} m^{3/4} \cdot \frac{\log(m+1)}{m} \leq 1.
\end{equation}
If we pick $C_2 \geq e^{N_0 A_0}$, then (\ref{Eq.UpperTailAlpha}) holds trivially for any $A,B \in \mathbb{N}$ with $\min(A,B) < N_0$. We may thus assume without loss of generality that $A, B \geq N_0$, which we do in what follows. For clarity we split the proof into four cases, depending on whether $\cm \geq 1$ or $\cm < 1$, and $B \geq A$ or $B < A$.

{\raggedleft \bf Case 1.} Suppose first that $\cm \geq 1$ and $B \geq A$. Consider the sequence $M = M(N)$, such that $M(N) = \lceil \cm N \rceil$ for $N \neq A$, and $M(A) = B$, and the corresponding measures $\mathcal{J}_{a^N, b^M}$. The latter sequence satisfies the conditions of Definition \ref{DefScale} with $\cd = \cd_2$, where $\cd_2 = \max(\cd_1, 1)$, and also $M(N) \geq N$. From our work in Section \ref{Section3.2} we see that $\mathcal{J}_{a^N, b^M}(\cdot |E_N^{\infty})$ satisfies the conditions of Assumption \ref{Ass.Infinite}. Consequently, Proposition \ref{Prop.UpperTail} implies that we can find $C_1(a,b,\theta, \cm, \cd_2, A_0)$ and $D_1(a,b,\theta, \cm, \cd_2, A_0)$, such that for $N \geq N_0$ 
\begin{equation}\label{Eq.Tail1}
\mathcal{J}_{a^N, b^M}(\lambda_1 \geq N D_1(a,b,\theta, \cm, \cd_2, A_0)) \leq C_1(a,b,\theta, \cm, \cd_2, A_0) e^{-NA_0}.
\end{equation}
Setting $N = A$ in (\ref{Eq.Tail1}) gives (\ref{Eq.UpperTailAlpha}) in this case.

{\raggedleft \bf  Case 2.} Suppose that $\cm \geq 1$ and $B < A$. Consider the sequence $M(N)$, such that $M(N) = \lceil \cm N \rceil$ for $N \neq B$, and $M(B) = A$, and note $M(N) \geq N$. Since $\cm \geq 1$ and $B < A$, we note that 
$$\cm B - A \leq \cm A - B \mbox{ and } A - \cm B \leq \cm A  -B.$$
Consequently, we obtain
$$|M(B) - \cm B| = |A -\cm B| \leq \cm A -B \leq \cd_1.$$
As in Case 1, we conclude that $\mathcal{J}_{a^N, b^M}$ satisfies the conditions of Definition \ref{DefScale} with $\cd = \max(\cd_1, 1)$, and so from Proposition \ref{Prop.UpperTail} we get (\ref{Eq.Tail1}). Setting $N = B$ in (\ref{Eq.Tail1}) gives (\ref{Eq.UpperTailAlpha}) in this case
 once we note $\mathcal{J}_{a^B, b^A} = \mathcal{J}_{a^A, b^B}$, which holds by homogeneity and $\mathcal{J}_{\rho_1, \rho_2}=\mathcal{J}_{\rho_2, \rho_1}$, cf. Remark \ref{Rem.Commute}.

{\raggedleft \bf Case 3.} Suppose that $\cm < 1$ and $B < A$. Consider the sequence $M(N)$, such that $M(N) = \lceil \cm^{-1} N \rceil$ for $N \neq B$, and $M(B) = A$, and note $M(N) \geq N$. We observe that 
$$|M(B) - \cm^{-1} B| = |A - \cm^{-1}B| \leq \cm^{-1} \cd_1,$$
and so we conclude that $\mathcal{J}_{a^N, b^M}$ satisfies the conditions of Definition \ref{DefScale} with $\cd = \cd_3$ and $\cm$ replaced with $\cm^{-1}$, where $\cd_3 = \max(\cm^{-1} \cd_1, 1)$. From our work in Section \ref{Section3.2} we see that $\mathcal{J}_{a^N, b^M}(\cdot |E_N^{\infty})$ satisfies the conditions of Assumption \ref{Ass.Infinite}. Consequently, Proposition \ref{Prop.UpperTail} implies we can find $C_1(a,b,\theta, \cm^{-1}, \cd_3, A_0)$ and $D_1(a,b,\theta, \cm^{-1}, \cd_3, A_0)$, such that for $N \geq N_0$ 
\begin{equation}\label{Eq.Tail2}
\mathcal{J}_{a^N, b^M}(\lambda_1 \geq N D_1(a,b,\theta, \cm^{-1}, \cd_3, A_0)) \leq C_1(a,b,\theta, \cm^{-1}, \cd_3, A_0) e^{-NA_0}.
\end{equation}
Setting $N = B$ in (\ref{Eq.Tail2}) gives (\ref{Eq.UpperTailAlpha}) in this case once we note $\mathcal{J}_{a^B, b^A} = \mathcal{J}_{a^A, b^B}$. 

{\raggedleft \bf Case 4.} Suppose that $\cm < 1$ and $B \geq A$. Consider the sequence $M(N)$, such that $M(N) = \lceil \cm^{-1} N \rceil$ for $N \neq A$, and $M(A) = B$, and note $M(N) \geq N$. Since $\cm < 1$ and $B \geq A$, we note that 
$$B -\cm^{-1} A \leq  \cm^{-1} B - A \mbox{ and } \cm^{-1} A - B \leq \cm^{-1} B - A.$$
Consequently, we obtain
$$|M(A) - \cm^{-1} A| = |B -\cm^{-1} A| \leq \cm^{-1} B - A \leq \cm^{-1}\cd_1.$$
We conclude that $\mathcal{J}_{a^N, b^M}$ satisfies the conditions of Definition \ref{DefScale} with $\cd = \cd_3$ as in Case 3, and $\cm$ replaced with $\cm^{-1}$, and so from Proposition \ref{Prop.UpperTail} we get (\ref{Eq.Tail2}). Setting $N = A$ in (\ref{Eq.Tail2}) gives (\ref{Eq.UpperTailAlpha}) in this case.

%
%
\subsection{Proof of Lemma \ref{Lem.CaseVILengthBound}} \label{Section3.9} Before we go to the proof of Lemma \ref{Lem.CaseVILengthBound} we introduce a bit of notation and prove a useful lemma.

Let $\mathbb{Y}(n) = \{\lambda \in \mathbb{Y}: |\lambda| = n\}$ be the set of partitions of $n$. In \cite{Ke97} Kerov introduced the following probability measure on $\mathbb{Y}(n)$
\begin{equation}\label{Eq.JPM}
\mathbb{M}^{(n)}_{\theta}(\lambda) =\frac{n!\theta^{n}}{\prod_{i=1}^{\ell(\lambda)}\prod_{j=1}^{\lambda_i}(\lambda_i-j+\theta(\lambda_j'-i)+\theta)(\lambda_i-j+\theta(\lambda_j'-i)+1)}.
\end{equation}
$\mathbb{M}^{(n)}_{\theta}$ is sometimes referred to as a {\em Jack-Plancherel measure}, as it forms a one-parameter generalization of the usual Plancherel measure on $\mathbb{Y}(n)$. We can extend $\mathbb{M}^{(n)}_{\theta}$ to a probability measure on $\mathbb{Y}$ by setting $\mathbb{M}^{(n)}_{\theta}(\lambda) = 0$ when $|\lambda| \neq n$. For a parameter $L > 0$, we also define the {\em Poissonized Jack-Plancherel} measure by
\begin{equation}\label{Eq.PJPM}
\mathbb{P}^{(L)}(\lambda) = \sum_{n = 0}^{\infty} \frac{e^{-L}L^n}{n!} M^{(n)}_{\theta}(\lambda).
\end{equation}
In other words, $\mathbb{P}^{(L)}$ is the measure we obtain by sampling $X$ according to a Poisson distribution with parameter $L$, and then sampling $\lambda$ according to $\mathbb{M}^{(X)}_{\theta}$. One key observation for us is that if $s_1, s_2 > 0$ and $L = s_1 s_2 \theta$, then
\begin{equation}\label{Eq.JPMEquals}
\mathbb{P}^{(L)}(\lambda) = \frac{J_\lambda(\tau_{s_1};\theta) \tilde{J}_{\lambda}(\tau_{s_2}; \theta)}{H_{\theta}(\tau_{s_1}, \tau_{s_2}) } = \mathcal{J}_{\tau_{s_1}, \tau_{s_2}}(\lambda).
\end{equation}
The latter can be deduced from equations (\ref{S12E2}), (\ref{Eq.PlanchEval}) and (\ref{Eq.CaseVINormalization_fml}).

The main result we require about $\mathbb{M}^{(n)}$ is contained in the following lemma.
\begin{lemma}\label{Lem.JPMTail} Fix $\theta, A_1 > 0$. We can find a constants $b_1, B_1, B_1'  > 0$, depending on $\theta, A_1$, such that for all $n \geq 1$ and $u \geq b_1$
\begin{equation}\label{Eq.PlUpperTail}
\mathbb{M}^{(n)}_{\theta}(\lambda_1 \geq u \sqrt{n}) \leq  B_1e^{- A_1 u \sqrt{n}}, \hspace{4mm} \mathbb{M}^{(n)}_{\theta}(\lambda_1' \geq u \sqrt{n}) \leq  B_1' e^{- A_1 u \sqrt{n}}.
\end{equation}
\end{lemma}
\begin{proof} We observe that the second inequality in (\ref{Eq.PlUpperTail}) follows from the first, once we use $\mathbb{M}^{(n)}_{\theta}(\lambda) = \mathbb{M}^{(n)}_{\theta^{-1}}(\lambda')$, which one deduces directly from (\ref{Eq.JPM}). Consequently, we only establish the first inequality in (\ref{Eq.PlUpperTail}). The argument we present is an adaptation of the one used in \cite[Lemma 6.6]{Fulman03}.

Fix $m \in \{1, \dots, n\}$ and $\lambda \in \mathbb{Y}(n)$, such that $\lambda_1 = m$. Set $\bar{\lambda} = (\lambda_2, \lambda_3, \dots)$, i.e., $\bar{\lambda}$ is the partition obtained from $\lambda$, by deleting its first row. Note that $\bar{\lambda} \in \mathbb{Y}(n-m)$. We observe from (\ref{Eq.JPM}) that 
\begin{equation*}
\begin{split}
&\mathbb{M}^{(n)}_{\theta}(\lambda) = \frac{n! \theta^{m}}{(n-m)!} \cdot  \mathbb{M}^{(n-m)}_{\theta}(\bar{\lambda}) \cdot \prod_{j = 1}^m \frac{1}{(m-j+\theta(\lambda_j'-1)+\theta)(m-j+\theta(\lambda_j'-1)+1)} \\
&\leq \frac{n! \theta^{m}}{(n-m)!} \cdot  \mathbb{M}^{(n-m)}_{\theta}(\bar{\lambda})\cdot \prod_{j = 1}^m \frac{1}{(m-j+\theta)(m-j+1)} \leq \frac{n! \theta^{m-1}}{(n-m)! m! (m-1)!} \cdot  \mathbb{M}^{(n-m)}_{\theta}(\bar{\lambda}). 
\end{split}
\end{equation*}
Summing over all $\lambda$'s such that $\lambda_1= m$, we conclude 
\begin{equation}\label{Eq.PlE1}
\begin{split}
&\mathbb{M}^{(n)}_{\theta}(\lambda_1 = m) \leq \frac{n! \theta^{m-1}}{(n-m)! m! (m-1)!} \leq \frac{(m/\theta) (n\theta)^m}{(m!)^2} \leq \theta^{-1} \cdot (e^2 n\theta/m^2)^m,
\end{split}
\end{equation}
where in the last inequality we used the following result due to \cite{R55} 
$$m! = \sqrt{2\pi} m^{m+1/2} e^{-m} e^{r_m}, \mbox{ for } \frac{1}{12m + 1} < r_m < \frac{1}{12m}.$$
We now pick $b_1 > 1$ sufficiently large so that $e^2 \theta \leq b_1^2 \cdot e^{-A_1 - 1},$ 
and note from (\ref{Eq.PlE1}) that for $ u \geq b_1$
$$\mathbb{M}^{(n)}_{\theta}(\lambda_1 \geq   u \sqrt{n}) \leq n \cdot \max_{n \geq m \geq u \sqrt{n}} \mathbb{M}^{(n)}_{\theta}(\lambda_1 = m) \leq  n \cdot \theta^{-1} \cdot  e^{-(A_1 +1) u \sqrt{n}} ,$$
which implies the first inequality in (\ref{Eq.PlUpperTail}) with $B_1 = \theta^{-1} \cdot \sup_{n \geq 1} n e^{-\sqrt{n}}$.
\end{proof}

With the above result in place we are ready to give the proof of Lemma \ref{Lem.CaseVILengthBound}.
\begin{proof}[Proof of Lemma \ref{Lem.CaseVILengthBound}] We continue with the same notation as in the statement of the lemma. Let $b_1, B_1, B_1'$ be as in Lemma \ref{Lem.JPMTail} for our present $\theta$ and $A_1 = 2 A_0$. We claim that for any integer $d \geq D_2^{VI} = b_1 (2 t_1t_2 \theta)^{1/2}$ we can find $C> 0$, depending on $t_1, t_2, \theta, A_0$ and $d$, such that for $K \geq 1$
\begin{equation}\label{Eq.TruncRed1}
\mathcal{J}_{\tau_{Kt_1/d}, \tau_{Kt_2/d}}\left(\lambda_1 \geq  K  \right) \leq C e^{-KA_0} \mbox{ and } \mathcal{J}_{\tau_{Kt_1/d}, \tau_{Kt_2/d}}\left(\lambda_1' \geq  K  \right) \leq C e^{-KA_0}.
\end{equation}
Note that (\ref{Eq.TruncRed1}) implies
$$\mathcal{J}_{\tau_{Kt_1/d}, \tau_{Kt_2/d}}\left( E_{K}^{K}  \right) \geq 1 - 2C e^{-KA_0}, \mbox{ and } \mathcal{J}_{\tau_{Kt_1/d}, \tau_{Kt_2/d}}\left(\lambda_1 =  K  \right) \leq C e^{-KA_0}.$$
The last two inequalities imply (\ref{Eq.CaseVIUpperTail}) and the first inequality with $K = Nd$ implies (\ref{Eq.CaseVITruncateLen}). We have thus reduced the proof to establishing (\ref{Eq.TruncRed1}).\\

Fix any integer $d \geq b_1 (2 t_1t_2 \theta)^{1/2}$. Set $L =  K^2t_1t_2 \theta/d^2$, and let $X$ be a Poisson random variable with parameter $L$. From (\ref{Eq.JPMEquals}) we know that
\begin{equation}\label{Eq.TruncRed2}
\begin{split}
&\mathcal{J}_{\tau_{Kt_1/d}, \tau_{Kt_2/d}}\left( \lambda_1 \geq K \right) = \sum_{n = 0}^{\infty} \mathbb{P}(X = n) \mathbb{M}^{(n)}_{\theta}(\lambda_1 \geq K) \\
&\leq \mathbb{P}(X \not \in [L/2, 2L]) + \max_{n \in [L/2, 2L]} \mathbb{M}^{(n)}_{\theta}(\lambda_1 \geq K).
\end{split}
\end{equation}
As $d(2 t_1t_2 \theta)^{-1/2} \geq b_1$, we get from Lemma \ref{Lem.JPMTail} for $n \in [L/2, 2L]$
\begin{equation*}
    \begin{split}
        &\mathbb{M}^{(n)}_{\theta}(\lambda_1 \geq K) = \mathbb{M}^{(n)}_{\theta} \left(\lambda_1 \geq d(2 t_1t_2 \theta)^{-1/2} \sqrt{2L} \right)\leq \mathbb{M}^{(n)}_{\theta} \left(\lambda_1 \geq d(2 t_1t_2 \theta)^{-1/2} \sqrt{n} \right) \\
        & \leq B_1 e^{-A_1 d(2 t_1t_2 \theta)^{-1/2} \sqrt{n} } \leq B_1 e^{-A_1 d(2 t_1t_2 \theta)^{-1/2} \sqrt{L/2} } = B_1 e^{- K A_1 /2} = B_1 e^{-KA_0}.
    \end{split}
\end{equation*}
In addition, from the Chernoff bounds for Poisson tails \cite[Theorem 4.1]{MR95} with $\delta = 1$, and \cite[Theorem 4.2]{MR95} with $\delta = 1/2$, we have 
\begin{equation*}
\mathbb{P}(X \not \in [L/2, 2L]) \leq (e/4)^{L} + e^{-L/8}.
\end{equation*}
Putting the last two bounds into (\ref{Eq.TruncRed2}) gives
$$\mathcal{J}_{\tau_{Kt_1/d}, \tau_{Kt_2/d}}\left( \lambda_1 \geq K \right)  \leq  (e/4)^{L} + e^{-L/8} + B_1 e^{-KA_0},$$
and an analogous argument gives
$$\mathcal{J}_{\tau_{Kt_1/d}, \tau_{Kt_2/d}}\left( \lambda'_1 \geq K \right)  \leq  (e/4)^{L} + e^{-L/8} + B_1' e^{-KA_0}.$$
The last two displayed equations give (\ref{Eq.TruncRed1}) once we recall that $L =  K^2t_1t_2 \theta/d^2$.
\end{proof}

%
%
\section{Computing the equilibrium measures}\label{Section4} In Section \ref{Section4.1} we recall some of the results for discrete $\beta$-ensembles from \cite{DK22}. Most notably, these results include the {\em loop equations} or {\em Nekrasov's equations}, see Proposition \ref{Prop.Loop}, which were introduced for discrete $\beta$-ensembles in \cite{BGG17}. The main consequence of these equations for the present paper is that they allow us to express the density of the equilibrium measure $\mu$ in Proposition \ref{Prop.GlobalLLN} in terms of certain holomorphic functions, see Lemma \ref{Lem.MuThroughR}. As it turns out, these functions are explicit low-degree polynomials for the Jack measures we study, which allows us to find explicit formulas for the density of the equilibrium measures in each of the six cases in Section \ref{Section3}.

%
%
\subsection{Loop equations}\label{Section4.1} In this section we consider a sequence of discrete $\beta$-ensembles as in Definition \ref{Def.BetaEnsembles} with $R = R_K$ that satisfy Assumption \ref{Ass.Finite}. From Proposition \ref{Prop.GlobalLLN} we know that the empirical measures $\mu_K$ converge weakly in probability to a deterministic measure $\mu$. Notice that since $\mu_K$ are supported in $[0, R_K  K^{-1} + (K-1) K^{-1}  \theta]$, we have that $\mu$ is supported in $[0, \mathsf{R} + \theta]$. Next, we make the following assumption about the structure of the weight function $w(K x;K)$ in a complex neighborhood of $[0, \mathsf{R} + \theta]$.

\begin{assumption} \label{Ass.Loop1} We assume that we have an open set $\mathcal{M}  \subseteq \mathbb{C}$, such that $[0, \mathsf{R} + \theta] \subset \mathcal{M}$. In addition, we assume that we have holomorphic functions $\Phi^+_K, \Phi^-_K$ on $K\cdot \mathcal{M}$, such that 
\begin{equation}\label{S41E1}
\begin{split}
&\frac{w(Kx;K)}{w(Kx-1;K)}=\frac{\Phi_K^+(Kx)}{\Phi_K^-(Kx)},
\end{split}
\end{equation} 
whenever $x \in [ K^{-1}, R_K  K^{-1}+  (K-1)  K^{-1}  \theta]$. Moreover, 
\begin{equation*}
\begin{split}
&\Phi^{-}_K(Kz) = \Phi^{-}(z) + O \left(K^{-1} \right) \mbox{ and } \Phi^{+}_K(Kz) = \Phi^{+}(z) + O \left(K^{-1} \right),
\end{split}
\end{equation*}
where the constants in the big $O$ notation are uniform over $z$ in compact subsets of $\mathcal{M}$. All the aforementioned functions are holomorphic in $\mathcal{M}$, $\Phi^{\pm}$ do not depend on $K$ and are positive (in particular real) on $(0, \mathsf{R} + \theta)$.
\end{assumption}

With the above notation in place we can state the loop equations we require. They are essentially a restatement of \cite[Proposition 3.3]{DK22} and \cite[Theorem 4.1]{BGG17}.
\begin{proposition}\label{Prop.Loop} Suppose that Assumptions \ref{Ass.Finite} and \ref{Ass.Loop1} hold. Define
\begin{equation}\label{S41E2}
H_K(z) = \Phi^-_K(z) \cdot \mathbb{E}_{\mathbb{P}_K} \left[ \prod_{i = 1}^K  \frac{z - \ell_i - \theta}{z - \ell_i}\right] + \Phi^+_K(z) \cdot \mathbb{E}_{\mathbb{P}_K} \left[ \prod_{i = 1}^K  \frac{z - \ell_i + \theta - 1}{z - \ell_i - 1} \right] - \frac{r^-(K)}{z} - \frac{r^+(K)}{z - s_K}, 
\end{equation}
where $s_K = R_K + 1 + (K-1) \cdot \theta$ and 
\begin{align}\label{S41E3}
\begin{split}
&r^-(K) = \Phi^-_K(0) \cdot (-\theta) \cdot \mathbb{P}_K(\ell_K = 0) \cdot \mathbb{E}_{\mathbb{P}_K} \left[ \prod_{i= 1}^{K-1}  \frac{ \ell_i + \theta }{ \ell_i}\Big{\vert}  \ell_K = 0 \right],\\
&r^+(K) = \Phi^+_K(s_K) \cdot \theta \cdot \mathbb{P}_K(\ell_1 = s_K -1 ) \cdot \mathbb{E}_{\mathbb{P}_K} \left[\prod_{i = 2}^K  \frac{s_K - \ell_i  + \theta -1}{s_K - \ell_i - 1} \Big{\vert}  \ell_1 = s_K - 1 \right].
\end{split}
\end{align}
Then, $H_K(z)$ is a holomorphic function on the rescaled domain $K \cdot \mathcal{M}$ in Assumption \ref{Ass.Loop1}. Moreover, if $\Phi^{\pm}_K(z)$ are
polynomials of degree at most $d$, then so is $H_K(z)$.
\end{proposition}
\begin{proof} Assumptions \ref{Ass.Finite} and \ref{Ass.Loop1} show that \cite[Assumptions 1,2 and 3]{DK22} hold with $N$ there replaced with $K$ here, $M_N = R_K$, and $\mathsf{M} = \mathsf{R}$. The analyticity of $H_K$ in $K \cdot \mathcal{M}$ is now a consequence of \cite[Proposition 3.3]{DK22}. If $\Phi^{\pm}_K$ are polynomials of degree at most $d$, then we can take $\mathcal{M} = \mathbb{C}$, and so $H_K$ is entire. Moreover, from (\ref{S41E2}) we conclude $|H_K(z)| = O(|z|^d)$ as $|z| \rightarrow \infty$, and so $H_K$ is a polynomial of degree at most $d$ by Liouville's theorem, see e.g. \cite[Corollary 7.4, Chapter III]{SL99}.
\end{proof}

The next assumption we require aims to upper bound the quantities $r^{\pm}(K)$ in (\ref{S41E3}).
\begin{assumption}\label{Ass.Loop2} We assume that there are constants $C, c, a > 0$, such that for all large $K$
\begin{align}\label{S41E4}
\begin{split}
&\mathbb{P}_K(\ell_K = 0) \cdot \left|\Phi_K^-(0)\right| \leq C \exp ( - c K^a) \mbox{ and } \\
&\mathbb{P}_K(\ell_1= R_K + (K-1) \cdot \theta ) \cdot \left| \Phi_K^+(R_K + 1 +(K-1) \cdot \theta) \right| \leq C \exp ( - c K^a).
\end{split}
\end{align}
\end{assumption}

If $\mu$ is the equilibrium measure from Proposition \ref{Prop.GlobalLLN}, we define its Stieltjes transform through
\begin{equation}\label{S41E5}
G_{\mu}(z) := \int_\mathbb{R} \frac{\mu(dx)}{z - x}.
\end{equation}
Since the support of $\mu$ is contained in $[0, \mathsf{R} + \theta]$, we have that $G_{\mu}(z)$ is holomorphic in the domain $\mathbb{C} \setminus [0, \mathsf{R} + \theta]$. If $\Phi^{\pm}(z)$ are as in Assumption \ref{Ass.Loop1}, we also define for $z \in \mathcal{M} \setminus [0, \mathsf{R} + \theta]$ the function
\begin{align}\label{S41E6}
\begin{split}
&R_{\mu}(z) = \Phi^-(z) \cdot e^{- \theta G_{\mu} (z) }+  \Phi^+(z) \cdot e^{ \theta G_{\mu} (z) }.
\end{split}
\end{align}
The following statement summarizes the main results we require about the function $R_{\mu}$ and its relationship to the equilibrium measure $\mu$.

\begin{lemma}\label{Lem.MuThroughR} Suppose that Assumptions \ref{Ass.Finite}, \ref{Ass.Loop1} and \ref{Ass.Loop2} hold. If $R_{\mu}$ is as in (\ref{S41E6}), then it can be analytically extended to $\mathcal{M}$ and is real-valued on $\mathcal{M} \cap \mathbb{R}$. Moreover, if $\mu$ is as in Proposition \ref{Prop.GlobalLLN}, then it has the following density on $\mathbb{R}$
\begin{equation}\label{S41E7}
\mu(x) = \frac{{\bf 1}\{x \in (0, \mathsf{R} + \theta)\} }{\theta \pi } \cdot \mathrm{arccos} \left( \frac{R_\mu(x)}{2 \sqrt{\Phi^-(x)  \Phi^+(x)}}\right).
\end{equation}
 Lastly, if we further suppose that $\Phi_K^{\pm}$ are polynomials of degree at most $d$ for all large $K$, then the same is true for $R_{\mu}$.
\end{lemma}
\begin{remark}\label{Rem.Density} We recall that $\mathrm{arccos}$ is as in Definition \ref{Def.Arccos} and $\Phi^+(x) \Phi^-(x) > 0$ on $(0, \mathsf{R}+ \theta)$ by Assumption \ref{Ass.Loop1}. In particular, the density in (\ref{S41E7}) is well-defined and continuous on $(0, \mathsf{R} + \theta)$.
\end{remark}
\begin{proof} Assumptions \ref{Ass.Finite}, \ref{Ass.Loop1} and \ref{Ass.Loop2} show that \cite[Assumptions 1-4]{DK22} hold with $N$ there replaced with $K$ here, $M_N = R_K$, and $\mathsf{M} = \mathsf{R}$. The first statement in the lemma now follows from \cite[Lemma 3.5]{DK22}, and the density formula in (\ref{S41E7}) from \cite[Lemma 3.6]{DK22}. 

In the remainder we focus on showing the last statement, and assume that $\Phi_K^{\pm}$ are polynomials of degree at most $d$ for all large $K$. From Proposition \ref{Prop.Loop} we know that $H_K$ are polynomials of degree at most $d$ for all large $K$. In addition, from \cite[Equation (A-8)]{DK22} we have for each $v \in \mathbb{C} \setminus [0, \mathsf{R} + \theta]$ 
$$\lim_{K \rightarrow \infty} H_K(Kv) - R_{\mu}(v) = 0.$$
Combining the last two statements, we see that $R_{\mu}$ is the pointwise limit of polynomials of degree at most $d$, and hence is itself a polynomial of degree at most $d$.
\end{proof}

In the sections below we use Lemma \ref{Lem.MuThroughR} to compute the densities of the equilibrium measures in each of the six cases in Section \ref{Section3}. In particular, we find expressions for the functions $\Phi_{K}^{\pm}, \Phi^{\pm}$ from Assumption \ref{Ass.Loop1} and $R_{\mu}$ from (\ref{S41E6}).

%
%
\subsection{Case I: $\rho_1 = a^N$, $\rho_2 = b^M$}\label{Section4.2} The goal of this section is to establish the following statement.
\begin{lemma}\label{Lem.CaseIFormulaMu} Make the same assumptions as in Lemma \ref{Lem.CaseIWeakConv}. Then, $\mu^I$ has the following density, denoted $\mu^I(x)$, with respect to Lebesgue measure
\begin{equation}\label{Eq.CaseIFormulaMu}
\mu^I(x) = \frac{{\bf 1}\{x > 0\} }{\theta \pi } \cdot \mathrm{arccos} \left( \frac{(1+ab)x + \theta(ab\cm -1)}{2 \sqrt{abx (x - \theta + \theta \cm ) }}\right),
\end{equation}
where we recall that $\arccos$ is as in Definition \ref{Def.Arccos}.
\end{lemma}
\begin{proof} Let $D_1^I$ be as in Lemma \ref{Lem.CaseIWeakConv}, and fix $\mathsf{R}$, such that
\begin{equation}\label{Eq.CaseIRConst}
\mathsf{R} \geq \max\left(D_1^I, \frac{\theta (1 + \sqrt{ab\cm})^2 }{1-ab} \right).
\end{equation}
Our first task is to check that for $R_N = \lceil N \mathsf{R} \rceil$, the measures $\mathcal{J}_{a^N, b^M}(\cdot |E_N^{R_N})$ satisfy Assumptions \ref{Ass.Finite}, \ref{Ass.Loop1} and \ref{Ass.Loop2} with $K$ replaced with $N$.  

As $R_N = \lceil N \mathsf{R} \rceil$, the first point in Assumption \ref{Ass.Finite} holds with $A_1 = 1$. The second, third and fourth points were verified in Section \ref{Section3.2}, cf. (\ref{Eq.CaseIBetaEnsembleV2}), (\ref{Eq.CaseILimitVN}) and (\ref{Eq.CaseIVDerBound}). Turning our attention to Assumption \ref{Ass.Loop1}, we see from $\Gamma(z+1) = z \Gamma(z)$ that 
$$\frac{w(Nx;N)}{w(Nx-1;N)} = ab \cdot \frac{\Gamma(Nx)}{\Gamma(Nx+1)} \cdot \frac{\Gamma(Nx + (M-N+1)\theta)}{\Gamma(Nx-1 + (M-N+1)\theta )} = \frac{ab \cdot ( Nx-1 + (M-N+1)\theta)}{Nx}.$$
The latter shows that Assumption \ref{Ass.Loop1} holds with $\mathcal{M} = \mathbb{C}$ and
\begin{equation}\label{Eq.CaseIPhis}
\begin{split}
&\Phi^-_N(z) = z/N, \hspace{2mm} \Phi^+_N(z) = ab((M-N+1)\theta -1)/N + abz/N, \hspace{2mm} \\
&\Phi^-(z) = z, \hspace{2mm} \Phi^+(z) = ab\theta (\cm -1) + ab z .
\end{split}
\end{equation}
Lastly, we note that $\Phi^-_N(0) = 0$, which verifies the first line in (\ref{S41E4}). In addition, from (\ref{Eq.CaseIRConst}) we have the inclusion of events
$$\{\ell_1 = R_N + (N-1) \cdot \theta\} \subseteq \{\lambda_1 \geq ND^I_1 \},$$
which together with (\ref{Eq.CaseIUpperTail}) implies the second line in (\ref{S41E4}). 

Our work so far shows that the conditions in Lemma \ref{Lem.MuThroughR} hold for the measures $\mathcal{J}_{a^N, b^M}(\cdot |E_N^{R_N})$, and our next step is to compute $\mu(x)$ in (\ref{S41E7}). From (\ref{Eq.CaseIPhis}) we have that $\Phi_N^{\pm}$ are degree $1$ polynomials, and so by Lemma \ref{Lem.MuThroughR}, we conclude that $R_\mu$ is a polynomial of degree at most $1$. As $|z| \rightarrow \infty$ we can expand the Stieltjes transform as
\begin{equation}\label{Eq.CaseIStielTrans}
G_{\mu}(z) = \int_{\mathbb{R}} \frac{\mu(dx)}{z-x} = \frac{1}{z} + O\left(\frac{1}{|z|^2} \right),
\end{equation}
and so 
$$R_{\mu}(z) = \Phi^-(z) \cdot e^{- \theta G_{\mu} (z) }+  \Phi^+(z) \cdot e^{ \theta G_{\mu} (z) } = z - \theta + ab\theta (\cm -1) + ab z + ab\theta + O(|z|^{-1}).$$
The last equation and the fact that $R_{\mu}(z)$ is a polynomial of degree at most $1$ show that
\begin{equation}\label{Eq.CaseIR}
R_{\mu}(z) = (1+ab)z + \theta(ab\cm -1).
\end{equation}
From Lemma \ref{Lem.MuThroughR} we conclude that 
\begin{equation}\label{Eq.CaseIMuF}
\mu(x) = \frac{{\bf 1}\{x \in (0, \mathsf{R} + \theta)\} }{\theta \pi } \cdot \mathrm{arccos} \left( \frac{(1+ab)x + \theta(ab\cm -1)}{2 \sqrt{abx [x + \theta (\cm -1) ] }}\right).
\end{equation}

We finally explain how (\ref{Eq.CaseIFormulaMu}) follows from (\ref{Eq.CaseIMuF}). Observe that 
$$R_\mu(x)^2 - 4 \Phi^-(x) \Phi^+(x) = (1-ab)^2(x - x_1)(x-x_2),$$
where 
$$x_1 =  \frac{\theta (1-\sqrt{ab\cm})^2}{1-ab} \mbox{, and } x_2 = \frac{\theta (1+\sqrt{ab\cm})^2}{1-ab}.$$
From (\ref{Eq.CaseIRConst}) $\mathsf{R} + \theta \geq x_2$, and so for $x \geq \mathsf{R}+ \theta$ we have $R_\mu(x)^2 - 4 \Phi^-(x) \Phi^+(x) \geq 0$, or equivalently
$$ \frac{(1+ab)x + \theta(ab\cm -1)}{2 \sqrt{abx [x + \theta (\cm -1) ] }} \geq 1.$$
The latter means that the arccosine in (\ref{Eq.CaseIMuF}) vanishes if $x \geq \mathsf{R} + \theta$, cf. Definition \ref{Def.Arccos}, and so the right side of (\ref{Eq.CaseIMuF}) agrees with the right side of (\ref{Eq.CaseIFormulaMu}). Lastly, from (\ref{Eq.CaseIRConst}) $\mathsf{R} \geq D_1^I$, and so by Lemma \ref{Lem.CaseIWeakConv} we have $\mu^I(x) = \mu(x)$. Consequently, the left sides of (\ref{Eq.CaseIFormulaMu}) and (\ref{Eq.CaseIMuF}) also agree.
\end{proof}

%
%
\subsection{Case II: $\rho_1 = a^N$, $\rho_2 = b_{\beta}^M$}\label{Section4.3}
The goal of this section is to establish the following statement.
\begin{lemma}\label{Lem.CaseIIFormulaMu} Make the same assumptions as in Lemma \ref{Lem.CaseIIWeakConv}. Then, $\mu^{II}$ has the following density, denoted $\mu^{II}(x)$, with respect to Lebesgue measure
\begin{equation}\label{Eq.CaseIIFormulaMu}
\mu^{II}(x) = \frac{{\bf 1}\{x \in (0, \cm + \theta )\} }{\theta \pi } \cdot \mathrm{arccos} \left( \frac{(1-ab)x + ab\cm-\theta}{2 \sqrt{ab x (\cm +\theta - x) }}\right),
\end{equation}
where we recall that $\arccos$ is as in Definition \ref{Def.Arccos}.
\end{lemma}
\begin{proof} Our first task is to check that the measures $\mathcal{J}_{a^N, b^M_\beta}(\cdot)$ satisfy Assumptions \ref{Ass.Finite}, \ref{Ass.Loop1} and \ref{Ass.Loop2} with $K$ replaced with $N$, and $R_K$ replaced with $M = M(N)$. 

Assumption \ref{Ass.Finite} was verified in Section \ref{Section3.2}. Turning our attention to Assumption \ref{Ass.Loop1}, we see from $\Gamma(z+1) = z \Gamma(z)$ that 
$$\frac{w(Nx;N)}{w(Nx-1;N)} = ab \cdot \frac{\Gamma(Nx)}{\Gamma(Nx+1)} \cdot \frac{\Gamma(M+N\theta-Nx+2-\theta)}{\Gamma(M+N\theta-Nx+1-\theta)} = \frac{ab \cdot (M+N\theta-Nx+1-\theta)}{Nx}.$$
The latter shows that Assumption \ref{Ass.Loop1} holds with $\mathcal{M} = \mathbb{C}$ and
\begin{equation}\label{Eq.CaseIIPhis}
\begin{split}
&\Phi^-_N(z) = z/N, \hspace{2mm} \Phi^+_N(z) = ab(M+N\theta+1-\theta)/N - abz/N, \hspace{2mm} \\
&\Phi^-(z) = z, \hspace{2mm} \Phi^+(z) = ab(\cm+\theta) - ab z .
\end{split}
\end{equation}
Lastly, we note that $\Phi^-_N(0) = 0 =\Phi^+_N(M+1+(N-1)\cdot\theta)$, which verifies (\ref{S41E4}).

Our work so far shows that the conditions in Lemma \ref{Lem.MuThroughR} hold for the measures $\mathcal{J}_{a^N, b^M_\beta}(\cdot)$, and our next step is to compute $\mu(x)$ in (\ref{S41E7}), which from Lemma \ref{Lem.CaseIIWeakConv} is the same as $\mu^{II}(x)$. From (\ref{Eq.CaseIIPhis}) we have that $\Phi_N^{\pm}$ are degree $1$ polynomials, and so by Lemma \ref{Lem.MuThroughR}, we conclude that $R_\mu$ is a polynomial of degree at most $1$. As $|z|\to\infty$, we can expand the Stieltjes transform as (\ref{Eq.CaseIStielTrans}) and so $$R_{\mu}(z) = \Phi^-(z) \cdot e^{- \theta G_{\mu} (z) }+  \Phi^+(z) \cdot e^{ \theta G_{\mu} (z) } = (1-ab)z+ab\cm-\theta+ O(|z|^{-1}).$$
The last equation and the fact that $R_{\mu}(z)$ is a polynomial of degree at most $1$ show that
\begin{equation}\label{Eq.CaseIIR}
R_{\mu}(z) = (1-ab)z+ab\cm-\theta.
\end{equation}
Equation (\ref{Eq.CaseIIFormulaMu}) now follows from Lemmas \ref{Lem.CaseIIWeakConv} and \ref{Lem.MuThroughR}.
\end{proof}

%
%
\subsection{Case III: $\rho_1 = a^N$, $\rho_2 = \tau_s$}\label{Section4.4}The goal of this section is to establish the following statement.
\begin{lemma}\label{Lem.CaseIIIFormulaMu} Make the same assumptions as in Lemma \ref{Lem.CaseIIIWeakConv}. Then, $\mu^{III}$ has the following density, denoted $\mu^{III}(x)$, with respect to Lebesgue measure
\begin{equation}\label{Eq.CaseIIIFormulaMu}
\mu^{III}(x) = \frac{{\bf 1}\{x > 0\} }{\theta \pi } \cdot \mathrm{arccos} \left( \frac{x+at\theta-\theta}{2 \sqrt{at\theta x }}\right),
\end{equation}
where we recall that $\arccos$ is as in Definition \ref{Def.Arccos}.
\end{lemma}
\begin{proof} Let $D_1^{III}$ be as in Lemma \ref{Lem.CaseIIIWeakConv}, and fix $\mathsf{R}$ such that
\begin{equation}\label{Eq.CaseIIIRConst}
\mathsf{R} \geq \max\left(D_1^{III}, 2(at + 1)\theta \right).
\end{equation}
Our first task is to check that for $R_N = \lceil N \mathsf{R} \rceil$, the measures $\mathcal{J}_{a^N, \tau_{Nt}}(\cdot |E_N^{R_N})$ satisfy Assumptions \ref{Ass.Finite}, \ref{Ass.Loop1} and \ref{Ass.Loop2} with $K$ replaced with $N$.

As $R_N = \lceil N \mathsf{R} \rceil$, we see that the first point in Assumption \ref{Ass.Finite} holds with $A_1 = 1$. The other three points were verified in \cite[Section 6.3]{DD22}, see the first two displayed equations on \cite[page 83]{DD22}. Turning our attention to Assumption \ref{Ass.Loop1}, we see from $\Gamma(z+1) = z \Gamma(z)$ that 
$$\frac{w(Nx;N)}{w(Nx-1;N)} = atN\theta \cdot \frac{\Gamma(Nx)}{\Gamma(Nx+1)} = \frac{at\theta}{x}.$$
The latter shows that Assumption \ref{Ass.Loop1} holds with $\mathcal{M} = \mathbb{C}$ and
\begin{equation}\label{Eq.CaseIIIPhis}
\begin{split}
&\Phi^-_N(z) = z/N, \hspace{2mm} \Phi^+_N(z) = at\theta, \hspace{2mm} \Phi^-(z) = z, \hspace{2mm} \Phi^+(z) = at\theta .
\end{split}
\end{equation}
Lastly, we note that $\Phi^-_N(0) = 0$, which verifies the first line in (\ref{S41E4}). In addition, from (\ref{Eq.CaseIIIRConst}) we have the inclusion of events
$$\{\ell_1 = R_N + (N-1) \cdot \theta\} \subseteq \{\lambda_1 \geq ND_1^{III} \},$$
which together with (\ref{Eq.CaseIIIUpperTail}) implies the second line in (\ref{S41E4}).

Our work so far shows that the conditions in Lemma \ref{Lem.MuThroughR} hold for the measures $\mathcal{J}_{a^N, \tau_{Nt}}(\cdot |E_N^{R_N})$, and our next step is to compute $\mu(x)$ in (\ref{S41E7}). From (\ref{Eq.CaseIIIPhis}) we have that $\Phi_N^{\pm}$ are polynomials of degree at most $1$, and so by Lemma \ref{Lem.MuThroughR}, we conclude that $R_\mu$ is also a polynomial of degree at most $1$. As $|z|\to\infty$, we can expand the Stieltjes transform as (\ref{Eq.CaseIStielTrans}) and so 
$$R_{\mu}(z) = \Phi^-(z) \cdot e^{- \theta G_{\mu} (z) }+  \Phi^+(z) \cdot e^{ \theta G_{\mu} (z) } =z+at\theta-\theta + O(|z|^{-1}).$$
The last equation and the fact that $R_{\mu}(z)$ is a polynomial of degree at most $1$ show that
\begin{equation}\label{Eq.CaseIIIR}
R_{\mu}(z) = z+at\theta-\theta.
\end{equation}
From Lemma \ref{Lem.MuThroughR} we conclude that 
\begin{equation}\label{Eq.CaseIIIMuF}
\mu(x) = \frac{{\bf 1}\{x \in (0, \mathsf{R} + \theta)\} }{\theta \pi } \cdot \mathrm{arccos} \left( \frac{x+at\theta-\theta}{2 \sqrt{at\theta x }}\right).
\end{equation}

We finally explain how (\ref{Eq.CaseIIIFormulaMu}) follows from (\ref{Eq.CaseIIIMuF}). From (\ref{Eq.CaseIIIRConst}) $\mathsf{R} \geq D_1^{III}$, and so by Lemma \ref{Lem.CaseIIIWeakConv} we have $\mu^{III}(x) = \mu(x)$, showing that the left sides of (\ref{Eq.CaseIIIFormulaMu}) and (\ref{Eq.CaseIIIMuF}) agree. From (\ref{Eq.CaseIIIRConst}) $\sqrt{\mathsf{R} + \theta} \geq\sqrt{at\theta} + \sqrt{\theta}$ and so for $x \geq \mathsf{R} + \theta$
$$  (\sqrt{x} -\sqrt{at\theta})^2 \geq \theta \mbox{, or equivalently } \frac{x+at\theta-\theta}{2 \sqrt{at\theta x }} \geq 1.$$ 
The latter implies that the arccosine in (\ref{Eq.CaseIIIMuF}) vanishes if $x \geq \mathsf{R} + \theta$, cf. Definition \ref{Def.Arccos}, and so the right sides of (\ref{Eq.CaseIIIFormulaMu}) and (\ref{Eq.CaseIIIMuF}) also agree.
\end{proof}

%
%
\subsection{Case IV: $\rho_1 = a_{\beta}^N$, $\rho_2 = b_{\beta}^M$}\label{Section4.5} The goal of this section is to establish Lemma \ref{Lem.CaseIVFormulaMu} below. For its proof we require the following lemma, whose proof is given in Section \ref{Section4.8}.

\begin{lemma}\label{Lem.CaseIVFormulaIntegral}Make the same assumptions as in Lemma \ref{Lem.CaseIVWeakConv} and fix an integer $d \geq D_2^{IV}$ as in Lemma \ref{Lem.CaseIVLengthBound} for some $A_0 > 0$. Then, 
\begin{equation}\label{CaseIV_int}
\int_{\mathbb{R}} x\mu^{IV}_d(dx)=\frac{ab\cm}{(1-ab)d^2\theta} +\frac{\theta}{2}.
\end{equation}
\end{lemma}

\begin{lemma}\label{Lem.CaseIVFormulaMu} Make the same assumptions as in Lemma \ref{Lem.CaseIVWeakConv} and also assume $d \geq \max(D_2^{IV}, \alpha \theta^{-1})$, where $D_2^{IV}$ is as in Lemma \ref{Lem.CaseIVLengthBound} for some $A_0 > 0$ and  
\begin{equation}\label{Eq.CaseIVLowerBoundD}
\alpha = \frac{ab (\cm + 1)+ 2\sqrt{ab \cm} }{1-ab}.
\end{equation}
Then, $\mu^{IV}_{d}$ has the following density, denoted $\mu^{IV}_{d}(x)$, with respect to Lebesgue measure
\begin{equation}\label{Eq.CaseIVFormulaMu}
\mu^{IV}_d(x)= \begin{cases} \dfrac{1 }{\theta \pi }\cdot \mathrm{arccos} \left(  \dfrac{\left(1+ab\right)(x-\theta)-ab d^{-1}(\cm+1) }{2 \sqrt{ab(d^{-1}+\theta-x)(d^{-1}\cm+\theta-x)}}\right)  & \hspace{-3mm}\mbox{ if }   \theta -d^{-1}\alpha < x< \cR + \theta, \\
 \theta^{-1}  & \hspace{-3mm} \mbox{ if } 0 < x \leq \theta -d^{-1}\alpha , \\
 0  & \hspace{-3mm} \mbox{ otherwise.}
\end{cases} 
\end{equation}
Here, $\cR = d^{-1}\min(1,\cm)$ and we recall that $\arccos$ is as in Definition \ref{Def.Arccos}.
\end{lemma}
\begin{proof} From our work in Section \ref{Section3.5} we know that $\mathbb{P}_K^d$ satisfy Assumption \ref{Ass.Finite}. Using $\Gamma(z+1) = z \Gamma(z)$, we get
\begin{equation*}
\frac{w(Kx;K)}{w(Kx-1;K)}
= \frac{ab \cdot (\Delta_K^N-Kx+1)(\Delta_K^M-Kx+1)}{(Kx)(Kx+\theta-1)},
\end{equation*}
where $\Delta_{K}^N=N_K+(K-1)\theta$, $\Delta_{K}^M=M_K+(K-1)\theta$. The latter shows that $\mathbb{P}_K^d$ satisfy Assumption \ref{Ass.Loop1} with $\mathcal{M} = \mathbb{C}$ and
\begin{equation}\label{Eq.CaseIVPhis}
\begin{split}
&\Phi^-_K(z) = \frac{z}{K}\cdot\left(\frac{z}{K}+\frac{\theta-1}{K}\right), \hspace{2mm} \Phi^+_K(z) = ab\left(\frac{\Delta_K^N+1}{K}-\frac{z}{K}\right)\left(\frac{\Delta_K^M+1}{K}-\frac{z}{K}\right),\\
&\Phi^-(z) = z^2,\hspace{2mm} \Phi^+(z) = ab\left(\frac{1}{d}+\theta-z\right)\left(\frac{\cm}{d}+\theta-z\right). 
\end{split}
\end{equation}
Lastly, we note that $\Phi^-_K(0) = 0 = \Phi^+_K(R_K+1+(K-1)\cdot\theta)=0$, which verifies the conditions in (\ref{S41E4}).\\

Our work so far shows that the conditions in Lemma \ref{Lem.MuThroughR} hold for the measures $\mathbb{P}_K^d$. Our next step is to compute $\mu(x)$ in (\ref{S41E7}), which from Lemma \ref{Lem.CaseIVWeakConv} is the same as $\mu^{IV}_d(x)$. Since $\Phi^{\pm}_K$ are degree $2$ polynomials, we conclude $R_{\mu}$ is a polynomial of degree at most $2$. As $|z|\to\infty$, we can expand the Stieltjes transform as 
\begin{equation}\label{Eq.CaseIVStielTrans}
G_{\mu}(z) = \int_{\mathbb{R}} \frac{\mu(dx)}{z-x} = \frac{1}{z} + \frac{1}{z^2}\cdot\int_{\mathbb{R}}x \mu(dx)+O\left(\frac{1}{|z|^3} \right).
\end{equation}
Combining (\ref{S41E6}), (\ref{CaseIV_int}), (\ref{Eq.CaseIVPhis}) and (\ref{Eq.CaseIVStielTrans}), we obtain
\begin{equation*}
\begin{split}
&R_{\mu}(z) =\left(1+ab\right)z^2-\left[\theta+ab\left(\frac{\cm+1}{d}+\theta\right)\right]z +\left(-1+ab\right)\theta\int_{\mathbb{R}}x\mu(dx)+ab\left(\frac{\cm}{d^2}-\frac{\theta^2}{2}\right) \\ 
& +\frac{\theta^2}{2}+O\left(|z|^{-1}\right) = \left(1+ab\right)z^2-\left[\theta+ab\left(\frac{\cm+1}{d}+\theta\right)\right]z + O\left(|z|^{-1}\right).
\end{split}
\end{equation*}
As $R_\mu$ is a polynomial of degree at most $2$, we conclude
\begin{equation}\label{Eq.CaseIVR}
R_{\mu}(z) = \left(1+ab\right)z^2-\left[\theta+ab\left(\frac{\cm+1}{d}+\theta\right)\right]z.
\end{equation}
Substituting $\Phi^{\pm}$ from (\ref{Eq.CaseIVPhis}), $R_{\mu}$ from (\ref{Eq.CaseIVR}) and $\cR = d^{-1}\min(1,\cm)$ into (\ref{S41E7}) gives
\begin{equation}\label{Eq.CaseIVMuF}
\mu(x) = \frac{{\bf 1}\{x\in(0,\theta+d^{-1}\min(1,\cm ))\} }{\theta \pi }\cdot \mathrm{arccos} \left( \frac{\left(1+ab\right)(x-\theta)-abd^{-1}(\cm+1) }{2 \sqrt{ab(d^{-1}+\theta-x)(d^{-1}\cm+\theta-x)}}\right).
\end{equation}

What remains is to show that the right sides of (\ref{Eq.CaseIVFormulaMu}) and (\ref{Eq.CaseIVMuF}) agree. We claim that 
\begin{equation}\label{Eq.CaseIVP}
P(x) < 0 \mbox{ and } P(x)^2 - 4 ab(d^{-1} + \theta -x)(d^{-1}\cm + \theta - x) \geq 0 \mbox{ if } x \in (0, \theta - d^{-1}\alpha],
\end{equation}
where 
$$P(x) = (1+ab)(x-\theta) -abd^{-1}(\cm + 1).$$
If true, then the argument in the $\arccos$ in (\ref{Eq.CaseIVMuF}) is at most $-1$ for $x\in (0, \theta - d^{-1}\alpha]$, which by the definition of $\arccos$ in Definition \ref{Def.Arccos} shows that the right sides of (\ref{Eq.CaseIVFormulaMu}) and (\ref{Eq.CaseIVMuF}) agree. 

In the remainder we quickly verify (\ref{Eq.CaseIVP}). Note that 
$$P(x)^2 - 4 ab(d^{-1} + \theta -x)(d^{-1}\cm + \theta - x) = (1-ab)^2(x-x_1)(x-x_2),$$
where 
$$x_1 = \theta - \frac{ab(\cm+1)}{(1-ab)d} + \frac{2\sqrt{ab\cm}}{(1-ab)d}, \hspace{2mm} x_2 = \theta - \frac{ab(\cm+1)}{(1-ab)d} - \frac{2\sqrt{ab\cm}}{(1-ab)d} = \theta - d^{-1}\alpha.$$
The last two displayed equations give the second inequality in (\ref{Eq.CaseIVP}). In addition, 
$$P(x)\leq P(\theta - d^{-1}\alpha) = -d^{-1}[(1+ab)\alpha +ab(\cm+1)] \mbox{ for $x\in (0, \theta - d^{-1}\alpha]$},$$
which implies the first inequality in (\ref{Eq.CaseIVP}). This concludes the proof of (\ref{Eq.CaseIVP}).
\end{proof}

%
%
\subsection{Case V: $\rho_1 = a_{\beta}^N$, $\rho_2 = \tau_s$}\label{Section4.6}  The goal of this section is to establish Lemma \ref{Lem.CaseVFormulaMu} below. For its proof we require the following lemma, whose proof is given in Section \ref{Section4.8}.
\begin{lemma}\label{Lem.CaseVFormulaIntegral} Make the same assumptions as in Lemma \ref{Lem.CaseVWeakConv} and fix an integer $d \geq D_2^{V}$ as in Lemma \ref{Lem.CaseVLengthBound} for some $A_0 > 0$. Then, 
\begin{equation}\label{CaseV_int}
\int_{\mathbb{R}}x\mu^V_d(dx)= \frac{at}{d^2}+\frac{\theta}{2}.
\end{equation}
\end{lemma}

\begin{lemma}\label{Lem.CaseVFormulaMu} Make the same assumptions as in Lemma \ref{Lem.CaseVWeakConv} and also assume $d \geq \max(D_2^{V}, \alpha \theta^{-1})$, where $D_2^{V}$ is as in Lemma \ref{Lem.CaseVLengthBound} for some $A_0 > 0$ and  
\begin{equation}\label{Eq.CaseVLowerBoundD}
\alpha =  at\theta + 2 \sqrt{at \theta}.
\end{equation}
Then, $\mu^{V}_{d}$ has the following density, denoted $\mu^{V}_{d}(x)$, with respect to Lebesgue measure
\begin{equation}\label{Eq.CaseVFormulaMu}
\mu^V_{d}(x) = \begin{cases} 
\dfrac{1 }{\theta \pi } \cdot \mathrm{arccos}\left(\dfrac{x-\theta-atd^{-1}\theta}{2\sqrt{atd^{-1}\theta \left(\theta + d^{-1} -x\right)}}\right) & \mbox{ if } \theta - d^{-1} \alpha < x < d^{-1}+\theta, \\
\theta^{-1} & \mbox{ if } 0 < x \leq \theta - d^{-1} \alpha, \\
0 & \mbox{ otherwise.}
\end{cases}
\end{equation}
where we recall that $\arccos$ is as in Definition \ref{Def.Arccos}.
\end{lemma}
\begin{proof} From our work in Section \ref{Section3.6} we know that $\mathbb{P}_K^d$ satisfy Assumption \ref{Ass.Finite}. Using $\Gamma(z+1) = z \Gamma(z)$, we get
\begin{equation*}
\begin{split}
\frac{w(Kx;K)}{w(Kx-1;K)} &= \left(K at\theta/d\right) \cdot \frac{\Gamma(Kx)}{\Gamma(Kx+1)} \cdot \frac{\Gamma(Kx+\theta-1)}{\Gamma(Kx+\theta)} \cdot\frac{\Gamma(\Delta_K^N-Kx+2)}{\Gamma(\Delta_K^N-Kx+1)} \\
&= \frac{(K at\theta/d)  \left(\Delta_K^N-Kx+1\right)}{\left(Kx\right)\left(Kx+\theta-1\right)},
\end{split}
\end{equation*}
where $\Delta_K^N=N_K+(K-1)\theta$. The latter shows that Assumption \ref{Ass.Loop1} holds with $\mathcal{M} = \mathbb{C}$ and
\begin{equation}\label{Eq.CaseVPhis}
\begin{split}
&\Phi^-_K(z) = \frac{z}{K}\cdot\left(\frac{z}{K}+\frac{\theta-1}{K}\right),\hspace{2mm} \Phi^+_K(z) = \frac{a t\theta}{d}\cdot\left(\frac{\Delta_K^N+1}{K}-\frac{z}{K}\right), \\
&\Phi^-(z) = z^2, \hspace{2mm}  \Phi^+(z) = atd^{-1}\theta\cdot\left(\theta + d^{-1} -z\right).
\end{split}
\end{equation}
Lastly, we note that $\Phi^-_K(0) = 0= \Phi^+_K(R_K+1+(K-1)\cdot\theta)=0$, which verifies the conditions in (\ref{S41E4}).\\

Our work so far shows that the conditions in Lemma \ref{Lem.MuThroughR} hold for the measures $\mathbb{P}_K^d$. Our next step is to compute $\mu(x)$ in (\ref{S41E7}), which from Lemma \ref{Lem.CaseVWeakConv} is the same as $\mu^{V}_d(x)$. Combining (\ref{S41E6}), (\ref{Eq.CaseIVStielTrans}), (\ref{CaseV_int}) and (\ref{Eq.CaseVPhis}), we obtain
\begin{equation*}
\begin{split}
&R_{\mu}(z) = z^2-\left(\theta+\frac{at\theta}{d}\right)z -\theta\int_{\mathbb{R}}x\mu(dx)+\frac{\theta^2}{2} -\frac{at\theta}{d^2} + O(|z|^{-1})  =z^2-\left(\theta+\frac{at\theta}{d}\right)z + O(|z|^{-1}).
\end{split}
\end{equation*}
From (\ref{Eq.CaseVPhis}) we have that $\Phi_K^{\pm}$ are polynomials of degree at most $2$, and so by Lemma \ref{Lem.MuThroughR}, we conclude the same for $R_\mu$. The last displayed equation then gives
\begin{equation}\label{Eq.CaseVR}
R_{\mu}(z) =z^2- \theta z - atd^{-1}\theta z.
\end{equation}
Substituting $\Phi^{\pm}$ from (\ref{Eq.CaseVPhis}) and $R_{\mu}$ from (\ref{Eq.CaseVR}) into (\ref{S41E7}) gives
\begin{equation}\label{Eq.CaseVMuF}
\mu(x) = \frac{{\bf 1}\{x\in(0,\theta+d^{-1})\} }{\theta \pi }\cdot\mathrm{arccos}\left(\frac{x-\theta - at d^{-1} \theta }{2\sqrt{at d^{-1} \theta \left(\theta + d^{-1} -x\right)}}\right).
\end{equation}

What remains is to show that the right sides of (\ref{Eq.CaseVFormulaMu}) and (\ref{Eq.CaseVMuF}) agree. We claim that 
\begin{equation}\label{Eq.CaseVP}
P(x) < 0 \mbox{ and } P(x)^2 - 4 a t d^{-1} \theta \left(\theta + d^{-1} -x\right) \geq 0 \mbox{ if } x \in (0, \theta - d^{-1}\alpha],
\end{equation}
where 
$$P(x) = x-\theta - at d^{-1} \theta.$$
If true, then the argument in the $\arccos$ in (\ref{Eq.CaseVMuF}) is at most $-1$ for $x\in (0, \theta - d^{-1}\alpha]$, which by the definition of $\arccos$ in Definition \ref{Def.Arccos} shows that the right sides of (\ref{Eq.CaseVFormulaMu}) and (\ref{Eq.CaseVMuF}) agree. 

In the remainder we quickly verify (\ref{Eq.CaseVP}). Note that 
$$P(x)^2 - 4 a t d^{-1} \theta \left(\theta + d^{-1} -x\right)= (x-x_1)(x-x_2),$$
where 
$$x_1 = \theta - \frac{at\theta }{d} + \frac{2\sqrt{at\theta}}{d}, \hspace{2mm} x_2 = \theta - \frac{at\theta }{d} -  \frac{2\sqrt{at\theta}}{d}= \theta - d^{-1}\alpha.$$
The last two displayed equations give the second inequality in (\ref{Eq.CaseVP}). In addition, 
$$P(x)\leq P(\theta - d^{-1}\alpha) = -d^{-1}[\alpha + at \theta] \mbox{ for $x\in (0, \theta - d^{-1}\alpha]$},$$
which implies the first inequality in (\ref{Eq.CaseVP}). This concludes the proof of (\ref{Eq.CaseVP}).
\end{proof}

%
%
\subsection{Case VI: $\rho_1 = \tau_{s_1}$, $\rho_2 = \tau_{s_2}$}\label{Section4.7} The goal of this section is to establish Lemma \ref{Lem.CaseVIFormulaMu} below. For its proof we require the following lemma, whose proof is given in Section \ref{Section4.8}.

\begin{lemma}\label{Lem.CaseVIFormulaIntegral}
Make the same assumptions as in Lemma \ref{Lem.CaseVIWeakConv} and fix an integer $d \geq D_2^{VI}$ as in Lemma \ref{Lem.CaseVILengthBound} for some $A_0 > 0$. Then, 
\begin{equation}\label{CaseVI_int}
\int_{\mathbb{R}}x\mu^{VI}_{d}(dx)=\frac{t_1 t_2 \theta}{d^2} + \frac{\theta}{2}.
\end{equation}
\end{lemma}

\begin{lemma}\label{Lem.CaseVIFormulaMu} Make the same assumptions as in Lemma \ref{Lem.CaseVIWeakConv} and also assume that
$$d \geq \max(D_2^{VI}, 2\sqrt{t_1t_2}, 2 \theta \sqrt{t_1t_2} ),$$
where $D_2^{VI}$ is as in Lemma \ref{Lem.CaseVILengthBound} for some $A_0 > 0$. Then, $\mu^{VI}_{d}$ has the following density, denoted $\mu^{VI}_{d}(x)$, with respect to Lebesgue measure
\begin{equation}\label{Eq.CaseVIFormulaMu}
\mu^{VI}_{d}(x) = \begin{cases} 
 \dfrac{1}{\theta \pi} \cdot \arccos \left( \dfrac{ x - \theta}{2 d^{-1} \theta \sqrt{t_1t_2}} \right)& \mbox{ if } \theta (1 - 2 d^{-1} \sqrt{t_1t_2})  < x < \theta (1 + 2d^{-1} \sqrt{t_1t_2}), \\
\theta^{-1} & \mbox{ if } 0 < x \leq \theta (1 - 2 d^{-1} \sqrt{t_1t_2}), \\
0 & \mbox{ otherwise,}
\end{cases}
\end{equation}
where we recall that $\arccos$ is as in Definition \ref{Def.Arccos}.
\end{lemma}
\begin{proof}From our work in Section \ref{Section3.7} we know that $\mathbb{P}_K^d$ satisfy Assumption \ref{Ass.Finite}. Using $\Gamma(z+1) = z \Gamma(z)$, we get
\begin{equation*}
\frac{w(Kx;K)}{w(Kx-1;K)} = (K^2t_1t_2 \theta^2/d^2) \cdot\frac{\Gamma(Kx)}{\Gamma(Kx+1)} \cdot  \frac{\Gamma(Kx+\theta-1)}{\Gamma(Kx + \theta)} = \frac{K^2t_1t_2\theta^2/d^2}{Kx(Kx + \theta -1)}.
\end{equation*}
The latter shows that Assumption \ref{Ass.Loop1} holds with $\mathcal{M} = \mathbb{C}$ and
\begin{equation}\label{Eq.CaseVIPhisK}
\begin{split}
&\Phi^-_K(z) = \frac{z}{K} \cdot \left(\frac{z}{K}+\frac{\theta-1}{K}\right), \hspace{2mm} \Phi^+_K(z) =t_1t_2 d^{-2} \theta^2, \hspace{2mm}\Phi^-(z) = z^2 ,\hspace{2mm} \Phi^+(z) = t_1t_2 d^{-2} \theta^2.
\end{split}
\end{equation}
Lastly, we note that $\Phi^-_K(0) = 0$, which verifies the first line in (\ref{S41E4}). In addition, as $d \geq D_2^{VI}$, we see that (\ref{Eq.CaseVIUpperTail}) implies the second line in (\ref{S41E4}).\\

Our work so far shows that the conditions in Lemma \ref{Lem.MuThroughR} hold for the measures $\mathbb{P}_K^d$ and from Lemma \ref{Lem.CaseVIWeakConv} we have $\mu = \mu^{VI}_d$. As $|z|\to\infty$, we can expand the Stieltjes transform as in (\ref{Eq.CaseIVStielTrans}) and then use (\ref{S41E6}),  (\ref{CaseVI_int}), and (\ref{Eq.CaseVIPhisK}) to obtain
\begin{equation*}
R_{\mu}(z) = z^2-\theta z - \theta \int_{\mathbb{R}} x\mu(dx) +\frac{ t_1t_2 \theta^2 }{d^2}+\frac{\theta^2}{2} + O(|z|^{-1})= z^2-\theta z + O(|z|^{-1}).
\end{equation*}
From (\ref{Eq.CaseVIPhisK}) we have that $\Phi_K^{\pm}$ are polynomials of degree at most $2$, and by Lemma \ref{Lem.MuThroughR} we conclude the same for $R_\mu$, which gives 
\begin{equation}\label{Eq.CaseVIR1}
R_{\mu}(z) =  z^2-\theta z.
\end{equation}
From Lemma \ref{Lem.MuThroughR} we conclude that
\begin{equation}\label{Eq.CaseVIMuF}
\mu(x) = \frac{{\bf 1}\{x \in (0, \theta + 1)\} }{\theta \pi } \cdot \mathrm{arccos}\left(\frac{x-\theta}{2d^{-1}\theta\sqrt{t_1t_2}}\right).
\end{equation} 
What remains is to see that the right side of (\ref{Eq.CaseVIMuF}) agrees with (\ref{Eq.CaseVIFormulaMu}), which is clear from the definition of $\arccos$ in Definition \ref{Def.Arccos} and the fact that $d \geq \max(2 \sqrt{t_1t_2}, 2 \theta\sqrt{t_1t_2})$.
\end{proof}

%
%
\subsection{Proof of Lemmas \ref{Lem.CaseIVFormulaIntegral}, \ref{Lem.CaseVFormulaIntegral} and \ref{Lem.CaseVIFormulaIntegral}}\label{Section4.8} As the proofs of the three lemmas are quite similar, we combine them into one. In addition, for clarity we split the proof into two steps.\\

{\bf \raggedleft Step 1.} Let $g(x) = {\bf 1}\{x > 0\} \cdot \min(x, 2 + \theta)$, which is a bounded continuous function on $\mathbb{R}$. Observe that $\mu_K$ in each case is supported on $[0, R_K K^{-1} + (K-1)K^{-1} \theta]$, which for large $K$ is contained in $[0, 2 + \theta]$. From Lemmas \ref{Lem.CaseIVWeakConv}, \ref{Lem.CaseVWeakConv} and \ref{Lem.CaseVIWeakConv} we conclude
\begin{equation}\label{Eq.XY1}
\lim_{N \rightarrow \infty} \mathbb{E}\left[ \int_{\mathbb{R}} x\mu_{Nd}(dx) \right] = \lim_{N \rightarrow \infty} \mathbb{E}\left[ \int_{\mathbb{R}} g(x)\mu_{Nd}(dx) \right] = \int_{\mathbb{R}} g(x) \mu_d(dx) = \int_{\mathbb{R}}x \mu_d(dx),
\end{equation}
where $\mu_d$ equals $\mu_d^{IV}, \mu_d^{V}, \mu_d^{VI}$ in Cases IV, V, VI, respectively. From (\ref{Eq.CaseIVSubseq}), (\ref{Eq.CaseVSubseq}), (\ref{Eq.CaseVISubseq}) and (\ref{Eq.XY1}) we see that 
\begin{equation}\label{Eq.XY2}
\int_{\mathbb{R}}x \mu_d(dx) = \lim_{N \rightarrow \infty} \mathbb{E}\left[ \frac{1}{Nd}\sum_{i = 1}^{Nd} \frac{\lambda_i^N + (Nd-i)\theta}{Nd} \vert E_{Nd}^{R_{Nd}} \right]  = \frac{\theta}{2} +  \lim_{N \rightarrow \infty} (Nd)^{-2} \mathbb{E}\left[ |\lambda^N| \vert E_{Nd}^{R_{Nd}} \right],
\end{equation}
where the law of $\lambda^N$ is $\mathcal{J}_{a^N_{\beta}, b^M_{\beta}}$, $\mathcal{J}_{a^N_{\beta}, \tau_{Nt}}$ and $ \mathcal{J}_{\tau_{Nt_1}, \tau_{Nt_2}}$, and $R_{Nd}$ equals $\min(M,N)$, $N$ and $Nd$ in Cases IV, V and VI, respectively. We have thus reduced the proof of the three lemmas to finding the limit on the right side of (\ref{Eq.XY2}) in each of the three cases.

We claim that 
\begin{equation}\label{Eq.XY3}
\begin{split}
&\lim_{N \rightarrow \infty} (Nd)^{-2} \mathbb{E}\left[ |\lambda^N| \right] = \frac{ab\cm}{(1-ab)d^2\theta}\mbox{ in Case IV,}\\
&\lim_{N \rightarrow \infty} (Nd)^{-2} \mathbb{E}\left[ |\lambda^N| \right] = \frac{at}{d^2}\mbox{ in Case V,}\\
&\lim_{N \rightarrow \infty} (Nd)^{-2} \mathbb{E}\left[ |\lambda^N| \right] = \frac{t_1 t_2 \theta}{d^2}\mbox{ in Case VI,}
\end{split}
\end{equation}
and also 
\begin{equation}\label{Eq.XY4}
\begin{split}
&\lim_{N \rightarrow \infty} N^{-4} \mathbb{E}\left[ |\lambda^N| (|\lambda^N|-1) \right] = \frac{a^2b^2\cm^2}{(1-ab)^2\theta^2} \mbox{ in Case IV,}\\
&\lim_{N \rightarrow \infty} N^{-4} \mathbb{E}\left[ |\lambda^N| ( |\lambda^N| - 1) \right] = a^2t^2 \mbox{ in Case V,}\\
&\lim_{N \rightarrow \infty} N^{-4} \mathbb{E}\left[ |\lambda^N|(|\lambda^N| - 1) \right] = t_1^2t_2^2 \theta^2 \mbox{ in Case VI.}
\end{split}
\end{equation}
We will prove (\ref{Eq.XY3}) and (\ref{Eq.XY4}) in Step 2 below. Here, we assume their validity and conclude the proofs of the three lemmas.\\

We observe that 
\begin{equation}\label{Eq.XY5}
\begin{split}
&\mathbb{E}\left[ |\lambda^N| \vert E_{Nd}^{R_{Nd}} \right] = \frac{ \mathbb{E}\left[ |\lambda^N| \cdot {\bf 1}_{E_{Nd}^{R_{Nd}}} \right]}{\mathbb{P}(E_{Nd}^{R_{Nd}})} = \frac{ \mathbb{E}\left[ |\lambda^N| \right]}{\mathbb{P}(E_{Nd}^{R_{Nd}})} -  \frac{ \mathbb{E}\left[ |\lambda^N| \cdot {\bf 1}_{(E_{Nd}^{R_{Nd}})^c} \right]}{\mathbb{P}(E_{Nd}^{R_{Nd}})} \\
&=  \mathbb{E}\left[ |\lambda^N| \right] + O(N^2 e^{-NA_0/2}), 
\end{split}
\end{equation}
where the constant in the big $O$ notation depends on $a,b,\cm, d,\theta$ in Case IV, $a,t, d,\theta$ in Case V, and $t_1, t_2, d,\theta$ in Case VI. We mention that in deriving (\ref{Eq.XY5}) we used that $d \geq D_2^{IV}$ in Case IV, $d \geq D_2^{V}$ in Case V, and $d \geq D_2^{VI}$ in Case VI, so that from Lemmas \ref{Lem.CaseIVLengthBound}, \ref{Lem.CaseVLengthBound} and \ref{Lem.CaseVILengthBound} we have
$$\mathbb{P}(E_{Nd}^{R_{Nd}}) =  1 +O(e^{-NA_0}).$$
We also used equations (\ref{Eq.XY3}), (\ref{Eq.XY4}), which by the Cauchy-Schwarz inequality imply
$$\mathbb{E}\left[ |\lambda^N| \cdot {\bf 1}_{(E_{Nd}^{R_{Nd}})^c} \right] \leq \mathbb{E}\left[ |\lambda^N|^2 \right]^{1/2} \cdot \mathbb{P}\left((E_{Nd}^{R_{Nd}})^c\right)^{1/2} = O(N^2e^{-NA_0/2}).$$
Combining (\ref{Eq.XY2}), (\ref{Eq.XY3}) and (\ref{Eq.XY5}), we obtain the statements of the three lemmas.\\ 

{\bf \raggedleft Step 2.} Using the homogeneity of the Jack symmetric functions and (\ref{Eq.CaseIVNormalization_fml}), we have in Case IV for any $u \in [0, a^{-1}b^{-1})$ 
$$\mathbb{E}\left[ u^{|\lambda^N|} \right] = \frac{\sum_{\lambda \in \mathbb{Y}} u^{|\lambda|} J_{\lambda}(a_{\beta}^N;\theta) \tilde{J}_{\lambda}(b_{\beta}^M;\theta)}{H_{\theta}(a^N_\beta, b^M_\beta) }  = \frac{(1-uab)^{\frac{-NM}{\theta}}}{(1-ab)^{\frac{-NM}{\theta}}}.$$
Differentiating the above with respect to $u$, and setting $u = 1$ gives
$$ \mathbb{E}\left[ |\lambda^N| \right] = \frac{d}{du}  \frac{(1-uab)^{\frac{-NM}{\theta}}}{(1-ab)^{\frac{-NM}{\theta}}} \Big{\vert}_{u=1} = \frac{abNM}{\theta(1-ab)},$$
$$ \mathbb{E}\left[ |\lambda^N| (|\lambda^N|-1) \right] = \frac{d^2}{du^2}  \frac{(1-uab)^{\frac{-NM}{\theta}}}{(1-ab)^{\frac{-NM}{\theta}}} \Big{\vert}_{u=1} = \frac{a^2b^2(NM \theta^{-1})(NM\theta^{-1} + 1) }{(1-ab)^2} .$$
The last two displayed equations imply the first lines in (\ref{Eq.XY3}) and (\ref{Eq.XY4}).

Similarly, in Case V we have from (\ref{Eq.CaseVNormalization_fml}) that for any $u \in [0, \infty)$ 
$$\mathbb{E}\left[ u^{|\lambda^N|} \right] = \frac{\sum_{\lambda \in \mathbb{Y}} u^{|\lambda|}J_{\lambda}(a_{\beta}^N;\theta) \tilde{J}_{\lambda}(\tau_{Nt};\theta)}{H_{\theta}(a^N_\beta, \tau_{Nt})} =  \exp(N^2uat - N^2at).$$
Differentiating the above with respect to $u$, and setting $u = 1$ gives
$$ \mathbb{E}\left[ |\lambda^N| \right] = \frac{d}{du}  \exp(N^2uat - N^2at) \Big{\vert}_{u=1} = N^2at,$$
$$ \mathbb{E}\left[ |\lambda^N| (|\lambda^N|-1) \right] = \frac{d^2}{du^2}  \exp(N^2uat - N^2at)\Big{\vert}_{u=1} =N^4a^2t^2 .$$
The last two displayed equations imply the second lines in (\ref{Eq.XY3}) and (\ref{Eq.XY4}).

Lastly, in Case VI we have from (\ref{Eq.CaseVINormalization_fml}) that for any $u \in [0, \infty)$ 
$$\mathbb{E}\left[ u^{|\lambda^N|} \right] = \frac{\sum_{\lambda \in \mathbb{Y}} u^{|\lambda|}J_{\lambda}(\tau_{Nt_1};\theta) \tilde{J}_{\lambda}(\tau_{Nt_2};\theta)}{H_{\theta}(\tau_{Nt_1}, \tau_{Nt_2})} =  \exp(N^2ut_1t_2\theta - N^2t_1t_2\theta ).$$
Differentiating the above with respect to $u$, and setting $u = 1$ gives
$$ \mathbb{E}\left[ |\lambda^N| \right] = \frac{d}{du}  \exp(N^2ut_1t_2\theta - N^2t_1t_2\theta) \Big{\vert}_{u=1} = N^2t_1t_2\theta,$$
$$ \mathbb{E}\left[ |\lambda^N| (|\lambda^N|-1) \right] = \frac{d^2}{du^2}  \exp(N^2ut_1t_2\theta - N^2t_1t_2\theta)\Big{\vert}_{u=1} =N^4t_1^2t_2^2 \theta^2 .$$
The last two displayed equations imply the third lines in (\ref{Eq.XY3}) and (\ref{Eq.XY4}).

%
%
\section{Global laws of large numbers}\label{Section5} In this section we prove Theorem \ref{thmMain}. Throughout this section we continue with the same notation as in the statement of the theorem.
%
%
\subsection{Proof of Case I}\label{Section5.1} We assume the conditions in Case I of Definition \ref{DefScale}. We seek to show that if $\{N_k\}_{k \geq 1}$ is any strictly increasing sequence of integers, then we can find a subsequence $\{N_{k_m} \}_{m \geq 1}$, such that for any compactly supported continuous function $h$ on $\mathbb{R}$, we have as $m \rightarrow \infty$
\begin{equation}\label{CaseIRed}
\int_{\mathbb{R}} h(x) \nu_{N_{k_m}}(dx) \Rightarrow  \int_{\mathbb{R}} h(x) f^{I}(x) dx,
\end{equation}
where $f^{I}$ is as in (\ref{MTE1}). In the sequel we suppose that $\{N_k\}_{ k \geq 1}$ as above is given, and for clarity split the proof into two steps.

{\bf \raggedleft Step 1.} Suppose first that $M(N_k) \geq N_k$ for infinitely many $k \geq 1$. Note that this implies that $\cm \geq 1$. We let $\{N_{k_m}\}_{ m \geq 1}$ be any subsequence such that $M(N_{k_m}) \geq N_{k_m}$. We now define the auxiliary sequence $\tilde{M}(N)$ through
\begin{equation}\label{Eq.CaseIAuxM}
\tilde{M} = \tilde{M}(N) = \begin{cases} M(N_{k_m}) &\mbox{ if } N = N_{k_m} \mbox{ for some $ m \geq 1$}, \\ \lceil \cm N \rceil & \mbox{ otherwise}.
\end{cases}
\end{equation}
In addition, we let $\tilde{Z}^N_i = \tilde{\lambda}_i^N - i \cdot \theta$ for $i \geq 1$ with $\tilde{\lambda}^N = (\tilde{\lambda}^N_i: i \geq 1)$ being distributed according to $\mathcal{J}_{a^N, b^{\tilde{M}}}$. Lastly, we let $\tilde{\nu}_N$ be as in (\ref{S13E1}) with $Z^N_i$ replaced with $\tilde{Z}^N_i$, and note that
\begin{equation}\label{Eq.CaseISplit}
\tilde{\nu}_N = \tilde{\nu}^1_N + \tilde{\nu}^2_N, \mbox{ where } \tilde{\nu}_N^1 = \frac{1}{N} \sum_{i = 1}^N \delta (\tilde{Z}_i^N/N ), \mbox{ and } \tilde{\nu}^2_N = \frac{1}{N}\sum_{i =N+1}^{\infty} \delta \left( -i \cdot \theta/N \right).
\end{equation}
We mention that in the last equality we used that $\mathcal{J}_{a^N, b^{\tilde{M}}}(\tilde{\lambda}^N_{N+1} > 0) = 0$, see the discussion after (\ref{Eq.Vanish}), and consequently for $i \geq N+1$ we have $\tilde{Z}^N_i = \tilde{\lambda}^N_{i} - i \cdot \theta = -i \cdot \theta$.

Notice that by construction $|\tilde{M}  - \cm N| \leq \cd + 1$, and $\tilde{M} \geq N$. Consequently, the measures $\mathcal{J}_{a^N, b^{\tilde{M}}}$ satisfy the conditions of Lemma \ref{Lem.CaseIWeakConv} and Lemma \ref{Lem.CaseIFormulaMu}. The last two lemmas show for $h_{\theta}(x) = h(x-\theta)$
\begin{equation*}
\frac{1}{N} \sum_{i = 1}^N h_{\theta}(\tilde{Z}_i^N/ N + \theta) \Rightarrow  \int_{\mathbb{R}} h_{\theta} (x) \mu^{I}(x) dx,
\end{equation*}
where $\mu^I$ is as in (\ref{Eq.CaseIFormulaMu}). Translating the arguments by $\theta$ gives
\begin{equation}\label{Eq.CaseIPiece1}
\int_{\mathbb{R}} h(x) \tilde{\nu}^1_{N}(dx) = \frac{1}{N} \sum_{i = 1}^N h(\tilde{Z}_i^N/ N ) \Rightarrow  \int_{\mathbb{R}} h_{\theta} (x) \mu^{I}(x) dx = \int_{\mathbb{R}} h (y) \mu^{I}(y + \theta) dy.
\end{equation}
On the other hand, we have by a straightforward Riemann sum approximation argument that
\begin{equation}\label{Eq.CaseIPiece2}
\int_{\mathbb{R}} h(x) \tilde{\nu}^2_{N}(dx) = \frac{1}{N} \sum_{i = N+1}^{\infty} h(-i \theta/N) \rightarrow  \int_{-\infty}^{-1} h (\theta x)  dx = \int_{-\infty}^{-\theta} h (y) \theta^{-1} dy.
\end{equation}
Combining (\ref{Eq.CaseIPiece1}) and (\ref{Eq.CaseIPiece2}) with the formula for $\mu^I$ from (\ref{Eq.CaseIFormulaMu}), we conclude 
$$\int_{\mathbb{R}} h(x) \tilde{\nu}_{N}(dx) \Rightarrow \int_{\mathbb{R}} h (x) f^{I}(x) dx.$$
The last equation implies (\ref{CaseIRed}), since $\tilde{\nu}_{N_{k_m}}$ and $\nu_{N_{k_m}}$ have the same laws by construction.\\

{\bf \raggedleft Step 2.} In this step we suppose instead that $M(N_k) \leq N_k$ for all but finitely many $k \geq 1$. Note that this implies that $\cm \leq 1$. We let $\{N_{k_m}\}_{ m \geq 1}$ be any subsequence such that $M(N_{k_m})$ is strictly increasing, and $ 1 \leq M(N_{k_m}) \leq N_{k_m}$. We also define the auxiliary sequence $\{\tilde{M}(\tilde{N})\}_{\tilde{N} \geq 1}$ through
\begin{equation}\label{Eq.CaseIAuxM2}
\tilde{M} = \tilde{M}(\tilde{N}) = \begin{cases} N_{k_m} &\mbox{ if } \tilde{N} = M(N_{k_m}) \mbox{ for some $ m \geq 1$}, \\ \lceil \cm^{-1} \tilde{N} \rceil & \mbox{ otherwise}.
\end{cases}
\end{equation}
Based on our construction, we make the following observations:
\begin{enumerate}
\item $\tilde{N} \leq \tilde{M}$ for all $\tilde{N} \geq 1$;
\item $|\tilde{M} - \cm^{-1} \tilde{N}| \leq 1 + \cm^{-1} \cd$;
\item if we set $\tilde{N}_m = M(N_{k_m})$ for $m \geq 1$, then $\mathcal{J}_{a^{N_{k_m}}, b^{M(N_{k_m})}} = \mathcal{J}_{b^{\tilde{N}_m}, a^{\tilde{M}(\tilde{N}_m)}}$. 
\end{enumerate}
We remark that the last statement follows from $\mathcal{J}_{\rho_1, \rho_2} = \mathcal{J}_{\rho_2, \rho_1} $, see Remark \ref{Rem.Commute}.

We now let $\tilde{Z}^{\tilde{N}}_i = \tilde{\lambda}_i^{\tilde{N}} - i \cdot \theta$ for $i \geq 1$ with $\tilde{\lambda}^{\tilde{N}} = (\tilde{\lambda}^{\tilde{N}}_i: i \geq 1)$ being distributed according to $\mathcal{J}_{b^{\tilde{N}}, a^{\tilde{M}}}$. In addition, we let $\tilde{\nu}_{\tilde{N}}$ be as in (\ref{S13E1}) with $N$ replaced with $\tilde{N}$, and $Z^{N}_i$ replaced with $\tilde{Z}^{\tilde{N}}_i$. From the first and second observations above we see that $\mathcal{J}_{b^{\tilde{N}}, a^{\tilde{M}}}$ satisfy the conditions of Lemma \ref{Lem.CaseIWeakConv} and Lemma \ref{Lem.CaseIFormulaMu}, with $a$ and $b$ swapped, and with $\cm$ replaced with $\cm^{-1}$. We may thus repeat our work in Step 1 above to conclude that as $\tilde{N} \rightarrow \infty$, for any compactly supported, continuous $g$
\begin{equation}\label{Eq.CaseIStep2Red1}
\int_{\mathbb{R}} g(x) \tilde{\nu}_{\tilde{N}}(dx) \Rightarrow \int_{\mathbb{R}} g (x) \hat{f}^{I}(x) dx,
\end{equation}
where
\begin{equation}\label{Eq.CaseIFHat}
 \hat{f}^{I}(x) = \begin{cases} \dfrac{1 }{\theta \pi } \cdot \mathrm{arccos} \left( \dfrac{(1+ab)(x+\theta) + \theta(ab\cm^{-1} -1)}{2 \sqrt{ab(x+\theta) ( x + \theta \cm^{-1}) }}\right) & \mbox{ if } x > -\theta, \\ \theta^{-1} &\mbox{ if } x < - \theta.\end{cases}
\end{equation}

Let us define the measures
$$\hat{\nu}_m := \frac{1}{N_{k_m}} \sum_{i = 1}^{\infty} \delta \left( \tilde{Z}^{\tilde{N}_m}_i/N_{k_m} \right) = \frac{1}{\tilde{M}(\tilde{N}_m)} \sum_{i = 1}^{\infty} \delta \left( \tilde{Z}^{\tilde{N}_m}_i/\tilde{M}(\tilde{N}_m) \right),$$
which by the third observation above have the same distribution as $\nu_{N_{k_m}}$. In addition, we note that since $h$ is continuous and of compact support, we have uniformly over $x \in \mathbb{R}$ that 
$$h(x \tilde{N}_m / \tilde{M}(\tilde{N}_m)) = h(x \cm) + o(1) \mbox{ as } m \rightarrow \infty.$$
The latter shows that as $m \rightarrow \infty$
\begin{equation*}
    \begin{split}
        &\int_{\mathbb{R}} h(x) \hat{\nu}_{m}(dx) = \frac{1}{\tilde{M}(\tilde{N}_m)} \sum_{i = 1}^{\infty} h \left( \tilde{Z}^{\tilde{N}_m}_i/\tilde{M}(\tilde{N}_m) \right) = \frac{\tilde{N}_m}{\tilde{M}(\tilde{N}_m)} \int_{\mathbb{R}} h( x \tilde{N}_m / \tilde{M}(\tilde{N}_m) ) \tilde{\nu}_{\tilde{N}_m}(dx)\\
         &= \frac{\tilde{N}_m}{\tilde{M}(\tilde{N}_m)} \int_{\mathbb{R}} h( x  \cm ) \tilde{\nu}_{\tilde{N}_m}(dx) + o(1) =  \int_{\mathbb{R}} \cm  h( x  \cm ) \tilde{\nu}_{\tilde{N}_m}(dx) + o(1) .
    \end{split}
\end{equation*}
Combining the last equation with (\ref{Eq.CaseIStep2Red1}) for $g(x) = \cm h(x \cm)$, and the convergence together lemma, see e.g. \cite[Theorem 3.1]{Billing}, we see that 
$$\int_{\mathbb{R}} h(x) \hat{\nu}_{m}(dx) \Rightarrow \int_{\mathbb{R}} \cm h (x \cm) \hat{f}^{I}(x) dx = \int_{\mathbb{R}} h (y) \hat{f}^{I}(\cm^{-1} y) dy.$$
The last equation and the distributional equality of $\hat{\nu}_m$ and $\nu_{N_{k_m}}$ together imply (\ref{CaseIRed}), once we observe that $\hat{f}^{I}(\cm^{-1} y) = f^I(y)$, which follows by comparing (\ref{Eq.CaseIFHat}) with (\ref{MTE1}).

%
%
\subsection{Proof of Cases II and III}\label{Section5.2} We assume the conditions in either Case II or Case III of Definition \ref{DefScale}. As the proofs in these cases are similar, we treat them together. We seek to show that for any compactly supported continuous function $h$ on $\mathbb{R}$, we have as $N \rightarrow \infty$
\begin{equation}\label{CaseIIRed}
\int_{\mathbb{R}} h(x) \nu_{N}(dx) \Rightarrow  \int_{\mathbb{R}} h(x) f(x) dx,
\end{equation}
where $f = f^{II}$ as in (\ref{MTE2}) in Case II and $f = f^{III}$ as in (\ref{MTE3}) in Case III. We note
\begin{equation}\label{Eq.CaseIISplit}
\nu_N = \nu^1_N + \nu^2_N, \mbox{ where } \nu_N^1 = \frac{1}{N} \sum_{i = 1}^N \delta (Z_i^N/N ), \mbox{ and } \nu^2_N = \frac{1}{N}\sum_{i = N+1}^{\infty} \delta \left( -i \cdot \theta/N \right).
\end{equation}
We mention that in the last equality we used that $\mathcal{J}_{a^N, b^M_\beta}(\lambda^N_{N+1} > 0) = 0$ (see the discussion under (\ref{Eq.CaseIINormalization})) and $\mathcal{J}_{a^N, \tau_{Nt}}(\lambda^N_{N+1} > 0) = 0$ (see the discussion above (\ref{Eq.CaseIIIBetaEnsemble})), and consequently for $i \geq N+1$ we have $Z^N_i = \lambda^N_{i} - i \cdot \theta = -i \cdot \theta$. From Lemmas \ref{Lem.CaseIIWeakConv} and \ref{Lem.CaseIIFormulaMu} in Case II, and Lemmas \ref{Lem.CaseIIIWeakConv} and \ref{Lem.CaseIIIFormulaMu} in Case III, we have for $h_{\theta}(x) = h(x-\theta)$
\begin{equation*}
\frac{1}{N} \sum_{i = 1}^N h_{\theta}(Z_i^N/ N + \theta) \Rightarrow  \int_{\mathbb{R}} h_{\theta} (x) \mu(x) dx,
\end{equation*}
where $\mu = \mu^{II}$ in Case II and $\mu = \mu^{III}$ in Case III. Translating the arguments by $\theta$ gives
\begin{equation}\label{Eq.CaseIIPiece1}
\int_{\mathbb{R}} h(x) \nu^1_{N}(dx) = \frac{1}{N} \sum_{i = 1}^N h(Z_i^N/ N ) \Rightarrow  \int_{\mathbb{R}} h_{\theta} (x) \mu(x) dx = \int_{\mathbb{R}} h (y) \mu(y + \theta) dy.
\end{equation}
On the other hand, we have by a straightforward Riemann sum approximation argument that
\begin{equation}\label{Eq.CaseIIPiece2}
\int_{\mathbb{R}} h(x) \nu^2_{N}(dx) = \frac{1}{N} \sum_{i = N+1}^{\infty} h(-i \theta/N) \rightarrow  \int_{-\infty}^{-1} h (\theta x)  dx = \int_{-\infty}^{-\theta} h (y) \theta^{-1} dy.
\end{equation}
Combining (\ref{Eq.CaseIIPiece1}) and (\ref{Eq.CaseIIPiece2}) with the formulas for $\mu^{II}$ from (\ref{Eq.CaseIIFormulaMu}) and for $\mu^{III}$ from (\ref{Eq.CaseIIIFormulaMu}), we conclude (\ref{CaseIIRed}).

%
%
\subsection{Proof of Cases IV, V and VI}\label{Section5.4} We assume the conditions in either Case IV, Case V or Case VI of Definition \ref{DefScale}. As the proofs in these cases are similar, we treat them together. We seek to show that for any compactly supported continuous function $h$ on $\mathbb{R}$ we have as $N \rightarrow \infty$
\begin{equation}\label{CaseIVRed}
\int_{\mathbb{R}} h(x) \nu_{N}(dx) \Rightarrow  \int_{\mathbb{R}} h(x) f(x) dx,
\end{equation}
where $f = f^{IV}$ as in (\ref{MTE4}), $f = f^V$ as in (\ref{MTE5}) and $f = f^{VI}$ as in (\ref{MTE6}) in Cases IV, V and VI, respectively.

Let $D_2^{IV}$ be as in Lemma \ref{Lem.CaseIVLengthBound}, $D_2^{V}$ as in Lemma \ref{Lem.CaseVLengthBound} and $D_2^{VI}$ as in Lemma \ref{Lem.CaseVILengthBound} for some $A_0 > 0$. We fix an integer $d$, such that in Case IV $d \geq \max(D_2^{IV}, \alpha \theta^{-1} )$, where $\alpha$ is as in (\ref{Eq.CaseIVLowerBoundD}); in Case V $d \geq \max(D_2^{V}, \alpha \theta^{-1} )$, where $\alpha$ is as in (\ref{Eq.CaseVLowerBoundD}); in Case VI $d \geq \max(D_2^{VI}, 2 \sqrt{t_1t_2}, 2 \theta\sqrt{t_1}{t_2} )$. Let $(\ell^K_1, \dots, \ell^K_K)$ have law $\mathbb{P}^d_K$ as in (\ref{Eq.CaseIVBetaEnsembleV2}) in Case IV, (\ref{Eq.CaseVBetaEnsembleV2}) in Case V and (\ref{Eq.CaseVIBetaEnsembleV2}) in Case VI, and set $\mu_K = (1/K)\sum_{i = 1}^K \delta(\ell^K_i/K)$. From Lemmas \ref{Lem.CaseIVWeakConv} and \ref{Lem.CaseIVFormulaMu} in Case IV, Lemmas \ref{Lem.CaseVWeakConv} and \ref{Lem.CaseVFormulaMu} in Case V and Lemmas \ref{Lem.CaseVIWeakConv} and \ref{Lem.CaseVIFormulaMu} in Case VI, we have for any bounded continuous $g$ that
\begin{equation}\label{Eq.CaseIVFL1}
\int_{\mathbb{R}} g(x) \mu_{Nd}(dx) \Rightarrow \int_{\mathbb{R}} g(x) \mu_d(x)dx, 
\end{equation}
 as $N \rightarrow \infty$, where $\mu_d$ is as in (\ref{Eq.CaseIVFormulaMu}) in Case IV, as in (\ref{Eq.CaseVFormulaMu}) in Case V, and as in (\ref{Eq.CaseVIFormulaMu}) in Case IV. Let $\tilde{Z}_i^N = \tilde{\lambda}_i^N - i \cdot \theta$ for $i \geq 1$ with $\tilde{\lambda}^N = (\tilde{\lambda}^N_i: i \geq 1)$ being distributed according to $\mathcal{J}_{a_{\beta}^N, b_{\beta}^M}(\cdot|E^{\min(M,N)}_{Nd})$ in Case IV, according to $\mathcal{J}_{a_{\beta}^N, \tau_{Nt}}(\cdot|E^{N}_{Nd})$ in Case V, and according to $\mathcal{J}_{\tau_{Nt_1}, \tau_{Nt_2}}(\cdot|E^{Nd}_{Nd})$ in Case VI. Note that from the definition of $E^R_K$ in (\ref{Eq.DefEventTrunc}) we have $\tilde{Z}_i^N = -i \cdot \theta$ for $i \geq Nd+1$. Consequently, if we define $\tilde{\nu}_N$ as in (\ref{S13E1}) with $Z^N_i$ replaced with $\tilde{Z}^N_i$, we see that 
\begin{equation}\label{Eq.CaseIVSplit}
\tilde{\nu}_N = \tilde{\nu}^1_N + \tilde{\nu}^2_N, \mbox{ where } \tilde{\nu}_N^1 = \frac{1}{N} \sum_{i = 1}^{Nd} \delta (\tilde{Z}_i^N/N ), \mbox{ and } \tilde{\nu}^2_N = \frac{1}{N}\sum_{i =Nd+1}^{\infty} \delta \left( -i \cdot \theta/N \right).
\end{equation}
From (\ref{Eq.CaseIVSubseq}) in Case IV, (\ref{Eq.CaseVSubseq}) in Case V, and (\ref{Eq.CaseVISubseq}) in Case VI, we have that $(\tilde{Z}^N_1, \dots, \tilde{Z}^N_{Nd})$ has the same distribution as $(\ell^{Nd}_1 - Nd\theta, \dots, \ell^{Nd}_{Nd} - Nd \theta)$. Setting $g(x) = d h(d(x - \theta))$ in (\ref{Eq.CaseIVFL1}) we see that it is equivalent to 
$$
\int_{\mathbb{R}} h(x) \tilde{\nu}^1_{N}(dx) \Rightarrow \int_{\mathbb{R}} d h(d(x - \theta)) \mu_d(x)dx = \int_{\mathbb{R}} h(y)  \mu_d(d^{-1} y  + \theta)dy.
$$
As previously seen, we have by a Riemann sum approximation
\begin{equation}\label{Eq.CaseIVRiemann}
\int_{\mathbb{R}} h(x) \tilde{\nu}^2_{N}(dx) = \frac{1}{N} \sum_{i = Nd + 1}^{\infty} h(-i\theta/N) \rightarrow \int_{-\infty}^{-d}h(\theta x) dx = \int_{-\infty}^{-d\theta} h(y) \theta^{-1} dy.
\end{equation}
The last two displayed equations, and the formulas for $\mu_d$ and $f$ give
$$\int_{\mathbb{R}} h(x) \tilde{\nu}_{N}(dx) \Rightarrow  \int_{\mathbb{R}} h(x) f(x) dx.$$
The last equation and the total variation distance bound in (\ref{Eq.CaseIVTruncate}) in Case IV, the one in (\ref{Eq.CaseVTruncate}) in Case V, and the one in (\ref{Eq.CaseVITruncateLen}) in Case VI, imply (\ref{CaseIVRed}).

%
%
\section{Proof of Theorem \ref{thmMain2}}\label{Section6} The goal of this section is to prove Theorem \ref{thmMain2} and we assume throughout the same conditions as in the statement of the theorem. We establish each part of the theorem in a separate section.

%
%
\subsection{Global CLT}\label{Section6.1} We proceed to verify that the measures $\mathcal{J}_{a^N, b_\beta^M}$ from (\ref{Eq.CaseIIBetaEnsembleV2}) satisfy \cite[Assumptions 1-5]{DK22}. In the proof of Lemma \ref{Lem.CaseIIFormulaMu} we showed that $\mathcal{J}_{a^N, b^M_\beta}(\cdot)$ satisfy Assumptions \ref{Ass.Finite}, \ref{Ass.Loop1} and \ref{Ass.Loop2}, which in turn imply Assumptions \cite[Assumptions 1-4]{DK22} with $M_N = M(N)$, $\mathsf{M} = \cm$, $V_N,V$ as in (\ref{Eq.CaseIIBetaEnsembleV2}) and (\ref{Eq.DefVCaseII}), and $\Phi^{\pm}_N, \Phi^{\pm}$ as in (\ref{Eq.CaseIIPhis}). What remains is to show that \cite[Assumption 5]{DK22} holds, for which we seek to show that if 
\begin{equation}\label{Eq.Qmu}
Q_{\mu}(z) = \Phi^-(z) \cdot e^{-\theta G_{\mu}(z)} - \Phi^+(z) \cdot e^{\theta G_{\mu}(z)},
\end{equation}
then 
\begin{equation}\label{Eq.CLTRed1}
    Q_{\mu}(z) = (1+ab)\cdot \sqrt{(z-\alpha)(z-\beta)}.
\end{equation}
If (\ref{Eq.CLTRed1}) holds, then it would imply that \cite[Assumption 5]{DK22} is satisfied with $H(z) = 1+ ab$, once we also note that $0 \leq \alpha < \beta \leq \cm + \theta$ by construction. Consequently, the ``Global CLT'' part of the theorem would follow from \cite[Corollary 5.4]{DK22}. In the remainder of this section we verify (\ref{Eq.CLTRed1}).

From (\ref{S41E6}), (\ref{Eq.CaseIIPhis}) and (\ref{Eq.CaseIIR}) we have  
$$ (1-ab)z+ab\cm-\theta = z \cdot e^{-\theta G_{\mu}(z)} + ab(\cm+\theta - z) \cdot e^{\theta G_{\mu}(z)}.$$
The above sets up quadratic equations for $e^{\pm \theta G_{\mu}(z)}$. Namely, we obtain the equations 
$$X  [(1-ab)z+ab\cm-\theta] = z + ab(\cm + \theta - z) X^2, \hspace{2mm} Y [ (1-ab)z+ab\cm-\theta] = Y^2z + ab(\cm + \theta - z), $$
which have roots
$$X_{\pm} = \frac{(1-ab)z+ab\cm-\theta \pm (1 + ab) \sqrt{(z- \alpha)(z-\beta)}}{2ab (\cm + \theta -z)},$$
$$Y_{\pm} = \frac{(1-ab)z+ab\cm-\theta \pm (1 + ab) \sqrt{(z- \alpha)(z-\beta)}}{2z}.$$
Using that $\lim_{z \rightarrow \infty} e^{\pm \theta G_{\mu}(z)} = 1$, we see that 
\begin{equation}\label{Eq.ExponSt}
\begin{split}
&e^{\theta G_{\mu}(z)} = X_- = \frac{(1-ab)z+ab\cm-\theta - (1 + ab) \sqrt{(z- \alpha)(z-\beta)}}{2ab (\cm + \theta -z)}, \\
& e^{-\theta G_{\mu}(z)} = Y_+ =  \frac{(1-ab)z+ab\cm-\theta + (1 + ab) \sqrt{(z- \alpha)(z-\beta)}}{2z}.
\end{split}
\end{equation}
Substituting the latter into (\ref{Eq.Qmu}) gives (\ref{Eq.CLTRed1}).

%
%
\subsection{Global LDP}\label{Section6.2} Let us extend the functions $V_N$ in (\ref{Eq.CaseIIBetaEnsembleV2})  and $V$ in (\ref{Eq.DefVCaseII}) from $I_N = [0, M  N^{-1} + (N-1)  N^{-1} \theta] = [0, s_N]$ and $I = [0, \cm + \theta]$ continuously to $\mathbb{R}$ by setting
$$V_N(x) = \begin{cases} V_N(0) -x &\mbox{ if } x \leq 0, \\ V_N(s_N) + (x - s_N) &\mbox{ if } x \geq s_N, \end{cases} \hspace{2mm}  V(x) = \begin{cases}  -x &\mbox{ if } x \leq 0, \\  x - \cm - \theta &\mbox{ if } x \geq \cm + \theta. \end{cases}$$
We proceed to verify that the measures $\mathcal{J}_{a^N, b_\beta^M}$ from (\ref{Eq.CaseIIBetaEnsembleV2}) satisfy \cite[Assumptions 1 and 2]{DZ23}.

We observe that \cite[Assumption 1]{DZ23} holds for $a_N = a = 0$, $b_N = M(N)$ and $b = \cm$. For \cite[Assumption 2]{DZ23} we note that due to the linear growth of both $V_N$ and $V$ near infinity, they satisfy \cite[(1.9)]{DZ23} and \cite[(1.13)]{DZ23}, respectively, and are both continuous on $\mathbb{R}$ by construction. Lastly, by combining (\ref{Eq.S33R1}) with the triangle inequality, we conclude for some $A_2 > 0$
$$\sup_{x \in \mathbb{R}} |V_N(x) - V(x)| \leq A_2 N^{-1} \log (N + 1) + |s_N - \cm - \theta| = O\left(N^{-1} \log(N+1) \right),$$
which shows that \cite[Assumption 2(a)]{DZ23} holds. Since  \cite[Assumptions 1 and 2]{DZ23} hold, we conclude from \cite[Theorem 1.3]{DZ23} that the laws of $\mu_N$ satisfy an LDP with speed $N^2$ and good rate function
\begin{equation}\label{Eq.GLDPRed1}
I_V^{\theta} (\mu) = \begin{cases} \theta (E_V(\mu) - F_V^{\theta}) &\mbox{ for } \mu \in \mathcal{M}_{\theta}([0, \cm + \theta]), \\ \infty &\mbox{ for } \mu \in \mathcal{M}(\mathbb{R}) \setminus \mathcal{M}_{\theta}([0, \cm + \theta]), \end{cases}
\end{equation}
where $F_V^{\theta} = \inf_{\mu \in \mathcal{M}_{\theta}([0, \cm + \theta])} E_V(\mu)$. From \cite[Theorem 1.1]{DZ23} we know that $F_V^{\theta}$ is finite and $F_V^{\theta} = E_V(\mu_{\mathsf{eq}}^{\theta})$ for a unique $\mu_{\mathsf{eq}}^{\theta} \in \mathcal{M}_{\theta}([0, \cm + \theta])$. The last few statements imply that $\mu_N$ converge weakly in probability to $\mu_{\mathsf{eq}}^{\theta}$; however, from Lemmas \ref{Lem.CaseIIWeakConv} and \ref{Lem.CaseIIFormulaMu} we know that $\mu_N$ converge weakly in probability to $\mu^{II}$. Since weak limits are unique, we conclude $\mu_{\mathsf{eq}}^{\theta} = \mu^{II}$ and hence $I_V^{\theta}$ from (\ref{Eq.GLDPRed1}) agrees with $I^{\mathsf{Global}}$ in (\ref{Eq.RateGLDP}). This completes the proof of the ``Global LDP'' part of the theorem.

%
%
\subsection{Edge LDP}\label{Section6.3} In Section \ref{Section3.3} we showed that $\mathcal{J}_{a^N, b_\beta^M}$ from (\ref{Eq.CaseIIBetaEnsembleV2}) satisfy  Assumption \ref{Ass.Finite}, which in turn imply Assumptions \cite[Assumptions 2.1 and 2.2]{DD22} with $M_N = M(N)$, $a_0 = A_0 = \cm$, $q_N = 1/N$, $V_N,V$ as in (\ref{Eq.CaseIIBetaEnsembleV2}) and (\ref{Eq.DefVCaseII}), and $p_N = N^{-1} \log(N + 1)$. The latter shows that $\mathcal{J}_{a^N, b_\beta^M}$ satisfy the conditions of \cite[Definition 2.4]{DD22}, and since $\lim_N N^{3/4} q_N  = \lim_N N^{3/4} p_N = 0$, we obtain from \cite[Proposition 5.2]{DD22} that:
\begin{enumerate}
\item[(a)] For any $t \in [\theta, \cm + \theta)$, we have 
\begin{equation}\label{Eq.ELDP1}
\limsup_{N \rightarrow \infty} \frac{1}{N} \log \mathcal{J}_{a^N, b_\beta^M} (\ell_1 \geq tN) \leq - J_V^{\theta, \cm + \theta}(t),
\end{equation} 
\item[(b)] For any $t \in [\theta, b_{\cm})$, we have 
\begin{equation}\label{Eq.ELDP2}
\liminf_{N \rightarrow \infty} \frac{1}{N} \log \mathcal{J}_{a^N, b_\beta^M} (\ell_1 \geq tN) \geq - J_V^{\theta, \cm + \theta}(t),
\end{equation}
\item[(c)] In addition, if $J_V^{\theta, \cm + \theta}(t) > 0$ for $t \in (b_{\cm}, \cm + \theta)$, then (\ref{Eq.ELDP2}) holds for any $t \in [\theta, \cm + \theta)$.
\end{enumerate}
In the above equations we have that $b_{\cm}$ is the rightmost point of the support of $\mu^{II}$ from (\ref{Eq.muII}) and 
\begin{equation}\label{Eq.ELDP3}
J_V^{\theta, \cm + \theta} (x) = \begin{cases} 0 &\mbox{ if } x \in [0, b_{\cm}), \\ \inf_{y \in [x, \cm + \theta]} G_V^{\theta, \cm + \theta}(y)  - G_V^{\theta, \cm + \theta}(b_{\cm}) & \mbox{ if } x \in [b_{\cm}, \cm + \theta], \end{cases}
\end{equation}
where 
\begin{equation}\label{Eq.ELDP4}
G_V^{\theta, \cm + \theta}(x) = -2 \theta \int_0^{\cm + \theta} \log |x - t| \mu^{II}(t) dt + V(x).
\end{equation}

We now observe directly from the definition of $\mu^{II}$ in (\ref{Eq.muII}) that if $\cm \in (0, ab \theta]$, then $b_{\cm} = \cm + \theta$. Indeed, if $\cm \in (0, ab \theta)$, we have that 
$$\lim_{x \rightarrow (\cm + \theta)-} [ (1-ab)x + ab\cm-\theta ] = \cm -ab\theta < 0,$$
while the denominator in (\ref{Eq.muII}) vanishes, making the argument of $\arccos$ converge to $-\infty$, and hence $\lim_{x \rightarrow (\cm + \theta)-} \mu^{II}(x) = \theta^{-1}$. If instead $\cm = ab \theta$, we have as $x \rightarrow (\cm + \theta)-$
$$\frac{(1-ab)x + ab\cm-\theta}{2 \sqrt{ab x (\cm +\theta - x) }} \sim \frac{(ab-1)\sqrt{\cm + \theta - x} }{2 \sqrt{ab(ab+1)\theta}},$$
and hence  $\lim_{x \rightarrow (\cm + \theta)-} \mu^{II}(x) = \theta^{-1}/2$. As $b_{\cm} = \cm + \theta$, we conclude (\ref{Eq.ELDPMain}) from (\ref{Eq.ELDP1}), (\ref{Eq.ELDP2}) and (\ref{Eq.ELDP3}) in the case when $\cm \in (0, ab \theta]$.\\

In the remainder of the proof we assume that $\cm > ab\theta$. We first proceed to find $b_{\cm}$ in this case. Let $\alpha, \beta$ be as in (\ref{Eq.AlphaBeta}). Then, we see that for $x \in [\beta, \cm + \theta)$
$$(1-ab)x + ab\cm-\theta \geq (1-ab)\beta + ab\cm-\theta  = \frac{2ab\cm + 2\sqrt{ab\cm \theta} - 2ab\theta - 2ab\sqrt{ab \cm \theta} }{1 + ab} > 0,$$
where in the last inequality we used that $\cm > ab\theta$. On the other hand, we have 
$$[(1-ab)x + ab\cm-\theta ]^2 - 4 abx(\cm + \theta -x ) > 0 \iff (1 + ab)(x-\alpha)(x-\beta) > 0, $$
which together with the last equation implies for $x \in (\beta, \cm + \theta)$
\begin{equation}\label{Eq.ELDPRat}
\frac{(1-ab)x + ab\cm-\theta}{2 \sqrt{abx(\cm + \theta -x )}} > 1.
\end{equation} 
We conclude that the argument of $\arccos$ in (\ref{Eq.muII}) is bigger than $1$ for $x \in (\beta, \cm + \theta)$ and so 
\begin{equation}\label{Eq.Void}
\mu^{II}(x) = 0 \mbox{ for } x \geq \beta.
\end{equation} 
Fix $\varepsilon > 0$, such that $\varepsilon < \beta$ and $\varepsilon < \cm + \theta - \beta$. We note that for $x \in [\varepsilon, \cm + \theta - \varepsilon]$ we have 
$$\frac{(1-ab)x + ab\cm-\theta}{2 \sqrt{ab x (\cm +\theta - x) }} = 1 + \frac{(1+ab)(\beta - \alpha) (x - \beta) }{8 ab \beta ( \cm + \theta - \beta)} + O(|\beta - x|^2),$$
and that $\arccos(z) = \sqrt{2 - 2z} +O(|z-1|^{3/2})$ near $1-$. Consequently, we conclude that for $x \in [\varepsilon, \beta]$
\begin{equation}\label{Eq.MuIISqRoot}
\mu^{II}(x) = \frac{s_B}{\pi} \cdot \sqrt{\beta - x} + O(|\beta - x|^{3/2}), \mbox{ where } s_B = \frac{\sqrt{(1+ab)(\beta - \alpha)} }{2 \theta \sqrt{ ab \beta ( \cm + \theta - \beta)}}.
\end{equation}
The constant in the big $O$ notation depends on $a,b, \cm , \theta$ and $\varepsilon$. Equations (\ref{Eq.Void}) and (\ref{Eq.MuIISqRoot}) show that in the case when $\cm > ab\theta$, we have that $b_{\cm} = \beta$.\\

Our next task is to find $\frac{d}{dx}G_V^{\theta, \cm + \theta}(x)$. From (\ref{Eq.ELDP4}) we see that 
\begin{equation*}
\frac{d}{dx} G_V^{\theta, \cm + \theta}(x) = -2 \theta G_{\mu^{II}}(x) + \log(x) - \log(\cm + \theta - x) - \log(ab),
\end{equation*}
where we used the formula for the derivative of $V$ from (\ref{Eq.CaseIIVPrime}) and the definition of the Stieltjes transform from (\ref{S41E5}). Combining the latter with (\ref{Eq.ExponSt}), we conclude for $x \in (\beta, \cm + \theta)$ 
\begin{equation*}
\frac{d}{dx} G_V^{\theta, \cm + \theta}(x) = 2 \log \left( \frac{(1-ab)x+ab\cm-\theta + (1 + ab) \sqrt{(x- \alpha)(x-\beta)}}{2 \sqrt{abx (m + \theta -x)}  } \right) > 0,
\end{equation*}
where in the last inequality we used (\ref{Eq.ELDPRat}). The last equation and (\ref{Eq.ELDP3}) together imply that for $x \in [\beta, \cm + \theta]$
\begin{equation}\label{Eq.JMain}
J_V^{\theta, \cm + \theta} (x)  = \int_{\beta}^{x} 2 \log \left( \frac{(1-ab)y+ab\cm-\theta + (1 + ab) \sqrt{(y- \alpha)(y-\beta)}}{2 \sqrt{aby (m + \theta -y)}  } \right)  dy,
\end{equation}
and the latter is positive for $x \in (\beta, \cm + \theta)$. In particular, condition (c) below (\ref{Eq.ELDP2}) is satisfied. Since we showed that $b_{\cm} = \beta$, we conclude (\ref{Eq.ELDPMain}) from (\ref{Eq.ELDP1}), (\ref{Eq.ELDP2}), (\ref{Eq.ELDP3}) and (\ref{Eq.JMain}) in the case when $\cm > ab\theta$. This completes the proof of the ``Edge LDP'' part of the theorem.

%
%
\subsection{Edge universality}\label{Section6.4} We proceed to verify that the measures $\mathcal{J}_{a^N, b_\beta^M}$ from (\ref{Eq.CaseIIBetaEnsembleV2}) satisfy Assumptions 1.2, 1.3, 1.5 and 1.7 in \cite{GH17}. In the proof of Lemma \ref{Lem.CaseIIFormulaMu} we showed that $\mathcal{J}_{a^N, b^M_\beta}(\cdot)$ satisfy Assumptions \ref{Ass.Finite}, \ref{Ass.Loop1} and \ref{Ass.Loop2}, which in turn imply \cite[Assumption 1.2]{GH17} with 
$$a(N) = -1, \hspace{2mm} b(N) = M(N) + 1 + (N-1)\theta, \hspace{2mm} \hat{a} = 0, \hspace{2mm} \hat{b} = \cm + \theta, \hspace{2mm} \mbox{$V_N$ and $V$ as in (\ref{Eq.CaseIIBetaEnsembleV2}) and (\ref{Eq.DefVCaseII})};$$
and also imply \cite[Assumption 1.3]{GH17} with $\psi^{\pm}_N = \Phi^{\pm}_N$ and $\phi^{\pm} = \Phi^{\pm}$. In addition, in Section \ref{Section6.1} above we showed that \cite[Assumption 1.5]{GH17} holds with $H(z) = 1+ab$, $A = \alpha$, and $B= \beta$, where we recall that $\alpha, \beta$ are as in (\ref{Eq.AlphaBeta}). We mention that the conditions $\cm > ab\theta$ and $ab\cm \neq \theta$ ensure $\alpha > 0$ and $\beta < \cm + \theta$, which are required in \cite[Assumption 1.5]{GH17}. What remains is to show that \cite[Assumption 1.7]{GH17} holds, for which we seek to show that in a neighborhood of $[\alpha, \beta]$
\begin{equation}\label{Eq.DerApprox}
V_N'(x) = V'(x) + O(N^{-1/3}).
\end{equation}

We recall from \cite[Theorem 1.2.5]{AAR} that 
$$\frac{\Gamma'(x)}{\Gamma(x)} = -\gamma + \sum_{n = 0}^{\infty} \left[\frac{1}{n+1} - \frac{1}{n+x} \right],$$
where $\gamma= 0.577215...$ is the Euler-Mascheroni constant. Combining the last equation with the formula for $V_N$ in (\ref{Eq.CaseIIBetaEnsembleV2}) and $V'$ in (\ref{Eq.CaseIIVPrime}) we conclude
\begin{equation*}
V_N'(x) - V'(x) =   \sum_{n = 0}^{\infty} \left[\frac{1}{n + M + N \theta - \theta -Nx +1} - \frac{1}{n+Nx + 1} \right]- \log(x) + \log(\cm + \theta - x). 
\end{equation*}
We also observe that 
$$\log(x) - \log(\cm + \theta -x) = \sum_{n = 0}^{\infty} \int_{n/N}^{(n+1)/N} \left[ \frac{1}{y + \cm + \theta - x} - \frac{1}{y + x}  \right] dy.$$
Since $[\alpha, \beta] \subset (0, \cm + \theta)$, we can find $\varepsilon > 0$ such that $[\alpha - 2 \varepsilon, \beta + 2 \varepsilon] \subset (0, \cm + \theta)$. If $x \in [\alpha - \varepsilon, \beta + \varepsilon]$ and $N$ is sufficiently large so that $N^{-1}(\mathsf{D} + \theta + 1) < \varepsilon/2$, we have for $y \in [n/N, (n+1)/N]$ and $\Delta_N = M + N\theta - \theta +1$
\begin{equation*}
\begin{split}
&\left| \frac{1}{N^{-1}n + N^{-1} \Delta_N -x } - \frac{1}{y + \cm + \theta - x}\right| \leq \frac{1}{N} \cdot \frac{\mathsf{D} + \theta + 2}{(\cm + \theta - N^{-1}(\mathsf{D} + \theta + 1) -x + n/N)^2},\\
&\left| \frac{1}{N^{-1}(n+1) + x } - \frac{1}{y + x}\right| \leq \frac{1}{N} \cdot \frac{1}{(x +n/N)^2} ,
\end{split}
\end{equation*}
where in deriving the inequalities in each line we used the mean value theorem and that $|N^{-1} \Delta_N - \cm - \theta| \leq N^{-1}(\mathsf{D} + \theta + 1) < \varepsilon/2$. Combining the last three displayed equations and using $N^{-1}(\mathsf{D} + \theta + 1) < \varepsilon/2$ we conclude for $x \in [\alpha - \varepsilon, \beta + \varepsilon]$
\begin{equation*}
\left|V_N'(x) - V'(x) \right| \leq \frac{1}{N^2}\sum_{n = 0}^{\infty} \left[\frac{1}{(\varepsilon +n/N)^2} + \frac{\mathsf{D} + \theta + 2}{(\varepsilon/2 + n/N)^2} \right] = O(N^{-1}),  \end{equation*}
which clearly implies (\ref{Eq.DerApprox}).\\

Our work so far verifies the assumptions required to apply \cite[Theorem 1.11]{GH17}, but we still need to verify a couple of statements within the theorem itself. From (\ref{Eq.Void}) and (\ref{Eq.MuIISqRoot}) we know that $\beta$ is the rightmost endpoint of the support of $\mu^{II}$, the interval $[\beta, \cm + \theta]$ is a void region for the measure, and $\mu^{II}$ has a square root behavior near $\beta$. The last few observations now confirm the conditions within \cite[Theorem 1.11]{GH17} itself, which yields the ``Edge universality'' part of the theorem.

\bibliographystyle{amsplain} 
\bibliography{ref}

\end{document}